\documentclass[a4paper,reqno,oneside]{amsart}

\usepackage[utf8]{inputenc}
\usepackage[usenames,dvipsnames]{xcolor}
\usepackage{palatino} 
\usepackage[abs]{overpic}

\usepackage{amssymb,amsmath,amsthm}
\usepackage{amsrefs}
\usepackage{amsthm}

\usepackage[all]{xy} 
\usepackage{graphicx} %
\usepackage{pinlabel}
\usepackage{tikz}

\usepackage{enumerate}

\newtheorem{innercustomgeneric}{\customgenericname}
\providecommand{\customgenericname}{}
\newcommand{\newcustomtheorem}[2]{%
  \newenvironment{#1}[1]
  {%
   \renewcommand\customgenericname{#2}%
   \renewcommand\theinnercustomgeneric{##1}%
   \innercustomgeneric
  }
  {\endinnercustomgeneric}
}

\newcustomtheorem{customthm}{Theorem}
\newcustomtheorem{customlemma}{Lemma}

\makeatletter
\DeclareRobustCommand\widecheck[1]{{\mathpalette\@widecheck{#1}}}
\def\@widecheck#1#2{%
    \setbox\z@\hbox{\m@th$#1#2$}%
    \setbox\tw@\hbox{\m@th$#1%
       \widehat{%
          \vrule\@width\z@\@height\ht\z@
          \vrule\@height\z@\@width\wd\z@}$}%
    \dp\tw@-\ht\z@
    \@tempdima\ht\z@ \advance\@tempdima2\ht\tw@ \divide\@tempdima\thr@@
    \setbox\tw@\hbox{%
       \raise\@tempdima\hbox{\scalebox{1}[-1]{\lower\@tempdima\box
\tw@}}}%
    {\ooalign{\box\tw@ \cr \box\z@}}}
\makeatother

\usepackage[marginpar=3cm, marginparwidth=2.5cm]{geometry}

\newtheorem{thm}{Theorem}[section]
\newtheorem{prop}[thm]{Proposition}
\newtheorem{lemma}[thm]{Lemma}
\newtheorem{cor}[thm]{Corollary}

\theoremstyle{definition}
\newtheorem{defn}[thm]{Definition}
\newtheorem{quest}[thm]{Question}
\newtheorem{Warning}[thm]{Warning}

\theoremstyle{remark}
\newtheorem{ex}[thm]{Example}
\newtheorem{rmk}[thm]{Remark}

\newcommand{\Z}{\mathbb{Z}}

\newcommand{\R}{\mathbb{R}}
\newcommand{\F}{\mathbb{F}}
\newcommand{\C}{\mathbb{C}}

\newcommand{\CP}{\mathbb{CP}^2}
\newcommand{\CPn}{\mathbb{CP}^n}
\newcommand{\CPbar}{\overline{\mathbb{CP}}\vphantom{\mathbb{CP}}^2}
\newcommand{\CPI}{\mathbb{CP}^1}
\newcommand{\Yxi}{(Y,\xi)}
\newcommand{\Xom}{(X,\omega)}
\newcommand{\Sst}{(S^3,\xi_{\rm std})}
\newcommand{\ghat}{\widehat{g}}
\newcommand{\shat}{\widehat{s}}
\newcommand{\Ghat}{\widehat{G}}
\newcommand{\inv}{^{-1}}
\renewcommand{\epsilon}{\varepsilon}

\renewcommand{\emptyset}{\varnothing}

\DeclareMathOperator{\slk}{sl}
\DeclareMathOperator{\tb}{tb}
\DeclareMathOperator{\rot}{rot}
\DeclareMathOperator{\Int}{Int}

\usepackage{hyperref}

\numberwithin{equation}{section}
\DeclareGraphicsExtensions{.pdf,.png,.jpg}
\graphicspath{{img/}}

\definecolor{John}{RGB}{0,82,0}
\definecolor{Marco}{RGB}{0,0,90}

\title{Symplectic hats} 
\author{John B. Etnyre}
\address{School of Mathematics, Georgia Institute of Technology, Atlanta, GA, U.S.A.}
\email{etnyre@math.gatech.edu}
\author{Marco Golla}
\address{CNRS; Laboratoire de Math\'ematiques Jean Leray, Universit\'e de Nantes, Nantes, France}
\email{marco.golla@univ-nantes.fr}
\date{}

\begin{document}

\begin{abstract}
We study relative symplectic cobordisms between contact submanifolds, and in particular relative symplectic cobordisms to the empty set, that we call hats.
While we make some observations in higher dimensions, we focus on the case of transverse knots in the standard 3--sphere, and hats in blow-ups of the (punctured) complex projective planes.
We apply the construction to give constraints on the algebraic topology of fillings of double covers of the 3--sphere branched over certain transverse quasipositive knots.
\end{abstract}

\maketitle

\section{Introduction}

There has been a great deal of study of cobordism and concordance of smooth knots in dimension 3, leading to a beautiful and rich field in low-dimensional topology.
There are two ways of formulating contact analogues of these objects. Both start with a symplectic cobordism, that is a symplectic manifold $(X,\omega)$ whose boundary consists of a concave part $(M_-,\xi_-)$ and a convex part $(M_+,\xi_+)$. Given a Legendrian submanifold $L_\pm$ in the contact manifold $M_\pm$ one can look for Lagrangian submanifold in $X$ with boundary $-L_-\cup L_+$.
Such Lagrangian cobordisms have been studied quite closely~\cites{Chantraine10,CornwellNgSivek16, EkholmHondaKalman16}. However, the corresponding question about symplectic cobordisms has seen comparatively little attention. More specifically, given contact submanifolds $(C_\pm, \widetilde{\xi}_\pm)$ of $(M_\pm, \xi_\pm)$, we say they are {\em relatively symplectically cobordant} if there is a properly embedded symplectic submanifold $(\Sigma, \omega|_\Sigma)$ of $(X,\omega)$ that is transverse to $\partial X$ and a symplectic cobordism from $(C_-,\widetilde{\xi}_-)$ to  $(C_-,\widetilde{\xi}_-)$.

We note that if we consider relative symplectic cobordisms in codimension larger than $2$ then there is an $h$-principle~\cite[Theorem~12.1.1]{EliashbergMishachev02}. In particular, if there is a smooth cobordism between the contact manifolds that is formally symplectic then the cobordism can be isotoped to be a relative symplectic cobordism. Thus we will restrict to the codimension-$2$ setting in this paper.

The only situation when the relative symplectic cobordism question has been extensively studied is when $X=B^4$ with its standard symplectic structure (so $M_-=\emptyset$ and $M_+$ is the standard contact $S^3$). In this context $C_+$ will be a transverse link and we are asking when $C_+$ bounds a symplectic surface in $B^4$. 
Thanks to work of Rudolph~\cite{Rudolph83} and Boileau and Orevkov~\cite{BoileauOrevkov}, we have a complete characterization of links that bound such surfaces: they are the closures of quasipositive braids. Moreover, the answer is the same in the complex and in the symplectic category.
Quasipositive knots are now a class that is very familiar to low-dimensional topologists, and some results about their fillings have even been partially generalized past $B^4$~\cite{Hayden17Pre}.

In this paper we will study the general problem of relative symplectic cobordism in all dimensions, but we will particularly focus on the $3$--dimensional setting. We will also focus much of our attention on the situation where $C_+=\emptyset$. When $M_+$ is empty we say $X$ is a {\em symplectic cap} for $M_-$ and if in addition $C_+$ is empty we say $\Sigma$ is a {\em symplectic hat} for $C_-$. 

If one does not restrict the topology of $X$ then it is not hard to show that $(C_-,\widetilde{\xi}_-)$ in $(M_-,\xi_-)$ has a symplectic hat~\cite{GadgilKulkarni12}, in fact one can even control the topology of $\Sigma$.

\begin{thm}\label{t:diskhats2}
Every transverse link $L$ in a contact $3$--manifold $\Yxi$ has a hat that is a disjoint union of disks in \emph{some} cap $\Xom$ for $\Yxi$.
\end{thm}

It is more difficult to find symplectic hats when the topology of $X$ is fixed. Below we will study the situation when $X$ is assumed to be simple; we will show that there is some rich structure to the problem and that hats can be used to build symplectic caps for contact manifolds and restrict the topology of symplectic fillings of certain contact manifolds. But before moving on to this, we end this discussion with the fundamental question:
\begin{quest}
Let $(X,\omega)$ be a symplectic cap for $(M,\xi)$. Does a contact submanifold $(C,\xi')$ of $(M,\xi)$ bound a symplectic hat in $(X,\omega)$ if and only if $C$ is null-homologous in $X$?
\end{quest}

While it seems unlikely that the answer can be YES in general, below we provide some mild evidence that it might indeed be YES. In particular, below we will see that there are many fewer restrictions on symplectic hats than on relative symplectic fillings and so one might hope the answer is YES. Even if the answer is NO, can one formulate conditions that will guarantee the existence of a hat?

\subsection{Projective hats}
The simplest symplectic cap for $\Sst$ is the \emph{projective cap} $\CP \setminus \Int(B^4)$, where $B^4$ is a Darboux ball in $\CP$. We call a symplectic hat in the projective cap a \emph{projective hat}. We begin by noticing the following. 

\begin{thm}\label{t:CPcap1}
Every transverse link $T$ in $\Sst$ has a projective hat.
\end{thm}

This result is in stark contrast with the results of Rudolph and Boileau--Orevkov, in that it poses no restriction on the link.
However, in some way it parallels the analogous result in the absolute case: while there are strong restrictions on contact manifolds in order for them to admit a symplectic filling 
(e.g.\ overtwistedness, non-vanishing of contact invariants), \emph{every} contact manifold has a symplectic cap~\cite{EtnyreHonda02a}.

One way of thinking about projective hats is in terms of singularities of curves (see, for instance,~\cite{GStarkston} for related questions and definitions); in fact, by coning over $(S^3,L)$, we can think of Theorem~\ref{t:CPcap1} as saying that every transverse knot is the link of the (unique) singularity of a singular symplectic surface in $\CP$.
From this perspective, quasipositive transverse links are links for which the surface admits a symplecting smoothing, while algebraic links are links for which the surface admits both a smoothing and a resolution in terms of blow-ups.

Since Theorem~\ref{t:CPcap1} provides us with an existence result, we can ask questions about complexity. 
We define the \emph{hat genus} of $T$ to be the smallest genus $\ghat(T)$ of a projective hat. This is one of the possible measures of complexity of hats. We prove various properties of the hat genus and compute it for some families of transverse knots. One of the more general theorems along these lines is the following.

\begin{thm}\label{negbraidnew}
Suppose $T$ is a transverse knot in $\Sst$ and there is a transverse regular homotopy to an unknot with only positive crossing changes. Then the hat genus is
\[
\ghat(T)=-\left(\frac{\slk(T)+1}{2}\right),
\] 
where $\slk(T)$ is the self-linking number of $T$. 
\end{thm}

This allows us to show, for example, that for $q>p\geq 2$, any transverse representative $T$ of the $(p,-q)$-torus knot $T_{p,-q}$ satisfies 
\[
\ghat(T)=-\left(\frac{\slk(T)+1}{2}\right).
\] 
In particular, the maximal self-linking number representative $T'$ (which has $\slk(T')=-pq+q-p$) has 
\[
\ghat(T')=\frac{(q-1)(p+1)}{2}.
\]

\begin{rmk}
It is interesting to note that in~\cite{Nouh09} it was shown that the smooth projective genus of $T_{2,-3}$ is $0$ where as we have computed the symplectic projective genus to be at least $3$ for any transverse representative. Thus we see quite a difference between the smooth and symplectic hat genus of a knot. 
\end{rmk}

Another sample computation is that for the $n$--twist knot with maximal self-linking number $T_n$~\cite{EtnyreNgVertesi13}, the hat genus is
\[
\ghat(T_n)=\begin{cases}
1 & \text{$n\leq -3$ and odd}\\
\frac{n+3}{2}& \text{$n\geq 1$ and odd}\\
\frac n2& \text{$n$ positive and even.}\\
\end{cases}
\]
\begin{quest}\label{non-simplicity}
The hat genus for $T_n$ with $n$ even and negative is not known. Recall there are several maximal self-linking number representatives of such twists knots. Is the hat genus of each representative the same? More generally, there are many knot types that are known to have distinct transverse representatives with the same self-linking number \cite{BirmanMenasco06II, ChongchitmateNg13, EtnyreHonda05, EtnyreLafountainTosun12, NgOzsvathThurston08}. Can the hat genus distinguish any of these?
\end{quest}

We note that the projective hat genus is not always a simple function of the self-linking number. Specifically, in Proposition~\ref{smalltorusknots} we show that if $T_{2,2k+1}$ is the maximal self-linking number torus knot, then 
\begin{table}[htp]
\begin{center}
\begin{tabular}{c| c c c c c c c c c c c}
$k$ 	& 1 	& 2 	& 3	& 4	& 5	& 6	& 7	& 8	& 9 & 10  & 11\\
\hline
$\ghat(T_{2,2k+1})$ 	&0	&1	&0	&2	&1	&0	&3	&2	&  1 & 5 & 4
\end{tabular}
\end{center}
\end{table}

\begin{quest}
Can one compute $\ghat(T_{2,2k+1})$?
\end{quest}

Note that the question has a clear counterpart in complex geometry, asking what is the minimal genus (or minimal degree) of an algebraic curve in $\CP$ with a singularity of type $T_{2,2k+1}$ (also referred to as an $A_{2k}$--singularity). It is also related to a question in the theory of deformations of singularities, asking what is the minimal $p$ such that a singularity of type $T_{p,p+1}$ deforms to $T_{2,2k+1}$. Both questions are open, and there is a relatively large gap between the available lower and upper bounds. See, for instance,~\cite{Feller-optimal, GSZ, Orevkov-An}.

We end the discussion of projective hats with the following question.
\begin{quest}\label{mainquest}
If $T$ is a slice, quasipositive transverse knot in $\Sst$ that has $\ghat(T)=0$, is $T$ the maximal self-linking unknot?
\end{quest}

While we do not know how to answer Question~\ref{mainquest}, we sketch an approach that seems promising at the end of Section~\ref{furtherex}.

We recall that there is a well-known and well-studied analogous question for Lagrangian concordance. Namely, if there is a Lagrangian concordance from $L$ to and from the maximal Thurston--Bennequin Legendrian unknot $U$, is $L$ isotopic to $U$?
A positive answer to Question~\ref{mainquest} would imply a positive answer to the Lagrangian question as well, via symplectic push-off. 

\subsection{Hats in other manifolds}
Above we saw that not all transverse knots bound a genus-$0$ surface in the projective hat of $\Sst$; however, we do have the following.

\begin{thm}\label{t:diskhats1}
Every transverse knot $T$ in $\Sst$ has a symplectic hat of genus $0$ in a \emph{blow-up} of the projective cap.
\end{thm}

Let $(X_0,\omega_0) = \CP\setminus B^4$, where $B^4$ is embedded as a Darboux ball with convex boundary. We will write $X_n$ to denote an $n$--fold symplectic blow-up of $X_0$, so that $X_n$ is diffeomorphic to $(\CP\#n\CPbar)\setminus B^4$. These are all caps for $\Sst$. We define the  {\em $n^\text{th}$ rational hat genus} $\ghat_n(T)$ of a transverse knot $T$ to be the smallest genus of a hat for $K$ in $X_n$

The proof of Theorem~\ref{t:diskhats1} shows that 
for any transverse knot $T$, the sequence $\Ghat(T) = \{\ghat_n(T)\}_n$ is non-increasing and eventually 0.
We define the \emph{hat slicing number} of $T$ to be 
\[
\shat(T) = \min\{n \mid \ghat_n(T) = 0\}.
\]
It takes some work to find examples where the hat slicing number is larger than $1$. 
\begin{prop}\label{goodslicenumberex}
Let $K_p$ be the unique transverse representative of $T_{p.p+1}\# T_{2,3}$ with maximal self-linking number which is $p^2-p+1 = 2g(T_{p,p+1}\#T_{2,3}) - 1$.  We have
\begin{align*}
\widehat{s}(K_p) &= p-1, \text{ for } p\le 7,\\
\Ghat(K_p) &= (p-1,p-2,\dots,2,1,0,\dots), \text{ for } p \le 4.
\end{align*} 
\end{prop}

\begin{quest}
Let $K_p$ be the transverse representative of $T_{p,p+1}\#T_{2,3}$ with $\slk(K_p) = p^2-p+1$. Is $\shat(K_p) = p-1$? Is $\Ghat(K_p) = (p-1,p-2,\dots,2,1,0,\dots)$?
\end{quest}

It is easy to see, modifying the proof of the above proposition, that 
$\shat(K_p) \le p-1$ and that $\ghat(K_p) = p-1$. In particular we also know $\ghat_n(K_p)\leq p-1-n$, for $n\leq p-1$. 

We also formulate the following question.
\begin{quest}	
Is $\ghat_{k+1}(K) < \ghat_{k}(K)$ if $k < \shat(K)$?
In other words, is $\Ghat(K)$ a strictly decreasing sequence until it hits $0$?
\end{quest}

In Section~\ref{Hhats} we also investigate hats in Hirzebruch caps for $\Sst$. The Hirzebruch caps are $H_0$, which is the standard symplectic $S^2\times S^2$ minus a Darboux ball, and $H_1=X_1$.

\subsection{Higher-dimensional hats}
We also consider higher-dimensional projective caps for links of isolated complete intersection singularities. Recall that a complete intersection in $\CPn$ is a complex $d$--dimensional subvariety defined by $n-d$ equations.
An \emph{isolated complex intersection singularity} is an isolated singularity of a complete intersection. Its link $\Sigma$ is a contact submanifold of dimension $2d-1$ in $(S^{2n-1}, \xi_{\rm std})$.
We view the ambient manifold as the concave boundary of $\CPn \setminus B^{2n}$, which we still call a projective cap.

\begin{prop}\label{p:Brieskornhats}
Let $Y \subset (S^{2n-1}, \xi_{\rm std})$ be the link of an isolated complete intersection singularity. Then $\Sigma$ has a hat in the projective cap of $(S^{2n-1}, \xi_{\rm std})$.
\end{prop}

Since the analogue proposition for torus knots (which are hypersurface singularities in $\C^2$, and in particular they are complete intersection singularities) is one of the main lemmas in our proof of Theorem~\ref{t:CPcap1}, we hope that the statement might be one of the ingredients in the proof of the higher-dimensional and codimension generalization of Theorem~\ref{t:CPcap1}.

In fact, we make an effort in setting up all definitions and technical statements in the general case, rather than restricting to the case of knots in $3$--manifolds. What is missing in the proof of the generalization of Theorem~\ref{t:CPcap1} to arbitrary dimension is a cofinality statement; below we will prove that the set of torus links is \emph{cofinal} in the set of transverse links, with respect to the partial ordering given by relative symplectic cobordisms. Untangling the definition, this means that for every transverse link $L$ in $\Sst$ there exist a torus link $T$ in $\Sst$ and a symplectic cobordism from $L$ to $T$.

Proving an analogue statement for links of isolated complete intersection singularities, together with the proposition above, would yield the existence of projective hats of contact submanifolds of $(S^{2n-1}, \xi_{\rm std})$ for arbitrary $n$.

\subsection{Hats and restrictions on fillings of contact manifolds}
Hats can give rise to caps via a branched cover construction: given a (suitable) projective hat $S$ for $T$ in $\Sst$, the $r$--fold cyclic  cover of the projective cap branched over $S$ is a cap for the $r$--fold cyclic cover $(\Sigma_r(T),\xi_{T,r})$ of $\Sst$ branched over $T$. (When $r=2$ we omit $r$ from the notation.)

\begin{Warning}
We advise the reader that, in the following two statements, we will abuse notation by denoting a transverse knot by its topological type; e.g., we will write $m(9_{46})$ to denote a transverse knot. What we mean is that we are considering the transverse knot obtained as the closure of the braid representing the knot taken from the KnotInfo database~\cite{knotinfo}. 
We also note that the data that we are using might agree with other knot databases or knot tables, as well as with other data on KnotInfo, \emph{only up to mirroring}.
\end{Warning}

We will use (some of) these caps to restrict the topology of symplectic fillings of branched double covers. 
In what follows, we denote with $E_8$ the unique negative definite, even, unimodular form of rank $8$, and with $H$ the hyperbolic quadratic form.
Our main result is the following. 

\begin{thm}\label{mainfilling}
If $K \subset (S^3,\xi_{\rm std})$ is one of the transverse knots in Table~\ref{table:braids}, then $\xi_K$ is Stein (and hence exactly) fillable. Let $(W,\omega)$ be an exact symplectic filling of $(\Sigma(K),\xi_K)$, with intersection form $Q_W$.
\begin{enumerate}
\item\label{item1} If $K$ is of type $12n_{242}$, then $W$ is spin, $H_1(W) = 0$, and $Q_W = E_8 \oplus H$.
\item\label{item2} If $K$ is of type $10_{124}$, $12n_{292}$, or $12n_{473}$, then $W$ is spin, $H_1(W) = 0$, and $Q_W = E_8$.
\item\label{item3} If $K$ is of type $m(12n_{121})$, then $W$ is spin, $H_1(W) = 0$, and $Q_W = H$.
\item\label{item4} If $K$ is of type $m(12n_{318})$, then $W$ is an integral homology ball.
\item\label{item5} If $K$ is of any of the following topological types, then $W$ is a rational homology ball:
\[
\begin{array}{lllll}
m(8_{20}), & m(9_{46}), &10_{140}, &m(10_{155}), &m(11n_{50}),\\
m(11n_{132}), &11n_{139}, &m(11n_{172}), &m(12n_{145}), &m(12n_{393}),\\
12n_{582}, &12n_{708}, &m(12n_{721}), &m(12n_{768}), &12n_{838}.
\end{array}
\]
\end{enumerate}
\end{thm}


\begin{table}[h!]
\begin{tabular}{lcl|lcl}
{\bf Knot type} & {\bf Strands} & {\bf Braid}&{\bf Knot type} & {\bf Strands} & {\bf Braid}\\
\hline
$m(8_{20})$ & 3 & $x^3yX^{3}y$ & $12n_{242}$ & 3 & $xy^2x^2y^7$\\
$m(9_{46})$ & 4 & $xYxYzyXyz$ & $12n_{292}$ & 4 & $xy^2x^3yZy^2xz^2$\\
$10_{124}$ & 3 & $(xy)^5$ & $m(12n_{393})$ & 5 & $yZwZyX^2zywz^2yx$\\
$10_{140}$ & 4 & $X^3yx^3yzYz$ & $m(12n_{318})$ & 4 & $xyzXzY^2xYzyxY$\\
$m(10_{155})$ & 3 & $x^3yX^2yX^2y$ & $12n_{473}$ & 4 & $xy^4z^2y^3xYz$\\
$m(11n_{50})$ & 4 & $x^2yXyzYxY^2z$ & $12n_{582}$ & 5 & $xYxyzwYwzyZWyZ$\\
$m(11n_{132})$ & 4 & $X^2yxzYxYzy^2$ & $12n_{708}$ & 3 & $x Y^3  x Yxy Xy^3$\\
$11n_{139}$ & 5 & $x^2yXzYzwZyZw$ & $m(12n_{721})$ & 3 & $Y^5x^4y^2x$\\
$m(11n_{172})$ & 4 & $xyXyxzYxY^{2}z$ & $m(12n_{768})$ & 4 & $Z^2y^2zY^2z^2yxYx$\\
$m(12n_{121})$ & 4 & $xyX^2yzYxy^2z^2Y$ & $12n_{838}$ & 5 & $xy Zw Xyzx Y Wzw$\\
$m(12n_{145})$ & 5 & $wZyZyX^2wyzYzyx$\\[0.2cm]
\end{tabular}
\caption{The braids whose closures are the transverse knots we consider in Theorems~\ref{mainfilling} and~\ref{highercovers}. We label positive generators in the braid group by $x,y,z,w$, in this order, and we denote with $X, Y, Z, W$ their inverses.}\label{table:braids}
\end{table}

\begin{rmk}\label{rmkabtfillings}
We note that $12n_{242}$ can also be described as the pretzel knot $P(-2,3,7)$ and $\Sigma(12n_{242})$ is known to be the Brieskorn homology sphere (with its natural orientation reversed) $-\Sigma(2,3,7)$.
By contrast to Item~\eqref{item1} in the theorem, $\Sigma(12n_{242})$ has minimal strong symplectic fillings with arbitrarily large $b_2^+$ (see the end of Section~\ref{ss:Sigmafillings} for a proof). Contact manifolds with a finite number of exact fillings but infinitely many strong fillings were already observed for cotangent bundles of hyperbolic surfaces~\cite{McDuff91, LiMakYasui} (see~\cite{SivekVanHornMorris} for a much stronger statement); as far as we are aware, this is the first example of an integral homology sphere such that the topology of Stein fillings is restricted, while that of strong symplectic fillings is not.

Compare also with work of Lin~\cite{LinF-fillings}; the manifold $-\Sigma(2,3,7)$ satisfies the assumptions of~\cite[Theorem 1]{LinF-fillings}, and therefore all of its Stein fillings \emph{that are not negative definite} have intersection form $E_8 \oplus H$. (In fact, it is easy to show using Floer-theoretic tools that $-\Sigma(2,3,7)$ cannot have any negative definite Stein fillings.)
\end{rmk}

\begin{rmk}
Recall that the branched double cover of the knot $10_{124} = T_{3,5}$ is the Poincar\'e sphere $\Sigma(2,3,5)$, endowed with the canonical contact structure $\xi_{\rm can}$, i.e. the one arising as the boundary of the singularity of $\{x^2+y^3+z^5 = 0\}$ at the origin of $\C^3$. We note here that Ohta and Ono~\cite[Theorem 2]{OhtaOno2} proved that every symplectic filling of $(\Sigma(2,3,5), \xi_{\rm can})$ is diffeomorphic to the $E_8$--plumbing, which is a stronger statement than what we are proving here.
\end{rmk}

We can also restrict the symplectic fillings of some higher-order cyclic branched covers. 
\begin{thm}\label{highercovers}
Let $(\Sigma_r(K), \xi_{K,r})$ denote the $r$--fold cyclic cover of $\Sst$, branched over the transverse knot $K$ of Table~\ref{table:braids}. Then $\xi_{K,r}$ is Stein fillable, and hence exactly fillable. Let $(W,\omega)$ be an exact filling of $(\Sigma_r(K), \xi_{K,r})$.
\begin{enumerate}
\item If $K$ is a quasipositive braid closure of knot type $m(8_{20})$, $m(9_{46})$, $10_{140}$, $m(10_{155})$, $m(11n_{{50}})$, and $r=3,4$, then $W$ is a spin rational homology ball.
\item If $K$ is a quasipositive braid closure of knot type $m(11n_{132})$, $11n_{139}$, $m(11n_{172})$, $m(12n_{318})$, $12n_{708}$, $m(12n_{838})$ and $r=3$, then $W$ is a spin rational homology ball.
\item If $K$ is a quasipositive braid closure of knot type $8_{21}$ and $r=3,4$, then $W$ is spin and $b_2(W) = 2(r-1)$.
\end{enumerate}
\end{thm}

These theorem follow by showing that each of these manifolds has a cap that embeds in a K3 surface. Thus the cap is Calabi--Yau and in~\cite{LiMakYasui} it was shown that such caps restrict the topology of fillings. 
Recall that a \emph{Calabi--Yau cap} of a contact $3$--manifold is a symplectic cap $(C,\omega)$ such that $c_1(\omega)$ is torsion~\cite{LiMakYasui}. We get the embedding of our cap into a K3 surface by taking the cover of $\CP$ or $\CPI \times \CPI$  branched over the union of a hat for a knot $K$ and a symplectic filling of $K$ which will be a curve of the appropriate degree or bi-degree.

The last statement in Theorem~\ref{highercovers} also uses Heegaard Floer theory to guarantee properties of the cap necessary to carry out the above argument. To illustrate a more subtle case where more sophisticated Heegaard Floer theory is used, we also prove the following result.

\begin{thm}\label{last}
Let $(W,\omega_W)$ be a Stein filling of $(\Sigma(2,3,7),\xi_{\rm can})$. Then $W$ is spin, it has $H_1(W)=0$ and either $H_2(W) \cong E_8 \oplus 2H$ or $H_2(W) \cong \langle -1\rangle$; moreover, both cases occur.
\end{thm}

We also establish a simpler analogous statement for $\Sigma(2,3,5)$ in Section~\ref{occ}.

\subsection{Hats and the generalized Thom conjecture}

Using hats we give a proof of the generalized Thom conjecture. The statement is well-known among specialists and we give a simple proof of it, but see also~\cite{GadgilKulkarni12} for another proof. 

\begin{thm}\label{t:majorThom}
Let $(X,\omega)$ be a strong symplectic filling of $(Y,\xi)$, $K$ a null-homologous transverse knot in $(Y,\xi)$, and $F\subset X$ an $\omega$--symplectic surface whose boundary is $K$, such that $F$ is transverse to $\partial X$.
Then $F$ minimizes the genus in its relative homology class (among all surfaces properly embedded in $X$ whose boundary is $K$).
\end{thm}

\subsection*{Organization of the paper} In Section~\ref{s:cobordisms} we discuss generalities on relative symplectic cobordisms between contact submanifolds in arbitrary dimension (and co-dimension). We study in more detail cobordisms between transverse links, giving general adjuction formulas for symplectic cobordisms, and we provide the basic building blocks for the construction: hats coming from complex curves and elementary symplectic cobordisms. In Section~\ref{s:projhat} we prove a (slight) strengthening of Theorems~\ref{t:CPcap1},~\ref{t:diskhats2}, and~\ref{negbraidnew}; we also provide many examples and computations.
In Section~\ref{s:hatsrational} we prove Theorem~\ref{t:diskhats1} and we compute minimal hat genus in blow-ups of the projective hats for some knots, including Proposition~\ref{goodslicenumberex}.
Section~\ref{s:brieskorn} is devoted to the proof of Proposition~\ref{p:Brieskornhats}, and Section~\ref{s:fillings} contains the proof of Theorems~\ref{mainfilling},~\ref{highercovers}, and~\ref{last}; some of the computations needed in this section are postponed to Appendix~\ref{a:braids}, while Appendix~\ref{a:majorThom} proves Theorem~\ref{t:majorThom}

\subsection*{Acknowledgments} We would like to thank Georgios Dimitroglou Rizell, Peter Feller, Kyle Hayden, Kyle Larson, Janko Latschev, Dmitry Tonkonog, and Jeremy Van Horn-Morris for interesting comments. We also thank the anonymous referees for their insightful comments that have improved the paper.
MG would like to warmly thank Laura Starkston for several inspiring conversations and for her insight. We started working on this project during the research program on \emph{Symplectic geometry and topology} at the Institute Mittag-Leffler; we acknowlegde the institute's hospitality and great working environment. Part of this work was carried over at the Institute of Advanced Studies (Princeton). The first author was partially supported by NSF grants DMS-1608684 and~1906414.

\section{General remarks on symplectic cobordisms between knots}\label{s:cobordisms}
In the first two subsections we will define relative symplectic cobordisms and discuss simple methods to build them. In the following section we discuss the adjunction equality for relative cobordisms in in symplectic 4--manifolds. The last two sections review quasipositive links and complex surfaces in $\CP$.

\subsection{Definitions and gluing}

A boundary component $M$ of a symplectic manifold $(X,\omega)$ is called {\em strongly convex} (respectively {\em strongly concave}) if there is a vector field $v$ defined near $M$ such that the the Lie derivative satisfies $\mathcal{L}_v\omega=\omega$ and $v$ points out of (respectively into) $X$ along $M$. We call $v$ a {\em Liouville vector field} (notice that we do not require $v$ to be defined on all of $X$).  

A {\em strong symplectic cobordism} from the contact manifold $(M_-,\xi_-)$ to the contact manifold $(M_+,\xi_+)$ is a compact symplectic manifold $(X,\omega)$ with $\partial X=-M_-\cup M_+$ where $(M_-, \xi_-)$ as a strongly concave boundary component and $(M_+,\xi_+)$ as a strongly convex boundary component. Unless otherwise specified, we will only consider strong symplectic cobordisms, hence we will systematically drop the adjective ``strong''.

We call $(C, \widetilde{\xi})$ a contact submanifold of $(M,\xi)$ if $C$ is transverse to $\xi$ and $T_pC\cap \xi_p=\widetilde{\xi}_p$ for all $p\in C$. 
Given two contact submanifolds $(C_\pm, \widetilde{\xi}_\pm)$ in $(M_\pm, \xi_\pm)$, we say they are {\em relatively symplectically cobordant} if there is 
\begin{enumerate}
\item\label{cond1} a symplectic submanifold $\Sigma$  of $(X,\omega)$ such that $(\Sigma,\omega|_\Sigma)$ is a symplectic cobordism from $(C_-, \widetilde{\xi}_-)$ to $(C_+, \widetilde{\xi}_+)$, and
\item\label{cond2} there there are Liouville vector fields $v_\pm$ for $(X,\omega)$ near $M_\pm$ that restrict to be Liouville vector fields for $(\Sigma,\omega|_\Sigma)$ near $C_\pm$.
\end{enumerate}
We call  $\Sigma$  a {\em relative symplectic cobordism}. We note that since the symplectic structure on $\Sigma$ comes from the restriction of the symplectic structure on $X$, Condition~\eqref{cond2} simply means that the Liouville vector fields for $(X,\omega)$ are tangent to $\Sigma$ near $C_\pm$. We note that while Condition~\eqref{cond2} is convenient to include in the definition, it may be replaced with 
\begin{enumerate}
\item[(2')]\label{cond2'} $\Sigma$ is transverse to the boundary of $X$
\end{enumerate}
if one is willing to deform the symplectic structure.
\begin{lemma}\label{maketangent}
Given a symplectic cobordism $\Sigma$ from $(C_-, \widetilde{\xi}_-)$ to $(C_+, \widetilde{\xi}_+)$ inside the symplectic cobordism $(X,\omega)$ from $(M_-,\xi_-)$ to $(M_+, \xi_+)$ as in Condition~\eqref{cond1} of relative symplectic cobordism, then as long as $\Sigma$ is transverse to $M_\pm$ we can assume, after deforming $\omega$ near $M_\pm$, that there are Liouville vector fields $v_\pm$ near $M_\pm$ that restrict to be Liouville vector fields for $(\Sigma,\omega|_\Sigma)$ near $C_+$. 

Moreover, this deformation is made by adding to $X$ a piece of the symplectization of $(M_-,\xi_-)$ and $(M_+,\xi_+)$.
\end{lemma}
We will need ideas from the proof of Lemma~\ref{immersedcob} to establish this lemma so the proof is given below. 

Below we will frequently build relative symplectic cobordisms in stages, so it is useful to note that the standard arguments for gluing together strongly convex and concave boundaries of symplectic manifolds, see for example \cite{Etnyre98}, easily generalize to give a relative gluing result. 
\begin{lemma}\label{glue}
Given two relative symplectic cobordisms $\Sigma^i$, $i=0,1$, from $(C_-^i,\widetilde{\xi}_-^i)$ to $(C_+^i,\widetilde{\xi}_+^i)$, inside the symplectic cobordisms $(X^i,\omega^i)$ from $(M_-^i,\xi_-^i)$ to $(M_+^i, \xi_+^i)$ for which there is a contactomorphic of pairs from $(M_+^0, C_+^0)$ to $(M_-^1, C_-^1)$ , then one may glue $X^0$ to $X^1$ along $M_+^0\cong M_-^1$ to obtain a symplectic cobordism $(X,\omega)$ from $(M_-^0,\xi_-^0)$ to $(M_+^1, \xi_+^1)$ and simultaneously glue $\Sigma^0$ to $\Sigma^1$ along $C_+^0\cong C_-^1$ to get a relative symplectic cobordism $\Sigma$ from $(C_-^0,\widetilde{\xi}_-^0)$ to $(C_+^1,\widetilde{\xi}_+^1)$. (We note that when gluing one can arrange that $(X^i,\omega^i)$ and a scaled version of $(X^{i+1},\omega^{i+1})$ are symplectic submanifolds of $(X,\omega)$, and similarly for the $\Sigma^i$. Here the indexing is taken mod 2.)\hfill\qed
\end{lemma}

Recall a {\em symplectic filling}, respectively {\em cap}, is a symplectic cobordism $(X,\omega)$ with $M_-=\emptyset$, respectively $M_+=\emptyset$. And given a contact submanifold $C$ in the boundary of a symplectic filling (or symplectic slice surface), respectively cap, then a relative cobordism from, respectively to, the empty set will be called a {\em symplectic filling}, respectively {\em hat}, for $C$.

\subsection{Constructing symplectic cobordisms}
We will need to consider regular homotopies of transverse knots. To this end we recall that a ``generic" regular homotopy $\phi_t:S^1\to M$ can be assumed to have isolated times at which there are isolated double points and at a double point the intersection is ``transverse" in the following sense:  if $t_i$ is a time for which there are values $\theta_1$ and $\theta_2$ such that $\phi_{t_i}(\theta_1)=\phi_{t_i}(\theta_2)$ then we consider the paths $\gamma_i(s)=\phi_s(\theta_i)$ and demand that $\gamma_1'(t_i)-\gamma_2'(t_i)$, $\phi_{t_i}'(\theta_1)$, and $\phi'_{t_i}(\theta_2)$ are linearly independent in $T_{\phi_{t_i}(\theta_1)}M$. We call a double point of a regular homotopy {\em positive} if the above basis defines the given orientation on $M$, and otherwise we call it {\em negative}.

We notice that if given a diagram of a knot in $\R^3$ and one switches a negative crossing to a positive crossing then that gives a generic regular homotopy with a positive double point. Switching a positive to a negative crossing gives a negative double point. 

\begin{rmk}
More generally, consider regular homotopies $\phi_t: C^{k}\to M^{2n+1}$. Generically if $k<n$ this will be an isotopy and for $k=n$ there will be isolated transverse double points and we can assign signs to them in a fashion analogous to the one discussed above. 
\end{rmk}

\begin{lemma}\label{immersedcob}
Let $\phi_t:(C^{2k+1},\xi')\to (M^{2n+1}, \xi), t\in [0,1]$ be a generic regular homotopy of contact immersions with $\phi_0$ and $\phi_1$ embeddings. 
The trace of this homotopy in any sufficiently large piece of the symplectization $[a,b]\times M$ of $\xi$ is an immersed symplectic cobordism from $\phi_0(C)$ in $\{a\}\times M$ to $\phi_1(C)$ in $\{b\}\times M$.

When $2k+1<n$, the trace is an embedded symplectic cobordism. When $2k+1=n$, the symplectic cobordisms has isolated double points that correspond 
to double points in the regular homotopy and will be positive double points if the crossing change in the homotopy is positive and negative otherwise.
\end{lemma}

\begin{proof}
Let $\beta$ be a contact form for $\xi'$ on $C$ and $\alpha$ be a contact form for $\xi$ on $M$.
Since $\phi_t$ is a contact homotopy we know that $\phi_t^*\alpha=f_t\, \beta$ for some $1$--parameter family of positive functions $f_t:C\to \R$. 
If $g:[0,1]\to [a,b]$ is any increasing function then the ``trace" of the isotopy is parameterized by 
\[
\Phi:[0,1]\times C\to [a,b]\times M: (t, p)\mapsto (g(t), \phi_t(p)).
\]
This clearly gives an immersion with double points corresponding to double points of the homotopy. 
Pulling back $d(e^t\alpha)$ yields 
\[
d(e^{g(t)} f_t\, \beta)=e^{g(t)}\left(g'(t)f_t+\frac{\partial f_t}{\partial t}\right)\, dt\wedge \beta + e^{g(t)} f_t d\beta
\] 
which is clearly a symplectic form on $[0,1]\times C$ whenever $g'(t)$ is sufficiently large and it may be taken to be arbitrarily large if $b-a$ is sufficiently large. We note for later use that if $h:[0,1]\to [a,c]$ is any function with derivative larger than $g$, then if $g$ can be used to parameterize a symplectic embedding then so can $h$. 

For the claim about the sign of the double point of the immersion use the notation for a double point established just before the statement of the lemma (here we only discuss the $3$--dimensional case that we will use below, but the higher-dimensional case is analogous).
The tangent space for one sheet of the surface at the intersection point $(t_i,\phi_{t_i}(\theta_1))$ will be spanned by the oriented basis $\{g'(t_i)\partial_t+\gamma_1'(t_i), \phi_{t_i}'(\theta_1)\}$ and the other sheet will be spanned by the oriented basis $\{g'(t_i)\partial_t+\gamma_2'(t_i), \phi_{t_i}'(\theta_2)\}$. This clearly gives an oriented basis equivalent to $\{\partial_t, \gamma_1'(t_i)-\gamma_2'(t_i), \phi_{t_i}'(\theta_1), \phi_{t_i}(\theta_2)\}$, 
which establishes the claim. 
\end{proof}
\begin{rmk}\label{generalglue}
Lemma~\ref{immersedcob} immediately allows us to generalize Lemma~\ref{glue} to allow for gluing relative symplectic cobordisms $(X^0,\Sigma^0)$ and $(X^1, \Sigma^1)$ under merely the hypothesis that $M_+^0$ and $M_-^1$ are contactomorphic by a contactomorphism taking $C_+^0$ to a contact submanifold that is contact isotopic to $C_-^1$. (See the lemma for notation.)
\end{rmk}

\begin{proof}[Proof of Lemma~\ref{maketangent}] We discuss the case of $C_+$, noting that the case of $C_-$ is analogous.
We can extend $(X,\omega)$ by adding a small piece $[0,\epsilon]\times M_+$ of the symplectization of $(M_+,\xi_+)$ and extending $\Sigma$ so that it is transverse to $\{t\}\times M_+$ and symplectic in the extension (this can be done since being symplectic is an open condition). Notice that $C_t=\Sigma\cap (\{t\}\times M_+)$, for $t\in [0,\epsilon]$, can be taken to be a contact submanifolds in $(M_+,\xi_+)$ (since being a contact embedding is an open condition). 

Now for sufficiently large $b$ let $(X',\omega')$ be the extension of $(X,\omega)$ by the piece $[0, b]\times M_+$ of the symplectization of $(M_+,\xi_+)$. Let $\phi:[0,\epsilon]\to [0,\epsilon]$ be a function that is the identity on $[0,\epsilon/2]$ and equal to zero near $\epsilon$. So $C_{\phi(t)}$ is a contact isotopy in $(M_+,\xi_+)$. If we take the function $g$ in the proof of Lemma~\ref{immersedcob} to be the identity on $[0, \epsilon/4]$ and have sufficiently large derivative outside of this interval, then the trace of $C_{f(t)}$ is a symplectic submanifold and can be used to extend $\Sigma$ to a symplectic submanifold $\Sigma'$ in $(X',\omega')$. Clearly $\Sigma$ is a symplectic cobordism from $C_-$ to $C_+$ and $\Sigma'$ is simply $C_+\times [b-\epsilon, b]$ near $\partial_+X'$ and hence tangent to the Liouville vector field $\partial_t$.
\end{proof}

We would now like to resolve double points, but we can only symplectically resolve positive double points. This results seems well-known, but the authors could not find a specific reference, so we provide an elementary proof based on the ideas above. 
\begin{lemma}\label{resolve}
Let $\Sigma$ be an immersed symplectic surface in the symplectic $4$--manifold $(X,\omega)$. If $p$ is a positive transverse double point of $\Sigma$, then one may remove a neighborhood of $p$ in $\Sigma$ and replace it with a symplectic annulus, resulting in a symplectic surface $\Sigma'$ with one less double point than $\Sigma$ and the genus increased: $g(\Sigma')=g(\Sigma)+1$. 
\end{lemma}
\begin{proof}
We first claim that $\Sigma$ can be deformed in a $C^0$--small way near a positive double point so that there is a Darboux chart about the double point in which $\Sigma$ is the union of the $(x_1,y_1)$--plane and the $(x_2,y_2)$--plane. To see this let $p$ be a transverse positive double point of $\Sigma$ and $U$ a neighborhood of $p$ such that the two sheets of $\Sigma\cap U$ are $S_1$ and $S_2$. A standard Moser-type argument constructs a symplectomorphism $\phi:U'\to V$ between a neighborhood $U'$ of $p$ contained in $U$ and an open ball $V$ about the origin in $(\R^4,\omega_{\rm std})$, so that $\phi(p)$ is the origin, $\phi(U\cap S_1)=S'_1$ is the intersection of the $(x_1,y_1)$--plane with $V$, and $\phi(U'\cap S_2)=S'_2$ is a surface tangent to the $(x_2,y_2)$--plane at the origin. So $S_2'$ near the origin is the graph of a function $F\! :\R^2\to \R^2\! :(x,y)\mapsto (f(x,y), g(x,y))$, with $f$, $g$, and their first derivatives vanishing at the origin.
Now let $\rho:[0,1)\to \R$ be a function that vanishes on $[0,\frac{\epsilon}2]$, is 1 outside $[0,\epsilon]$, and is monotonically increasing on $[\frac{\epsilon}2, \epsilon]$, with $\rho' < \frac4\epsilon$; let $\rho_t = t\rho(r) + 1-t$.
Consider the family of functions $(F_t)_{t \in [0,1]}$ defined by $F_t(x,y) = \rho_t(r)\cdot F(x,y)$, where $r=\sqrt{x^2+y^2}$. One may check that the symplectic form evaluated on $dF_t(\frac{\partial}{\partial x})$ and $dF_t(\frac{\partial}{\partial y})$  (that is on a basis for the tangent space to the graph of $F_t$) is:
\begin{equation}\label{e:omegavsrho}
1 + \rho_t^2(r)\cdot(f_xg_y- g_yg_x) + \frac{\rho_t'(r)\rho_t(r)}{r} \left(y f_x g + x g_y f - x f_y g - y g_x f \right).
\end{equation}
Therefore, the graph of $\overline{F}$ is a symplectic surface in $\R^4$ if and only if the quantity above is positive. Since the graph of $F$ is symplectic and $\rho_t$ is identically 1 for $r>\epsilon$, the only part to check is when $r < \epsilon$.
When $r \le \frac{\epsilon}2$, the third summand vanishes and the second summand is larger than $-1$ by our assumption on $F$.
Finally, when $\frac{\epsilon}2 < r < \epsilon$, the first two summands in~\eqref{e:omegavsrho} are strictly larger than $0$ since $S_2'$ is symplectic and $\rho_t$ is between 0 and 1. Each part of the last summand in~\eqref{e:omegavsrho} is of order $r^2$ by our assumption on $\rho'_t$. 
Thus if $\epsilon$ is taken small enough the last term can be made arbitrarily small, and hence the graph of $F_t$ is symplectic for each $t$, giving a symplectic isotopy from $S_2'$ to the graph of $F_1$. We have thus established our first claim.

Now to resolve the double point. Let $B$ be a round ball contained in our Darboux chart. Notice that $(S_1'\cup S_2')\cap \partial B$ is a transverse Hopf link. The surface $C_\epsilon= \{z_1z_2=\epsilon\}\cap B$ is a complex surface for positive $\epsilon$. In particular,  $C_\epsilon$ is symplectic with boundary a transverse link that is transversely isotopic to $(S_1'\cup S_2')\cap \partial B$ (via the isotopy given by $\epsilon$ going to zero). We may now use Remark~\ref{generalglue} to glue $C_\epsilon$ to $\overline{\Sigma-(B\cap \Sigma)}$ and thus resolve the double point at the expense of adding genus. 
\end{proof}

The above two observations immediately yield the following result. 
\begin{lemma}\label{l:transversedoublept}
If $K$ is a transverse link in $(M^3,\xi)$ that is obtained from the transverse link $K'$ by transverse isotopy and $g$ positive crossing changes, then there is a relative symplectic cobordism $\Sigma$ from $K'$ to $K$ in any sufficiently large piece $([a,b]\times M,d(e^t\alpha))$ of the symplectization of $(M,\xi)$.
Moreover, for knots the surface $\Sigma$ can be taken to have genus $g$. \hfill \qed
\end{lemma}

We also observe that a positive crossing can be added to a transverse knot via a symplectic cobordism.
This result also follows from combining~\cite[Lemma~5.1]{Hayden17Pre} and~\cite[Example~4.7]{Hayden17Pre}, but the simple argument is presented here for completeness.
\begin{lemma}\label{addposcrossing}
If $K$ is a transverse link in $(M,\xi)$ and a portion of $K$ in a Darboux ball is as shown on the left of Figure~\ref{PosCrossing}, then there is a symplectic cobordism $\Sigma$ in a piece of the symplectization of $(M,\xi)$ from $K$ to the knot $K'$ obtained from $K$ by replacing the tangle on the left of Figure~\ref{PosCrossing} by the one on the right. 
\end{lemma}
\begin{figure}[ht]
\tiny
\begin{overpic}
{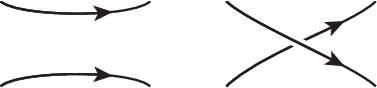}
\end{overpic}
\caption{Front diagrams for transverse tangles in a Darboux ball.}
\label{PosCrossing}
\end{figure}    
\begin{proof}
We claim that we can construct a surface $\Sigma'$ in $M$ by adding a twisted 1--handle to $K\times [0,\delta]$ so that $\partial \Sigma'= -K\cup K'$ and $d\alpha$ is positive on $\Sigma'$ where $\alpha$ is a contact 1--form for $\xi$.
Given this take any piece $[a,b]\times M$ of the symplectization of $(M,\xi)$ and take $\Sigma'$ to be a subset of $\{b\}\times M$. Any small isotopy that pushes $\Sigma'-K'$ into $[a,b)\times M$ will result in a surface $\Sigma$ that is symplectic. So take the isotopy so that $K\subset \Sigma$ sits on $\{b-\epsilon\}\times M$ and $\Sigma-\partial\Sigma$ is in $(b-\epsilon,b)\times M$. This is a  cobordism from $K$ to $K'$ satisfying condition (1) of symplectic cobordism. Lemma~\ref{maketangent} allows us to extend the cobordism to satisfy both conditions of a symplectic cobordism. 

We are left to show that $\Sigma'$ exists. To this end notice that one may easily construct an annulus $A$ with one boundary $K$ and the other boundary a copy, $\widetilde K$, of $K$ so that the characteristic foliation on the annulus is by arcs running form one component boundary to the other. We can then add a 1--handle to $A$ to get a surface $\Sigma'$ with transverse boundary $-K$ and $K'$ and the only singular point in the characteristic foliation of $\Sigma'$ a positive hyperbolic point in the 1--handle. From this it is easy to construct an area form $\omega$ on $\Sigma'$ and a vector field $v$ directing the characteristic foliation so that $d\iota_v \omega$ is a positive multiple of $\omega$ (that is, $v$ has positive divergence on $\Sigma'$), see~\cite{Giroux91}.
Let $\beta=\iota_v\omega$. One may easily see that $f \beta = \alpha|_{T\Sigma'}$ for some positive function $f$. Now in a neighborhood $N=[-\epsilon,\epsilon]\times \Sigma'$ of $\Sigma'$, with $\Sigma'=\{0\}\times \Sigma'$, we know that $\alpha$ is of the form $\beta_t+ u_t \, dt$ where $\beta_t$ and $u_t$ are 1-forms and functions, respectively, on $\Sigma'$ and $\beta_0=f\beta$. Multiplying $\alpha$ by $1/f$ we can assume that the contact form for $\xi$ is $\beta_t+ u_t \, dt$ with $\beta_0=\beta$. But now $d\alpha$ on $T\Sigma'$ is $d\beta$ which is a positive area from on $\Sigma'$ as desired. 
\end{proof}

Below we will sometimes use the well-known notion of an open book decomposition and it supporting a contact structure. We do not discuss this here, but refer the reader to \cite{Etnyre06} for more details. 
\begin{ex}
The main application of Lemma~\ref{addposcrossing} in this paper is to braid closures; recall that one can associate to a braid a transverse knot in $\Sst$, which is just the closure of the braid, viewed as being transverse to the pages of the standard open book of $\Sst$ with disk pages.
In this context, the operation of adding a crossing to the closure of $\beta \in B_n$ in the lemma corresponding to just adding a positive braid generator to any braid factorization of $\beta$ (in any position).
By contrast, Lemma~\ref{l:transversedoublept} corresponds to adding the square of a generator.
\end{ex}

More generally, we note that Lemma~\ref{l:transversedoublept} also follows from Lemma~\ref{addposcrossing} and Lemma~\ref{immersedcob} for isotopies (we do not need the statement for regular homotopies) since a negative to positive crossing change can be effected by adding two positive crossings. Again we note that Lemma~\ref{addposcrossing} and the isotopy version of Lemma~\ref{immersedcob} are contained in~\cite{Hayden17Pre}, and thus our main observation of this section, namely Lemma~\ref{l:transversedoublept}, easily follows from~\cite{Hayden17Pre} as well. 

We end this subsection by noting that open book decompositions can be used to construct relative symplectic fillings. 
\begin{lemma}
Let $B$ be the binding of an open book decomposition of $M$ that supports the contact structure $\xi$. If $\Sigma$ is a page of the open book then in in a piece of the symplectization $([a,b]\times M, d(e^t\alpha))$, for some contact form $\alpha$, we can take $\Sigma$ in $\{b\}\times M$ and push its interior into the interior of $[a,b]\times M$ to get a symplectic filling of $B$. 
\end{lemma}
\begin{proof}
Since the open book decomposition supports $\xi$ there is a contact form $\alpha$ for $\xi$ for which $d\alpha$ is positive on the pages of the open book. Thus the symplectic form $e^t(dt\wedge \alpha + d\alpha)$ is positive on $\Sigma$ and hence on $\Sigma$ when its interior is pushed slightly into the interior of $[a,b]\times M$. Now this can be done so that the perturbed $\Sigma$ is transverse to $\{b\}\times M$. Thus Lemma~\ref{maketangent} gives the desired result. 
\end{proof}
\begin{rmk}
One might expect the same argument to work to construct a symplectic hat for the binding of an open book, but this does not work since the orientation on $B$ induced from the page is not correct to be the lower boundary component of a relative symplectic cobordism. 
\end{rmk}

\subsection{Symplectic submanifolds}

A simple bundle theory argument yields the following useful fact for closed, immersed, symplectic surfaces. 
\begin{lemma}[McCarthy--Wolfson, 1996, \cite{MW}]\label{MWequality}
Let $\Sigma$ be an immersed symplectic sub-surface of a symplectic 4--manifold $(X,\omega)$. Then,
\[
 \langle c_1(X,\omega), [\Sigma]\rangle = 2-2g + [\Sigma]\cdot [\Sigma] - 2D,
\]
where $g$ is the genus of $\Sigma$, $[\Sigma]$ denote the homology class determined by $\Sigma$, and $D$ be the number of double points of $\Sigma$ counted with sign. \hfill \qed
\end{lemma}
We have the following relative version of this for symplectic  cobordisms.
\begin{lemma}\label{filladjunction}
Let $(X,\omega)$ be a symplectic cobordisms from the contact manifold $(M,\xi)$ to $(M',\xi')$, $C$ a transverse knot in $(M,\xi)$, and $C'$ a transverse knot in $(M',\xi')$. Further assume that $M$ and $M'$ are homology spheres. If $\Sigma$ is any immersed symplectic surface with transverse double points in $(X,\omega)$ with boundary $-C\cup C'$, then 
\[
 \langle c_1(X,\omega), [\overline\Sigma]\rangle = \chi(\Sigma) - \slk(C) + \slk(C') + [\overline\Sigma]\cdot [\overline\Sigma] - 2D,
\]
where $[\overline\Sigma]$ is the homology class of the closed surface $\overline\Sigma= \Sigma\cup S\cup -S'$ where $S$  is any Seifert surface for $C$ in $M$ and $S'$ is a Seifert surface for $C'$ in $M'$, $g(\Sigma)$ is the genus of $\Sigma$ and $D$ be the number of double points of $\Sigma$ counted with sign.
\end{lemma}

\begin{rmk}
It is not essential that $M$ and $M'$ are homology spheres, but when they are not one must still assume that $C$ and $C'$ are null-homologous so that the self-linking number is defined. In this case the self-linking number will depend on the choice of Seifert surface and this surface must also be used in defining $\overline\Sigma$. 
\end{rmk}

We note a couple of consequences.

\begin{enumerate}
\item (The relative symplectic Thom conjecture) A symplectic surface $\Sigma$ with boundary properly embedded in a symplectic filling $(X,\omega)$ that is transverse to the boundary, minimizes genus in its relative homology class. We prove this in Appendix~\ref{a:majorThom}, {\em cf.\ }\cite{GadgilKulkarni12}.
\item If a transverse knot $T$ in $\Sst$ boundary a symplectic surface $\Sigma$ in $(B^4,\omega_{\rm std})$ then 
\[
\slk(T)=2g(\Sigma) - 1.
\] 
In particular, a stabilized transverse knot cannot be the boundary of an embedded symplectic surface in $B^4$. 

To see this note that such a surface would have lower genus that the one $T$ bounds and then this surface would violate the  relative symplectic Thom conjecture. 
\end{enumerate}

\begin{proof}
Let $R_\alpha$ and $R_\alpha'$ be a Reeb vector fields for $\xi$ and $\xi'$, respectively, and $t$ the coordinate normal to $-M\cup M=\partial X$. By adding a collar neighborhood to the boundary of $X$ and extending $\Sigma$ we can assume that $C$ and $C'$ are orbits of the Reeb vector field. 

Notice that the tangent space $TX$ restricted to $\overline\Sigma$ splits (as a symplectic bundle) as $E_1\oplus E_2$ where $E_1$ is $T\Sigma$ along $\Sigma$ and the span of $R_\alpha$ and $\partial_t$ along $S\cup S'$, and $E_2$ is the symplectic normal bundle to $\Sigma$ along $\Sigma$, $\xi$ along $S$ and $\xi'$ along $S'$. So restricted to $\overline\Sigma$ we have 
\[
c_1(X,\omega)= c_1(TX)=c_1(E_1)+c_1(E_2).
\]

To compute $\langle c_1(E_1), [\overline\Sigma]\rangle$ we choose the section $\partial_t$ over $S\cup S'$ and extend it arbitrarily over $T\Sigma$. So clearly $\langle c_1(E_1), [\overline\Sigma]\rangle= \chi(\Sigma)$. Now to compute $\langle c_1(E_2), [\overline\Sigma]\rangle$ choose a non-zero section $s$ of $\xi$ over $S$, $s'$ of $\xi'$ over $S'$, and extend it arbitrarily over the normal bundle to $\Sigma$.  Clearly $\langle c_1(E_2), [\overline\Sigma]\rangle$ is the relative Chern class of the normal bundle $\nu$ to $\Sigma$ relative to $s\cup s'$ along $-C\cup C'= \partial \Sigma$. To compute this we choose another section of the normal bundle. Let $\sigma$ and $\sigma'$ be the unit normal vector fields along $S$ and $S'$, respectively. Along $\partial \Sigma$, $\sigma$ and $\sigma'$ are contained in the normal bundle to $\Sigma$. Computing the relative Chern class of $\nu$, relative to $\sigma$ and $\sigma'$, evaluated on $\Sigma$ clearly gives  $[\overline\Sigma]\cdot [\overline\Sigma] - 2D$ since we can use $\sigma$, $\sigma'$ and their extension over $\Sigma$ to create a section of the normal bundle of $\overline\Sigma$

We finally notice that the difference between the framings that $s$ and $\sigma$ give to $C$ is $-\slk(C)$ and the difference between the $s'$ and $\sigma'$ framings of $C'$ is $\slk(C')$. The former is just the definition of the self-linking number, while the latter is also the definition but we must remember that $\partial X= -M\cup M'$ and the linking numbers in $M$ and $-M$ differ by a sign.  Hence 
\[
\langle c_1(E_2), [\overline\Sigma]\rangle= -\slk(C)+ \slk(C') + [\overline\Sigma]\cdot [\overline\Sigma] - 2D.\qedhere
\]
\end{proof}

\subsection{Quasipositivity and links bounding symplectic slice surfaces}\label{qpintro}

Recall that the $n$-strand braid group $B_n$ is generated by $n-1$ elementary generators, $\sigma_1,\ldots, \sigma_{n-1}$, where $\sigma_i$ interchanges the $i^{\rm th}$ and $(i+1)^{\rm st}$ strands with a positive half-twist. For more on the braid group see \cite{Birman}. A braid is called {\em quasipositive} if it can be written as a product of conjugates of non-negative powers of the standard generators and it is called {\em strongly quasipositive} if it can be written as a product of the elements
\[
\sigma_{ij}= (\sigma_i\ldots \sigma_{j-2})^{-1}\sigma_{j-1}(\sigma_i\ldots \sigma_{j-2}),
\]
for $1\leq i<j<n$. A link in $S^3$ is called quasipositive or strongly quasipositive when it can be realized as the closure of such a braid. Combining work of Rudolph \cite{Rudolph83} and Boileau and Orevkov \cite{BoileauOrevkov} it is known that the class of quasipositive links is precisely the class of links that arise as the transverse intersection of a complex surface in $\C^2$ with the unit sphere; these are sometimes called {\em transverse $\C$-links}. Moreover Theorem~2 in  \cite{BoileauOrevkov} makes it clear that the class of links is also precisely the class of links that arise as the transverse intersection of a symplectic surface in the unit ball in $\C^2$ (with standard symplectic structure) with the unit sphere. Given Lemma~\ref{maketangent}, we see that a transverse link in $\Sst$ bounds a symplectic slicing surface in the 4--ball if and only if it is given as the closure of a quasipositive braid. 

We now turn to a special class of quasipositive links, namely links of algebraic singularities. Given a complex polynomial $f(z,w)$ in two variables, let $V\!(f)= f^{-1}(0)$. Suppose that $x\in V\!(f)$ is an isolated singular point of $f$. Then for small enough $\epsilon>0$ the sphere of radius $\epsilon$, $S_\epsilon$, about $x$ intersects $V\!(f)$ transversely in a link $L_{f,x}$. For $\delta$ sufficiently small $f^{-1}(\delta)$ will also intersect the $S_\epsilon$ transversely in a link isotopic to $L_f$. This surface is called the Milnor fiber of $L_f$.  So $L_f$ is a quasipositive link (in fact it is strongly quasipositive).
For a topologist-friendly introduction to singularity of curves in the spirit of this paper, we refer to~\cite[Section 2]{GStarkston}
The main example we will consider in this paper is that of $f(z,w) = z^p-w^q$. In this case $L_{z^p-w^q,0}$ is the $(p,q)$-torus link. It is also well known that, when $p$ and $q$ are coprime, the complex surface that $L_{z^p-w^q,0}$ bounds in the $4$--ball has genus $\frac12(p-1)(q-1)$.

\subsection{Complex curves in $\CP$}\label{planecurves}
We will be considering algebraic curves in $\CP$. More specifically, given a non-zero homogeneous polynomial $f(x,y,z) \in \C[x,y,z]$ one can consider the set 
\[
V\!(f)=\{[x:y:z]\in \CP \mid f(x,y,z)=0\}.
\] 
This is a complex surface in $\CP$. We say it has degree $d$ if the polynomial has degree $d$. Moreover, recall that the second homology of $\CP$ is generated by the homology class of a line $\ell\subset \CP$ and one can easily check that the homology class defined by $V\!(f)$ agrees with $d[\ell]$, thus giving another interpretation of the degree of $V\!(f)$. 

A point in $V\!(f)$ where the derivative of $f$ vanishes will be called a {\em singular point}. If $P$ is a singular point then for sufficiently small ball $B$ about $P$, $V\!(f)$ will intersect $\partial B=S^3$ transversely in some link $L_{f,P}$. Clearly $L_{f,P}$ is a quasipositive link and so bounds a complex surface $\Sigma_{f,P}$ in $B$. If the links associated to all the singular points of $V\!(f)$ are connected (that is are knots) then we say $V\!(f)$ is a {\em cuspidal curve}. A cuspidal curve is a PL embedded surface of some genus $g$. Replacing neighborhoods of all the singular points of $V\!(f)$ with the complex surfaces $\Sigma_{f,P}$ and recalling that $c_1(\CP)=3[\ell]$ one can apply Lemma~\ref{MWequality} to see that 
\[
3d = \langle c_1(\CP), [\Sigma'']\rangle = 2- 2\left(g + \sum g(\Sigma_{f,P})\right) +d^2,
\]
where the sum is taken over all the singular points of $V\!(f)$. This yields 
\begin{equation}\label{singulargenusformula}
g + \sum g(\Sigma_{f,P}) = \frac{(d-2)(d-1)}{2}.
\end{equation}
We will take a topological viewpoint on singularities, similar to that of~\cite[Section 2.2]{GStarkston}. In particular, we will use the following fact: if we blow up the plane $\C^2$ at the origin and we let $E$ denote the exceptional divisor, the proper transform of the curve $V\!(x^p - y^q)$ (with $p<q$) has multiplicity of intersection $p$ with the $E$, and it has a singularity isomorphic to that of $V\!(x^p - y^{q-p})$.

\section{Hats in the punctured projective plane}\label{s:projhat}

In this section we will show that all transverse knots in the standard contact $S^3$ have a hat in $\Xom = (\CP\setminus B^4,\omega_{\rm FS})$, where $B^4$ is embedded as a Darboux ball with convex boundary in $\CP$ and $\omega_{\rm FS}$ is the Fubiny--Study metric; and compute the hat genus for many examples. In particular, we show that the symplectic hat genus can differ from the genus of a smooth surface in $X$ with boundary the knot. 

Let $\ell_\infty$ be a line at infinity in $\CP$ (that is the standard $\CPI$ in $\CP$) that is in the complement of the $B^4$ removed above.
Let $K$ be a transverse knot in $\Sst$. A hat $\Sigma$ for $K$ in $\Xom$ will be called a {\em projective hat} for $K$.
By Poincar\'e--Lefschetz duality and elementary algebro-topological manipulations,
\[
H_2(X, \partial X) \cong H^2(X) \cong H^2(\CP) \cong H_2(\CP) \cong \Z.
\]
We can give an explicit isomorphism by choosing a line $\ell_\infty$ in $\CP$ that is contained in $X$, and use the intersection pairing $H_2(X,\partial X) \otimes H_2(X) \to \Z$ to define the {\em degree} of $\Sigma$ as the intersection number of $\Sigma$ and $\ell_\infty$.

\subsection{Existence of projective hats}
It is easy to see that, for each knot $K$ there are always smoothly embedded surfaces in $X$ with any degree and boundary $K$. These have been studied in~\cite{Nouh09}, but more work is necessary to prove the existence of symplectic hats, and we will see that the degree cannot be arbitrary for a given $K$.

The following is a slight extension of Theorem~\ref{t:CPcap1} from the introduction.

\begin{thm}\label{p:CPcap}
Every transverse link $K$ in $\Sst$ has a projective hat. Moreover, this hat can have any sufficiently large degree. 
\end{thm}

We begin with a lemma. 

\begin{lemma}\label{l:torushats}
The transverse representative of the positive torus knot $T_{p,q}$ in $\Sst$ with self-linking number $pq-p-q$ wears a symplectic projective hat of genus
\[
\frac{(q-p-1)(q-1)}2
\]
and degree $q$, where we are assuming, without loss of generality, that $q>p$.
\end{lemma}

\begin{rmk}
The knot $T_{p,q}$ might wear a hat of smaller genus and degree if $p$ is sufficiently small, but it is interesting to note that these numbers are optimal if $q<2p$. To see this suppose that $q<2p$, and that there is a degree-$d$ hat for $T_{p,q}$.
We apply~\cite[Equation $(\star_j)$, page 523]{BCG} with $j=1$.
The inequality reads
\begin{equation}\label{e:starj}
\Gamma(3) = \Gamma(\Delta_1) \le d,
\end{equation}
where $d$ is the degree of the hat, and $\Gamma(k)$ is the $k^{\rm th}$ element of the semigroup comprising all non-negative integer linear combinations of $p$ and $q$ (starting at $\Gamma(1) = 0 = 0p+0q$).

Since $q<2p$, the first three elements of the semigroup are $0, p, q$, and therefore $\Gamma(3) = q$; substituting in the inequality above, we obtain that $q \le d$, as claimed.
\end{rmk}

\begin{proof}
There are many ways to construct a hat for $T_{p,q}$. A particularly simple one was pointed out to us by Dmitry Tonkonog. Consider the curve  $C = V\!(x^q - y^pz^{q-p} + y^q)$. It is immediate to check that the only singularity of $C$ is at $(0:0:1)$. Moreover, the following construction gives a local change of coordinates around $(0,0)$ in the affine chart $\{z=1\}$ that maps $V\!(x^q-y^p)$ to $C$. Let $g$ be a $p^{\rm th}$ root of the function $ w\mapsto 1-w^{q-p}$: this exists locally in a ball $B$ centered at $w=0$ since $g(0)\neq 0$, and set $h(w)=wg(w)$. The latter function is a biholomorphism since $g(0)\neq 0$.
It is immediate to check that in the chart $\{z=1, y\in B\}$, the biholomorphism $(x,y)\mapsto (x,h(y))$ maps $V\!(x^q-y^p)$ to $C$.
Thus the complement of a neighborhood of the singular point gives the desired hat. 
The degree of the hat is clearly $q$; the Adjunction Formula~\eqref{singulargenusformula} gives its genus to be 
\[\frac12(q-1)(q-2) - \frac12(p-1)(q-1) = \frac{(q-p-1)(q-1)}2,\]
as claimed.
\end{proof}

\begin{rmk}
There is an alternative approach to proving the lemma, which is closer to the spirit of this paper. One can start from the curve $V\!(x^q-y^pz^{q-p})$; it is a rational curve, since the map $[s:t] \mapsto [s^pt^{q-p}:s^q:t^{q}]$ gives a parametrization by $\CPI$.
Moreover, it has two singularities at $(0:0:1)$ and at $(0:1:0)$. The singularity at $(0:0:1)$ is of the desired type $x^q-y^p = 0$, and we can trade the other singularity, which is of type $x^q - y^{q-p}$, for its Milnor fiber, which has genus $\frac12(q-p-1)(q-1)$.
\end{rmk}

\begin{rmk}
The statements in~\cite{BCG} are about complex curves; however, since the proofs use smooth 4-dimensional topology techniques, they hold more generally for reals surfaces whose singularities are cones over knots.
The Inequality~\eqref{e:starj} holds for smooth curves having only one singularity whose cone is a cone over a torus knot, so they apply to our case.
\end{rmk}

\begin{proof}[Proof of Theorem~\ref{p:CPcap}]
We will build a symplectic cobordism from $K$ to a positive torus knot and then use Lemma~\ref{glue} to glue this to a symplectic hat for the torus knot constructed in Lemma~\ref{l:torushats}.

Given a transverse knot $K$ we can transversely isotope it so that it is braided. Thus we can use Lemma~\ref{immersedcob} to build a symplectic cobordism from $K$ to a closed $n$-braid. Now Lemma~\ref{addposcrossing}, which says we can add positive crossings wherever we like, allows us to build a symplectic cobordism from the braid to the closure of $(\sigma_1\cdots \sigma_{n-1})^k$ for any sufficiently large $k$.
Since for $k$ relatively prime to $n$ the braid  $(\sigma_1\cdots \sigma_{n-1})^k$ is a positive torus knot we have constructed the desired symplectic cobordism. 
\end{proof}

The following result will be useful in the next section, and it can easily be combined with Lemmas~\ref{glue} and~\ref{l:torushats} to give an alternate proof of Theorem~\ref{p:CPcap} .

\begin{lemma}\label{immersedannulus}  
Every transverse link $K$ in $\Sst$ has an immersed symplectic cobordism with only positive double points to a torus knot in a piece of the symplectization of $\Sst$ . 
\end{lemma}

To prepare for the proof of this lemma we set up some notation. Given two braids $\beta$ and $\beta' $ we will write $\beta \uparrow \beta'$ to indicate that $\beta'$ is obtained from $\beta$ by inserting a square of a generator into some presentation of $\beta$ as a word in the generators. 
A braid $\beta'$ is \emph{generated by squares} from $\beta$ if there exists a sequence 
\[
\beta = \beta_0, \beta_1, \dots, \beta_m = \beta'
\] 
such that $\beta_{k+1}$ is obtained from $\beta_k$ by inserting the square of a generator.
When $\beta = e$ is the identity braid, we simply say that $\beta'$ is generated by squares.

Observe that if $\beta'$ is generated by squares from $\beta$, there is a sequence of positive crossing changes from the closure of $\beta$ to the closure of $\beta'$.

\begin{lemma}\label{l:garside}
The square of the Garside element $\Delta_n^2\in B_n$ is generated by squares from $\sigma_i^2$ for each $1\le i \le n-1$.
\end{lemma}

\begin{proof}
We prove this by induction on $n$.
If $n=2$, $\Delta_1^2 = \sigma_1^2$, and this is clearly generated by squares from $\sigma_1^2$.

If $n\ge 2$, then $\Delta_{n+1}^2 = (\iota_n\Delta_n^2) \cdot (\sigma_n\cdots\sigma_2 \sigma_1^2\sigma_2 \cdots \sigma_n)$, where we denoted by $\iota_n: B_n \to B_{n+1}$ the inclusion of the first $n$ strands.
Both factors are generated by squares: the first by the inductive assumption, and the second by direct inspection.

In particular, this shows that $\Delta_{n+1}^2$ is generated by squares from $\sigma_n^2$ (since the last factor is) or by any of $\sigma_1^2,\dots,\sigma_{n-1}^2$ (since the first factor is).
\end{proof}

\begin{proof}[Proof of Lemma~\ref{immersedannulus}]
The transverse knot $K$ is the closure of an $n$-braid $\beta\in B_n$.
Up to conjugation, we can suppose that $\beta$ induces the permutation $(1\,2\cdots n)$.
Let $\beta_0 = \sigma_1\sigma_2\cdots \sigma_{n-1}$, and observe that $\gamma = \beta_0^{-1}\beta$ is in the pure braid group.

In particular $\gamma$ is a product of elements of the form $w_i\sigma^{\epsilon_i}_{k_i} w_i^{-1}$, where $\epsilon_i = \pm2$ and $w_i$ is an arbitrary word in the braid group for each $i$ (see, e.g., \cite[Lemma 1.8.2]{Birman}).
We claim that for some integer $m$, $\Delta^{2m}$ is generated by squares from $\gamma$. 
This proves that $\beta_0\Delta^{2m}$ is generated by squares from $\beta$, and in particular there is a symplectic cobordism from the closure of $\beta$ to the closure of $\beta_0\Delta^{2m}$, which is the torus knot $T_{n,mn+1}$. Now the desired immersed cobordism follows from Lemma~\ref{immersedcob}. 

Let us now prove the claim.
For each $i$ such that $\epsilon_i = -2$ we simply change the corresponding crossing by inserting a $\sigma_{k_i}^2$:
\[
w_i\sigma^{-2}_{k_i} w_i^{-1} \uparrow w_i\sigma^{-2}_{k_i}\sigma^{2}_{k_i} w_i^{-1} = e.
\]

For each $i$ such that $\epsilon_i = 2$, we use Lemma~\ref{l:garside}, which asserts that  $\Delta^2$ is generated by square from $\sigma_{k_i}^2$; indeed, we have
\[
w_i \Delta^2 w_i^{-1} = \Delta^2,
\]
since $\Delta^2$ is central in the braid group $B_n$.
That is, we have proven $ \Delta^{2m}$ is generated by squares from $\gamma$.
\end{proof}

\subsection{Projective hat genus}
We can now define two invariants for transverse knots in $\Sst$.

\begin{defn}
We call the \emph{hat genus} of $K$ the smallest genus $\ghat(K)$ of a symplectic hat of $K$ in $\Xom$ and the \emph{hat degree} to be the minimal degree $\widehat{d}(K)$ of a symplectic hat for $K$. 
\end{defn}
Later we will discuss hats in other caps for $\Sst$ and then when confusion might arise we will refer to the hat genus and hat degree, as the {\em projective hat genus} and {\em projective hat degree}, respectively. 

\begin{ex}
Notice that $\ghat(K) = 0$ if and only if $K$ has a symplectic projective hat that is a disk.
For example, $\ghat(K) = 0$ for whenever $K=T_{p,p+1}$ has maximal self-linking number. 
In fact, there exists a rational singular curve whose unique singularity has link $T_{p,p+1}$, namely $V(x^pz-y^{p+1})$. 
\end{ex}

We note that using the adjunction formula for hats, in Lemma~\ref{filladjunction}, for a fixed transverse knot the hat genus determines the hat degree and vice-versa. 
\begin{lemma}\label{lb}
If $\Sigma$ is a projective hat for a transverse knot $K$ in $\Sst$ then 
\[
\ghat(K)= -\left(\frac{\slk(K) + 1}{2}\right) + \frac{\widehat{d}(K)^2-3\widehat{d}(K)+2}{2}\geq -\left(\frac{\slk(K)+1}2\right)
\]
and
\[
\slk(K)=(\widehat{d}(K)^2-3\widehat{d}(K)+1)-2\ghat(K).
\]
\end{lemma}

\begin{proof} 
If $\Sigma'$ is a Seifert surface for $K$ and $\overline{\Sigma}=\Sigma\cup \Sigma'$, then $\overline \Sigma$ represents a homology class $dh$ in $H_2(\CP)\cong\Z$ where $h$ is the generator of homology given by a line. Recalling that  
\[
c_1(\CP)= c_1(\CP-B^4) = 3h,
\] 
then the equation in Lemma~\ref{filladjunction} immediately gives 
\[
3\widehat{d}(K) = 1-2\ghat(K) -\slk(K) + \widehat{d}(K)^2
\]
which is equivalent to the stated formula. 
\end{proof}
As a corollary we see that the adjunction formula gives lower bounds on the hat genus of quasipositive knots.
\begin{cor}\label{p:next-triangular}
If $K$ is a quasipositive knot with $4$--ball genus $g_s(K)$, and  
\[
m= \frac{(d-2)(d-1)}{2}
\] 
is the smallest triangular number $m\ge g_s(K)$, then 
\[
\ghat(K) \ge m-g_s(K).
\]
Moreover, the hat degree must satisfy
\[
\widehat{d}(K)\geq d.
\]
\end{cor}
\begin{rmk}
Since the gaps between consecutive triangular numbers can be made arbitrarily large and any genus can be realized by a quasipositive knot, this result shows that the hat genus and hat degree can each be made arbitrarily large. 
\end{rmk}

\begin{proof}
Since $K$ is quasipositive, $K$ bounds a symplectic curve $\check\Sigma$ in $(B^4,\omega)$ by~\cite{Rudolph83}, 
and gluing $\check\Sigma$ with a cap $\widehat\Sigma$ of minimal genus in $\Xom$ yields a closed symplectic surface $\Sigma$ in $\CP$.
The genus of $\check\Sigma$ is the quasipositive genus $g_s(K)$. (Here, and below, we use the phrase ``quasipositive genus" of a knot to mean the genus of a symplectic surface in $(B^4,\omega)$ with boundary the given knot.) We know from Lemma~\ref{filladjunction}, and the comment after the lemma, that $sl(K)=2g_s(K)-1$. Thus Proposition~\ref{lb} gives 
\[
\ghat(K)= -g_s(K)+ \frac{(\widehat{d}(K))^2-3\widehat{d}(K)+2}{2}
\]

So if $d$ is the smallest $d$ as in the statement of the corollary then the stated results follows. 
\end{proof}

\begin{rmk}
In fact, we claim here that the set of genera that are realized by hats for $K$ contains all possible genera ({\em i.e.\ }all genera satisfying the adjunction formula for some degree) past some sufficiently large constant.

To see this, observe that the hat constructed in Proposition~\ref{p:CPcap} is algebraic outside a tubular neighborhood $N$ of $S^3$.
Therefore, there is family of complex lines in the complement of $N$.
A generic line in this family intersects the hat transversely, and only where the hat is algebraic; therefore, all intersections are positive, and hence smoothable in the symplectic category.
Choosing any finite set of such generic lines and smoothing all double points yields the desired hats.
\end{rmk}

We now make an observation concerning the relation between self-linking numbers and the hat genus. Recall that given a transverse knot $T$ one can form the transverse stabilization $S(T)$ of $T$ (if $T$ is given as the closure of a braid then $S(T)$ is the closure of a negative Markov stabilization of $T$). We know that stabilization decreases the self-linking number by $2$: $\slk(S(T))=\slk(T)-2$. 

\begin{prop}\label{ulbounds}
Given a transverse knot $T$ in $\Sst$ we have 
\[
\ghat(S^k(T))\leq \ghat(T)+k. 
\]
Moreover, if $T$ is the closure of a quasipositive braid, then 
\[-g_s(T)+k\leq \ghat(S^k(T)).
\]
\end{prop}

\begin{cor}
A complete list of transverse unknots in $\Sst$ is $U_k=S^k(U)$ for $k\geq 0$. We know that
\[
 \ghat(U_k)= -\left(\frac{\slk(U_k)+1}{2}\right)= k,
\]
and the hat degree is $1$.\hfill\qed
\end{cor}

\begin{proof}[Proof of Proposition~\ref{ulbounds}]
It has long been known \cite{OrevkovShevchishin03} that if a transverse knot $T$ is realized as the closure of a braid $\beta$, which it always can be \cite{Bennequin83}, then the closure of a negative Markov stabilization of $\beta$ is the transverse stabilization of $T$ and the closure of the positive Markov stabilization of $\beta$ is transversely isotopic to $T$. Thus we see that there is a transverse regular homotopy from $S(T)$ to $T$ with one positive crossing change. Thus from Lemma~\ref{immersedcob} we see there is an immersed symplectic cobordism $C$ from $S^k(T)$ to $T$ with $k$ positive double points and a symplectic cobordism $C'$ from $T$ to $S^k(T)$ with $k$ negative double points. 

Now given a hat $S$ for $T$ in$\Xom$ we can compose this with $C$ and resolve the double points as in Lemma~\ref{resolve} to construct a hat for $S^k(T)$ with genus $\ghat(T)+k$, thus establishing the first inequality.

Let us now suppose that $T$ is quasipositive.
Let $S'$ be a minimal genus hat for $S^k(T)$ and $F$ be a symplectic filling of $T$ in $(B^4,\omega_{\rm std})$. We can glue $S'$, $C'$, and $F$ together to get an immersed symplectic surface $\Sigma$ in $(\CP,\omega_{\rm FS})$ of genus $g=g_s(T)+\ghat(S^k(T))$ with $k$ negative double points. Assume the homology class of $\Sigma$ is $ah$, where $h$ is the generator of $H_2(\CP)$ on which the symplectic form is positive. Lemma~\ref{MWequality} now yields 
\[
3a = \langle c_1(\CP), [\Sigma]\rangle = 2-2g(\Sigma)+\Sigma\cdot\Sigma+2m = 2-2g+a^2+2k,
\]
from which $2g-2k = a^2-3a+2 = (a-1)(a-2) \ge 0$. That is, $\ghat(S^k(T)) \ge k-g_s(T)$.
\end{proof}

Propositions~\ref{ulbounds} and~\ref{lb} lead to the following natural question. 
\begin{quest}
Is the function $\ghat(S^k(T)):\mathbb{N}\to \mathbb{N}:k\to \ghat(S^k(T))$ non-decreasing? From Lemma~\ref{lb} this is equivalent to asking: can the hat degree drop when a transverse knot is stabilized?
\end{quest}

\subsection{Further examples}\label{furtherex}
We begin with a strengthening of Theorem~\ref{negbraidnew} that shows for some transverse knots the bound in the estimate in Proposition~\ref{lb} is sharp. 

\begin{thm}\label{negbraid}
Suppose $T$ is a transverse knot in $\Sst$ and there is a transverse regular homotopy to an unknot with only positive crossing changes. Then the hat degree is $1$ and hence the hat genus is
\[
\ghat(T)=-\left(\frac{\slk(T)+1}{2}\right),
\] 
where $\slk(T)$ is the self-linking number of $T$. 
\end{thm}

We notice that this allows us to compute the hat genus for many knots.
\begin{cor}
If $T$ is the closure of a negative braid (that is a product of non-positive powers of the generators of the braid group), then
\[
\ghat(T)=-\left(\frac{\slk(T)+1}{2}\right)
\] 
and the hat degree is $1$.
\end{cor}
\begin{proof}
Changing negative powers in a braid word to positive powers corresponds to a transverse regular homotopy of the braid closures with positive crossing changes. Clearly by changing a subset of the letters in the braid word representing $T$ one arrives at the unknot. 
\end{proof}
A special case of the previous corollary is the following computation.
\begin{cor}
Suppose $q>p\geq 2$. Then any transverse representative $T$ of the $(p,-q)$-torus knot $T_{p,-q}$ satisfies 
\[
\ghat(T)=-\left(\frac{\slk(T)+1}{2}\right).
\] 
In particular, the maximal self-linking number representative $T'$ (which has $\slk(T')=-pq+q-p$) has 
\[
\ghat(T_{-p,q})=\frac{(q-1)(p+1)}{2}
\]
and hat degree $1$.\hfill\qed
\end{cor}

We note one further corollary of Theorem~\ref{negbraid}.
\begin{cor}
Let $T_n$ be a twist knot with maximal self-linking number. The hat genus of $T_n$ is
\[
\ghat(T_n)=\begin{cases}
1 & \text{$n\leq -3$ and odd}\\
\frac{n+3}{2}& \text{$n\geq 1$ and odd}\\
\frac n2& \text{$n$ positive and even.}\\
\end{cases}
\]
The hat degree is $1$ in all these cases. 
\end{cor}
\begin{proof}
The proof is similar to the ones above given the classification of Legendrian and transverse twist knots in \cite{EtnyreNgVertesi13}.
\end{proof}

\begin{proof}[Proof of Theorem~\ref{negbraid}]
We begin by noticing that if $T'$ is obtained from $T$ by a regular homotopy through transverse knots with a single positive double point then
\[
\slk(T')=\slk(T)+2.
\]
One may see this through a relative Euler characteristic argument, but as we are only considering knots in $S^3$ there is a simpler argument. Specifically, notice that we can remove a point $p$ from $S^3$ and obtain a contactomorphism of $S^3\setminus\{p\}$ to $\R^3$ with its standard contact structure taking another point $q$ in $S^3$ to the origin in $\R^3$. Now we can find a contactomorphism from a neighborhood of the double point in the regular homotopy to a ball in $S^3$ about $q$  with the double point going to $q$. This contactomorphism can be extended to all of $S^3$ and so we can assume our double point occurs at $q$. We can moreover assume the homotopy misses $p$ so that the entire homotopy occurs in $\R^3$ and that near the double point the two strands of the knot are both oriented in the positive $z$ direction. Now as we know the self-linking number of a transverse knot in the standard contact structure on $\R^3$ can be computed as the writhe of its front projection, see \cite{Etnyre05}, it is clear that the change in self-linking numbers is as claimed. 

Now given a transverse knot $T$ as in the statement of the theorem, denote its self-linking number $\slk(T)=-2n-1$. By hypothesis there is a regular homotopy with, say, $k$ positive crossing changes to an unknot $U'$. From the discussion above we see $\slk(U')=-2(n-k)-1<0$. (So we see the self-linking number of $T$ must be negative.) We can represent $U'$ as the closure of the braid $\sigma_1^{-1}\cdots \sigma_{n-k}^{-1}$. So $n-k$ more positive crossing changes will result in an unknot $U$ with $\slk(U)=-1$. Thus there is a regular homotopy from $T$ to the maximal self-linking number unknot $U$ with $n$ positive crossing changes. 

Now applying Lemma~\ref{immersedcob} we can find an immersed concordance from $T$ to $U$ with $n$ positive double points and glue this to a genus-0, degree-1 hat or $U$ to get an immersed genus-0 hat for $T$ in $\Xom$. Lemma~\ref{resolve} now yield an embedded genus-$n$ hat for $T$ with degree 1. So $\ghat(T)\leq n$. 

We now show that $\ghat(T)\geq n$. To this end let $\widehat{\Sigma}$ be a hat for $T$ with minimal genus. From Lemma~\ref{immersedcob} we can get an immersed concordance from $U$ to $T$ with $n$ negative double points. Gluing together $\widehat{\Sigma}$ the concordance and the slice disk for $U$ we construct an immersed symplectic surface $\Sigma$ in $\CP$ of genus $g(\widehat{\Sigma})$ having $n$ negative double points. If the homology class of $\Sigma$ is $ah$ then applying in immersed adjunction equality, Lemma~\ref{MWequality}, we see
\[
3a=\langle c_1(\CP), [\Sigma]\rangle = 2-2g(\widehat{\Sigma})+a^2+2n.
\]
So $g(\widehat{\Sigma})=\frac{(d-1)(d-2)}{2} + n$, or more specifically $g(\widehat{\Sigma})\geq n$ as claimed.
\end{proof}

Recall, in Question~\ref{mainquest} we asked if a slice, quasipositive transverse knot $T$ in $\Sst$ that has $\ghat(T)=0$, must be the maximal self-linking unknot.
While we cannot answer Question~\ref{mainquest}, we sketch an approach that seems promising.

\begin{proof}[Approach to Question~\ref{mainquest}]
Given such a transverse knot $T$, let $\widecheck{\Sigma}$ and $\widehat{\Sigma}$ be the filling and hat for $T$.
The symplectic surface $\Sigma=\widecheck{\Sigma}\cup\widehat{\Sigma}$ in $\CP$ has genus zero and thus, by the adjunction formula, $\widehat{\Sigma}$ has degree 1 or 2. 
Choose an almost complex structure $J$ such that $\Sigma$ is $J$--holomorphic.

Assuming that $T$ is a non-trivial knot we derive a contradiction. We begin by showing that $T$ is symplectically concordant to the unknot.

Suppose that $T$ wears a hat of degree $1$.
There exists an almost complex line $\ell_1$ lying entirely inside the cap; it intersects $\widehat{\Sigma}$ positively,
hence it intersects $\widehat{\Sigma}$ transversely once, inside the cap.
Removing a neighborhood of $\ell_1$ from the cap, we obtain a $J$--holomorphic concordance from $T$ to the link at infinity of $\widehat{\Sigma}$; since $\widehat{\Sigma}$ intersects $\ell_1$ transversely, the link is the unknot $U$.

Similarly, suppose that $T$ wears a hat of degree $2$.
There exists an almost complex line $\ell_2$ lying entirely inside the cap, which is tangent to $\widehat{\Sigma}$.
By positivity of intersections, the tangent point is the only intersection.
Removing a neighborhood of $\ell_2$ from the cap, we obtain a $J$--holomorphic concordance between $T$ and the link at infinity of $\widehat{\Sigma}$; since $\widehat{\Sigma}$ has an order one tangency to $\ell_2$, the link is the unknot.

In either case, we get a $J$--holomorphic concordance from $T$ to the unknot in $S^3\times [a,b]$. Now if one can deform $J$, keeping the concordance $J$--holomorphic, so that the standard height function on $S^3\times [a,b]$ is pluri-subharmonic, then the maximum principle implies that the concordance is ribbon (i.e. the restriction of the projection map $S^3\times [a,b]\to [a,b]$ has no maxima).

However, this contradicts a result of Gordon~\cite{Gordon}: since $T$ is slice and quasipositive, there is a ribbon concordance from $U$ to $T$; the argument above produces a ribbon concordance in the other direction. Finally, since the unknot has fundamental group $\Z$, $U$ is in particular transfinitely nilpotent, and~\cite[Theorem~1.2]{Gordon} implies that $T$ is isotopic to $U$.
(See also~\cite{CornwellNgSivek16}*{Theorem~3.2}: while their statement is in terms of decomposable Lagrangian concordances, the proof is entirely topological and applies more generally to ribbon concordances.)
\end{proof}

We note that Question~\ref{mainquest} would also follow from the arguments above together with a positive answer to the following stronger question. 
\begin{quest}
If there is a relative symplectic cobordism of genus $0$ from a transverse knot $T$ to a maximal self-linking transverse unknot $U$ in a piece of the symplectization of $\Sst$, then is $T$ transversely isotopic to $U$?
\end{quest}

\section{Hats in other manifolds}\label{s:hatsrational}
In this section we will show that a transverse knots always bounds a symplectic disk in some cap for the contact manifolds and then we will consider hats in rational surfaces. 

We begin by establishing some notation for rational surfaces caps. More specifically, let $(X_0,\omega_0) = \CP\setminus B^4$, where $B^4$ is embedded as a Darboux ball with convex boundary, as in Section~\ref{s:projhat}. We will write $X_n$ to denote an $n$--fold symplectic blow-up of $X_0$, so that $X_n$ is homeomorphic to $(\CP\#n\CPbar)\setminus B^4$. These are all caps for $\Sst$ and we will consider hats for transverse knots in these caps. 

There is another type of rational surface. We will denote it by $Y_0 = (S^2\times S^2, \omega_0)$ the Hirzebruch surface $\CPI\times\CPI$, endowed with its natural K\"ahler structure, and with $Y_1 = (\CP\#\CPbar, \omega_1)$ the blow-up of $\CP$, endowed with a blow-up symplectic structure.

Notice that there are infinitely many Hirzebruch surface up to complex diffeomorphism; these are all K\"ahler, and any two such complex surfaces are symplectomorphic (after possibly deforming their symplectic form) if and only if they are diffeomorphic.
Since the only two diffeomorphism classes of the underlying manifolds are $S^2\times S^2$ and $\CP\#\CPbar$, there are only two symplectic Hirzebruch surfaces, $Y_0$ and $Y_1$ as defined above.

A \emph{Hirzebruch cap} $(H_e,\omega_e)$ of $\Sst$ is obtained by removing a ball $B^4$ with convex boundary from the Hirzebruch surface $Y_e$, for $e=0$ or $e=1$. By abuse of notation, we will index Hirzebruch caps cyclically modulo 2, i.e.~$H_k = H_0$ whenever $k$ is even, and $H_k = H_1$ whenever $k$ is odd.

We notice that there is some overlap in our notation. Specifically $\Xom=(X_0,\omega_0)$ and $(X_1,\omega_1)=(H_1, \omega_1)$.

In the second subsection we will consider hats, which we call {\em rational hats}, in the non-minimal rational caps $X_n$ and in the following subsection we will consider hats, which we call {\em Hirzebruch hats}, in the Hirzebruch caps.

\subsection{Disk hats}

The goal of this section is to prove Theorem~\ref{t:diskhats2} from the introduction; recall that the theorem asserts that a transverse link $L$ in arbitrary contact 3--manifold $\Yxi$ has a hat that is a collection of pairwise disjoint disks in \emph{some} cap $\Xom$ of $\Yxi$.

\begin{proof}[Proof of Theorem~\ref{t:diskhats2}]
We begin by recalling a result of Gay,~\cite[Theorem~1.1]{Gay02a}. The theorem states that if $T$ is a transverse link in the convex boundary of a symplectic manifold $(W,\omega_W)$, then one may attach 2--handles to $T$ with sufficiently negative framing and extend the symplectic structure so that the upper boundary is weakly convex. It is clear from the proof that the core disk of the model symplectic 2--handle is symplectic, and that if $T$ bounded a symplectic manifold in $W$ then this core will cap off the surface symplectically. 

In our situation we consider the symplectic manifold $W = [0,1]\times Y$, which is a piece of the symplectization of $(Y,\xi)$. Inside $W$ we have the symplectic annuli $A=[0,1]\times L$. We can now attach Gay's symplectic 2--handles to $L$ in $\{1\}\times Y$ and cap off the upper boundaries of $A$ with symplectic disks. Thus we get a symplectic manifold $(W',\omega')$ that has $Y$ as a concave boundary, symplectic disks forming a hat for $L$, and a weakly convex boundary $Y'$. We can now cap off $Y'$ using~\cite{Eliashberg04, Etnyre04a}, to obtain a cap $\Xom$ containing the required hat. 
\end{proof}

\begin{rmk}
Theorem~\ref{t:diskhats2} can also be proven, in a very similar fashion, by adding Weinstein handles to a Legendrian approximation of $L$, and then perturbing the symplectic structure on the cobordism. We used Gay's symplectic handles since they give a more direct proof.
\end{rmk}

\subsection{Hats in non-minimal rational surfaces}

We start by proving Theorem~\ref{t:diskhats1}, which asserts that every transverse knot $K$ in $\Sst$ has a disk hat in some blow-up $X_N$ of the projective cap $X_0$.
This is a refinement of Theorem~\ref{t:diskhats2} for knots in $\Sst$.

\begin{proof}[Proof of Theorem~\ref{t:diskhats1}]
By Lemma~\ref{immersedannulus} there is a symplectic concordance from $K$ to $T_{p,q}$ with positive double points.

Without loss of generality, we assume $p<q$. Consider the complex curve $C = V\!(x^q-y^pz^{q-p})$ in $\CP$. The curve $C$ has two singular points, namely $(0:1:0)$ and $(0:0:1)$, whose links are the torus knots $T_{p-q,q}$ and $T_{p,q}$ respectively. Removing a neighborhood of the $T_{p,q}$ singular point from $\CP$ will result in a \emph{singular} hat for $T_{p,q}$ with genus-0, whose singularity is of type $T_{q-p,q}$.

Thus $K$ wears a singular projective hat with genus-0 and positive double points. Notice that there is a regular homotopy from the transverse unknot to $T_{q-p,q}$ with only positive crossing changes. Thus there is a concordance from the unknot to $T_{q-p,q}$. We can use this concordance and a symplectic slicing disk for the unknot to replace a neighborhood of the $T_{q-p,q}$ singularity with an immersed disk. The result is an immersed genus-0 hat for $K$ in $(X_0,\omega_0)$. 

Now instead of replacing the singularities with genus, we now resolve the singularities by blowing up: positive double points can be resolved using blow-ups \cite{MW}. Thus $K$ wears a genus-0 hat in some $X_N$. 
\end{proof}

\begin{defn}
Let $\ghat_n(K)$ define the smallest genus of a hat for $K$ in $X_n$. We call this the {\em $n^\text{th}$ rational hat genus} of $K$.
\end{defn}

An immediate corollary of the above proposition is the following observation. 
\begin{cor}
For each $K$, the sequence $\Ghat(K) = \{\ghat_n(K)\}_n$ is non-increasing and eventually 0.\hfill \qed
\end{cor}

\begin{defn}
We let $\shat(K) = \min\{n \mid \ghat_n(K) = 0\}$, and we call this the \emph{hat slicing number} of $K$.
\end{defn}

\begin{ex}
Suppose $K$ is a transverse unknot with $\slk(K) = -1-2s <-1$. We claim that $\shat(K) = 1$.
In fact, as noted above, $\ghat(K)>0$, hence $\shat(K)\ge 1$; moreover, $K$ is the closure of the $s+1$-braid $(\sigma_1\dots\sigma_s)\inv$, and a single full twist takes it to the braid $(\sigma_1\dots\sigma_s)^s$, whose closure is the transverse representative of $T_{s,s+1}$ with maximal self-linking number, which has been shown in Lemma~\ref{l:torushats} to have a symplectic disk hat.
It follows that $\shat(K)\le 1$.
\end{ex}

\begin{prop}\label{smalltorusknots}
For the torus knots $T_{2,2k+1}$ we know $\widehat{s}(T_{2,2k+1})\leq 1$. In particular, the sequence $\Ghat(T_{2,2k+1})$ is 
\[
(\ghat(T_{2,2k+1}), 0, \ldots)
\]
where the values of $\ghat(T_{2,2k+1})$ for $k\leq 11$ are
\begin{table}[htp]
\begin{center}
\begin{tabular}{c| c c c c c c c c c c c}
$k$ 	& 1 	& 2 	& 3	& 4	& 5	& 6	& 7	& 8	& 9 & 10  & 11\\
\hline
$\ghat(T_{2,2k+1})$ 	&0	&1	&0	&2	&1	&0	&3	&2	&  1 & 5 & 4
\end{tabular}
\end{center}
\end{table}%
\end{prop}

\begin{proof}
We first observe that $\ghat_1(T_{2,2k+1}) = 0$ for each $k\ge 0$: indeed, by~\cite[Theorem~1.1]{Fenske}, there is a degree-$(k+2)$ curve in $\CP$ with two singularities, one of type $T_{k,k+1}$ and one of type $T_{2,2k+1}$, and blowing up at the former (as discussed at the end of Section~\ref{planecurves}) yields a genus-0 hat for $T_{2,2k+1}$. Thus  $\widehat{s}(T_{2,2k+1})\leq 1$.

For the computation of $\ghat(T_{2,2k+1})$ we begin by noting that from \cite{55letters} we have 
\[
\ghat(T_{2,3}) = \ghat(T_{2,7}) = \ghat(T_{2,13}) = 0.
\]
Notice that the $4$--ball genus $g_s(T_{2,2k+1})=\frac{(2-1)(2k+1-1)}2 = k$. Thus by Corollary~\ref{p:next-triangular} we see that a lower bound on $\ghat(T_{2,12k+1})$ is  $d-l$ where $k=\frac{d(d-1)}2 +l $ for $1\leq l\leq d$.

We are left to show that there are indeed hats of the appropriate genus. For $T_{2, 5}$ we know that there is a positive crossing change to get to $T_{2,7}$, so it has hat genus 1. Similarly for $T_{2, 11}$. For $T_{2, 9}$ notice that one can make two positive crossing changes to get to $T_{2,13}$  and hence its hat genus is 2. 

Following~\cite{Yang-sextics} (see also~\cite{ArtalBartolo-sextics}), there exists a degree-6, genus-1 curve $C_{19}$ 
with a singularity of type $T_{2,19}$ thus with our lower bound given above we have $\ghat(T_{2,19})=1$. Since there are one, respectively two, crossing changes from $T_{2, 17}$, respectively $T_{2,15}$, to $T_{2,19}$ we see that their hat genus is 1, respectively 2.

For $T_{2,23}$, we claim that there is a deformation from $T_{6,7}$ to $T_{2,23}$; we exhibit such a deformation by removing eight generators to a braid representative of $T_{6,7}$ to obtain one of $T_{2,23}$. Since $T_{6,7}$ has a genus-$0$ hat (coming from the curve $V\!(x^6z-y^7)$), $T_{2,23}$ has a genus-$4$, degree-$7$ hat.

To prove the claim, one checks that, in the braid group $B_6$ (see Lemma~\ref{l:T223toT67} below for details):
\[
(\sigma_1\cdots\sigma_5)^7 = (\sigma_1\cdots\sigma_5)^2\cdot (\sigma_1\sigma_3\sigma_2\sigma_3\sigma_4\sigma_5\sigma_1\sigma_3\sigma_2\sigma_3\sigma_3\sigma_4\sigma_5\sigma_1\sigma_3\sigma_2\sigma_3\sigma_4\sigma_3\sigma_5\sigma_1\sigma_2\sigma_3\sigma_4\sigma_5).
\]
The second factor on the right-hand side contains eight generators $\sigma_i$ with $i$ even; removing them, one reduces to the braid
\[
(\sigma_1\cdots\sigma_5)^2\cdot \sigma_1\sigma_3 \sigma_3 \sigma_5\sigma_1\sigma_3 \sigma_3\sigma_3 \sigma_5\sigma_1\sigma_3 \sigma_3 \sigma_3\sigma_5\sigma_1 \sigma_3 \sigma_5,
\]
whose closure is the transverse representative of $T_{2,23}$ with maximal self-linking number. (An easy way to see this is the following: the closure of the braid is a $(2,2h+1)$-cable of the unknot, viewed as the closure of the $3$--braid $\sigma_1\sigma_2\sigma_3$; moreover, the braid is positive and has self-linking number $45$.)

This proves that $\ghat(T_{2,23}) \le 4$, and Corollary~\ref{p:next-triangular} gives the lower bound $\ghat(T_{2,23}) \ge 4$.

The usual crossing-change argument shows that $\ghat(T_{2,21}) \le 5$. However,  $g_s(T_{2,21}) = 10$ is a triangular number, so tweaking the argument of Corollary~\ref{p:next-triangular}, in order to show that $\ghat(T_{2,21}) \ge 5$, it is enough to show that $\ghat(T_{2,21}) > 0$, or, equivalently, that the minimal hat degree of $T_{2,21}$ is strictly larger than $6$.

Suppose the contrary; then there would exist a symplectic rational curve $C'$ of degree $6$ with a singularity of type $T_{2,21}$. Let $N$ be a small neighborhood of the singularity and $C$ the result of replacing the singularity in $N$ with the symplectic surface $T_{2,21}$ bounds in $N$. Since $C$ is degree 6 and symplectic isotopy problem is true in degree 6~\cite{Shev} (see also~\cite{SiebertTian}) we know $C$ is isotopic to a complex curve. Now it is well known \cite[Corollary~7.3.25]{GompfStipsicz99}, that the cover of $\CP$ branched over $C$ is a K3 surface. But inside of this K3 surface we see an embedding of the cover of the ball $N$ branched over the symplectic surface that $T_{2,21}$ bounds in $N$. This is a plumbing $P$ of 20 $(-2)$--spheres, \cite[Section~7.2]{GompfStipsicz99}. However, $b_2^-(P) = 20 > 19 = b_2^-(X)$ and so any hat for $T_{2,21}$ has degree larger than $6$.
\end{proof}

\begin{rmk}
There are other possible arguments to conclude that $\ghat(T_{2,21}) > 0$: one can either use the semigroup obstruction of Borodzik and Livingston~\cite{BorodzikLivingston}, or the Levine--Tristram signature obstruction of Borodzik and N\'emethi~\cite{BorodzikNemethi}. The argument above is very close to that of~\cite[Proposition~7.13]{GStarkston}.
\end{rmk}

\begin{rmk}
In the proof of Proposition~\ref{smalltorusknots} we saw that if $k$ is of the form $\frac{d(d-1)}2 +l $ for $1\leq l\leq d$ (that is, $k$ is larger than the triangular number $\frac{d(d-1)}2$ and less than or equal to $\frac{d(d+1)}2$) then $\ghat(T_{2,2k+1})\geq d-l$.
We actually have equality for $k\leq 9$. However, when $k=10$ this gives a lower bound of 0 on the hat genus, and we see from the table above that the actual hat genus of $T_{2,21}$ is 5, but for $k=11$ our lower bound is again accurate.

In fact, either using positivity of intersections (which gives the almost-complex counterpart of B\'ezout's inequalities in the complex setup) or using topological techniques (Heegaard Floer correction terms or Levine--Tristram signatures), one can show that this lower bound is eventually not sharp.
\end{rmk}

It takes some work to find examples where the hat slicing number is larger than 1, but we noted their existence in Proposition~\ref{goodslicenumberex} which we repeat here for the readers convenience. 

\begin{proof}[\bf Proposition~\ref{goodslicenumberex}]{\em 
Let $K_p$ be the unique transverse representative of $T_{p.p+1}\# T_{2,3}$ with maximal self-linking number which is $p^2-p+1 = 2g(T_{p,p+1}\#T_{2,3}) - 1$.  We have 
\[
\widehat{s}(K_p) = p-1,
\]
for $p\leq 8$. Moreover, 
\[
\Ghat(K_p) = (p-1,p-2,\dots,2,1,0,\dots),
\]
for $p\leq 4$. \renewcommand{\qedsymbol}{}}
\renewcommand{\qedsymbol}{}
\end{proof}

Is it always true that $\shat(K_p) = p-1$? and that $\Ghat(K_p) = (p-1,p-2,\dots,2,1,0,\dots)$?
It is easy to see, modifying the proof of the above proposition, that 
\[
\Ghat(K_p) \prec (p-1,p-2,\dots,2,1,0,\dots)
\] 
(so that, in particular, $\shat(K_p) \le p-1$) and that $\ghat(K_p) = p-1$.
It is clear however where the proof of the previous proposition ceases to work: as soon as $N$ is large (presumably $N \ge 8$ is already large in this sense), we have too many coefficients $b_i$ to allow for the same kind of bounds. (See the proof below for notation.)

Our proof of Proposition~\ref{goodslicenumberex} requires a special case of~\cite[Proposition 3.18]{GStarkston2}. We provide an alternative proof (of the special case) below.

\begin{prop}\label{p:highself}
A closed symplectic $4$--manifold cannot contain a rational cuspidal curve
of self-intersection strictly larger than $p^2+9$ whose singularities are of type $T_{p,p+1}$ and $T_{2,3}$.
\end{prop}

\begin{proof}
Suppose by contradiction that such a curve $C_0$ exists; in particular, $C_0\cdot C_0 \ge p^2 + 10$.
Blow up at the singularity of type $T_{p,p+1}$ and look at the proper transform $C$ of $C_0$.
The effect of the blow-up is to smooth the singularity: in fact, as observed at the end of Section~\ref{planecurves}, in the blow-up, $C$ has a singularity of type $T_{p,p+1-p} = T_{1,p}$, i.e. the unknot. That is, the singularity is resolved by a single blow-up.
Therefore, $C$ has a unique singularity left, which is of type $T_{2,3}$; moreover, $C\cdot C = C_0\cdot C_0 - p^2 \ge  10$.
However, this contradicts a result of Ohta and Ono~\cite{OhtaOno}, which asserts that no pseudo-holomorphic rational curve with a simple cusp (i.e. a singularity of type $T_{2,3}$) in a closed symplectic $4$--manifold has self-intersection larger than 9. (Note that we can make $C$ $J$--holomorphic with respect to some almost complex structure $J$ by blowing up the singularity of $C_0$ and applying results of McDuff \cite{McDuff92} to the total transform.)
\end{proof}

We will also need the following lemma.
\begin{lemma}\label{ratcurve}
In $\CP$ there is a symplectic sphere $C^k_{p+2}$ of degree $(p+2)$ that has $(p-k)$ positive double points and two singularities of type $T_{p,p+1}$ and $T_{2,2k+1}$, for $0\leq k\leq p$. We will denote the the $k=1$ sphere by $C'_p$. 
\end{lemma}

\begin{proof} From~\cite[Theorem~1.1, case 8]{Fenske}, for each  $p$ there exists a degree-$(p+2)$ rational curve $C_{p+2}$ in $\CP$ with singularities of type $T_{p,p+1}$ and $T_{2,2p+1}$.
For each $0\le k \le p$, the latter can be deformed to a singularity of type $T_{2,2k+1}$ and $p-k$ ordinary double points.
Here by ``deformed'' we mean that we can modify the curve in an arbitrarily small neighborhood of the singularity, replacing it with a curve with the ``smaller'' singularities described. To see this notice that one may go from $T_{2,2k+1}$ to $T_{2, 2p+1}$ by $(p-k)$ positive crossing changes. Thus by Lemma~\ref{immersedcob} there is a symplectic cobordism in $[a,b]\times S^3$ from $T_{2,2k+1}$ to $T_{2,2p+1}$ with $(p-k)$ double points. We can now excise a neighborhood of the singularity of type $T_{2,2p+1}$, glue our constructed cobordism in its place using Lemma~\ref{glue}, and finally glue in a new symplectic ball and the cone on $T_{2,2k+1}$. 
The case $k=1$ gives the claimed symplectic curve $C'_{p+2}$. (With more work this construction can be done in the algebraic category yielding a complex curve with the stated properties.)

An alternate construction of $C_{p+2}$ can also be given as follows. 
Look at the (reducible) curve $V\!(x^pz^2 - y^{p+1}z)$; it consists of a rational curve $R$ with a cusp of type $T_{p,p+1}$ and a line $L$ with a tangency of order $p+1$ to $R$ at a smooth point of $R$.
The link of the tangency point is of type $T_{2,2p+2}$. As above we can replace a neighborhood of this point by a pair of pants with a singular point of type $T_{2,2p+1}$ since we can build a symplectic cobordism from $T_{2,2p+1}$ to $T_{2,2p+2}$ by adding a single positive crossing using Lemma~\ref{addposcrossing}.
The resulting curve $C$ has degree $p+2$ (since deformations don't change the degree), is irreducible (since we connected the two irreducible components with the deformation), and has genus 0 (either by adjunction or by an Euler characteristic computation).
\end{proof}

\begin{proof}[Proof of Proposition~\ref{goodslicenumberex}]
We begin by noticing that blowing up $\CP$ at each of the $(p-1)$ double points of $C'_{p}$ from Lemma~\ref{ratcurve} gives an embedded sphere $C$ with two singular points of type $T_{p,p+1}$ and $T_{2,3}$. Removing a neighborhood containing the two singular points shows that $\ghat_{p-1}(K_p)=0$. Thus to see $\widehat{s}(K_p) = p-1$ we merely need to see $\ghat_{p-2}(K_p)\not=0$.

For $p=1$ we are done. For $p=2$ we notice that the genus of $K_2$ is $2$, and since the minimal triangular number larger than $2$ is $3$, Corollary~\ref{p:next-triangular} give a lower bound of $1$ for $\ghat(K_2)$. To see that $\ghat(K_2)=1$ we notice that resolving the double point of $C'_2$ gives a genus one surface with two singular points of type $T_{2,3}$ and removing a neighborhood containing the singular points gives the desired surface.  

We will now compute whole sequence $\Ghat(K_p)$ for $p=3$ and $4$. Then for $K_5, \dots, K_8$ we use the same techniques as for $K_4$, but we only achieve the computation of the hat slicing number. It is possible that, pushing the arguments a little bit further, one can compute $\Ghat(K_p)$ for some other small value of $p$, but we do not pursue this here.

\noindent
{\bf The knot $K_3$:} We show that $\Ghat(K_3) = (2,1,0, \ldots)$.

Since the slice genus of $K_3$ is $4$ and the minimal triangular number larger than this is $6$, we again have from Corollary~\ref{p:next-triangular} that $\ghat(K) \ge 2$; by looking at $C_3'$ and smoothing its two double points, we create a curve in $\CP$ with genus 2 and whose singularities are the cusps of $C'_3$, i.e. one of type $T_{3,4}$ and one of type $T_{2,3}$. It follows that $\ghat(K) \le 2$.

To compute $\ghat_1(K_3)$, look again at $C'_3$. Blowing up at one of its double points and resolving the other, we construct a genus-1 hat in $\CP\#\CPbar$, thus proving that $\ghat_1(K_3) \le 1$.

We are left to prove that $\ghat_1(K_3) > 0$.
Suppose by contradiction that $K_3$ has a genus-0 hat $H$ in punctured $X_1$; filling $(S^3,\xi_{\rm std}, K_3)$ with the cone filling, we obtain a  symplectic rational cuspidal curve $C$ in the sense of~\cite{GStarkston}.
Suppose that the homology class of $C$ is $ah-be \in H_2(X_1)$, where $h$ and $e$ are the homology classes of a line and of the exceptional divisor.
By positivity of intersections, either $a=0$ or $a,b\ge 0$: this follows as in~\cite[Lemma~4.5]{Lisca-fillings}, once we observe that the class $h$ can always be realized by a smooth $J$--holomorphic $+1$--sphere for each $J$, since the space of $J$--holomorphic spheres in the class $h$ has positive dimension and is non-empty (because $h$ satisfies automatic transversality and it is represented by a symplectic sphere)~\cite{Gromov} (see also~\cite[Section~2]{rationalruled}).
Applying the adjunction formula, we obtain:
\begin{equation}\label{e:adjunctionX1}
\frac12 (a-1)(a-2) - \frac 12 b(b-1) = g(T_{3,4}\#T_{2,3}) = 4
\end{equation}
So $a\not=0$ and we have $a,b\ge 0$. Since the only way to express $4$ as a difference of triangular numbers is as $4 = 10-6$, we obtain that $(a,b) = (6,4)$, and in particular $C\cdot C = 6^2 - 4^2 = 20$. However, $C\cdot C = 20 > 18 = 3\cdot 3 + 9$, thus contradicting Proposition~\ref{p:highself} above.

\noindent
{\bf The knot $K_4$:} We show that $\Ghat(K_4) = (3, 2,1,0, \ldots)$.

We begin with a lemma that follows from positivity of intersections and Gromov's work on $J$--holomorphic lines and conics in $\CP$.
\begin{lemma}\label{l:lineineq}
Suppose that $K_p$ has a hat in $X_n$ in the homology class $ah-\sum b_ie_i$, with 
\[
 b_1 \ge \dots \ge b_n \ge 0.
\] 
Then (assuming $n$ is large enough for each inequality to make sense) we have: 
\begin{align*}
a&\ge b_1 + p\\
a &\ge b_1+b_2\\
2a &\ge b_1 + b_2 + b_3 + b_4 + p, \text{ and}\\
2a &\ge b_1 + b_2 + b_3 + b_4 + b_5.
\end{align*}
\end{lemma}
\begin{proof}
Let $C$ be the singular curve obtained by gluing the hat with a singular symplectic filling of $K_p$ with two singular points, one of type $T_{2,3}$ and the other of type $T_{p,p+1}$.
The homology class of $C$ is $ah-\sum b_ie_i$. (When proving the first two inequalities we assume $n=2$ for convenience; in general, the other expectational curves can be ignored.)

Realize the homology classes of $e_1$ and $e_2$ by symplectically embedded disjoint spheres $E_1, E_2$, and then choose an almost-complex structure $J_2$ on $X_2$ that makes $E_1$, $E_2$, and $C$ simultaneously $J_2$--holomorphic (this follows as in the case of $K_3$ above).
Contracting $E_1$ and $E_2$ we get to an almost-complex $\CP$ (with almost-complex structure $J$), and by work of Gromov there is a (unique) $J$--holomorphic line $L$ in $\CP$ passing through the contractions of $E_1$ and $E_2$, and its proper transform in $X_2$ is $J_2$--holomorphic, and thus intersects $C$ positively.
But the homology class of the proper transform $L'$ of $L$ is $h-e_1-e_2$, and its intersection with $C$ is precisely
\[
0 \le C\cdot L' = (ah-b_1e_1-b_2e_2)\cdot(h-e_1-e_2) = a-b_1-b_2.
\]
To prove that $a\ge b_1 + p$, just consider the line going through the singular point of $C$ of multiplicity $p$ (i.e. the point where $C$ is a cone over $T_{p,p+1}$) and the contraction of the exceptional divisor corresponding to $e_1$.

The second part of the statements is proved in analogous way by using conics instead of lines.
Indeed, instead of considering a line, consider the (unique) conic through the contractions of $e_1,\dots,e_4$ and the singularity of type $T_{p,p+1}$, or through the contractions of $e_1,\dots, e_5$, and then apply positivity of intersections.
\end{proof}

In the following, we find it convenient to re-write the adjunction formula for the homology class $ah-b_1e_1 - \dots - b_Ne_N \in H_2(X_N)$, represented as a rational cuspidal curve with two singularities of type $T_{p,p+1}$ and $T_{2,3}$ as:
\begin{equation}\label{e:newadjunction}
a^2 - b_1^2 - \dots - b_N^2 = p^2 - p + (3a - b_1 - \dots - b_N).
\end{equation}
An immediate corollary to Lemma~\ref{l:lineineq}, since the $b_i$ are decreasing, is that for each set of distinct indices $i_1,\dots, i_5$:
\begin{equation}\label{e:abbounds}
2a \ge b_{i_1} + \dots + b_{i_5}, \quad  a \ge b_{i_1} + b_{i_2},
\end{equation}
where by convenience we let $b_{N+1} = p$ and $b_{N+2} = 2$.
The fact that we can use $b_{N+2} = 2$ in the bounds can be seen, as in the proof of the lemma, by taking a curve through the singularities of type $T_{2,3}$ that lives in the singular filling of $K_p$. In particular, $a\ge p+2$.

Finally returning to $K_4$, we see, as above, that $\ghat(K_4) \ge 3$, since $g(T_{4,5}\#T_{2,3}) = 7$ and the smallest triangular number larger than $7$ is $10$;
by either blowing up or resolving the double points of $C'_4$ as we did for $C'_3$ above, we easily see that $\ghat(K_4) \le 3$, and that in fact $\ghat_1(K_4) \le 2$, $\ghat_2(K_4) \le 1$, and $\ghat_3(K_4) = 0$.

Suppose that $\ghat_2(K_4) = 0$.
Using Equation~\eqref{e:abbounds}, we can write:
\[
3a-b_1-b_2 = (a-b_1) + (a-b_2) + a \ge 2p+a \ge 3p+2 = 14 = p+10,
\]
And thus Equation~\ref{e:newadjunction} gives
\[
a^2-b^2_1-b^2_2> p+10
\]
which contradicts Proposition~\ref{p:highself}. Since $\ghat_2(K) > 0$, then a fortiori $\ghat_1(K) > 0$.

Suppose that $\ghat_1(K) = 1$; and let $ah-be$ be the homology class of the corresponding curve $D$ in $\CP\#\CPbar$.
Applying the adjunction formula, $\frac 12 (a-1)(a-2) - \frac 12 b(b-1) = 8$ (note that the right-hand side is $8$ instead of $7$---as it was above---since now the curve has genus $1$); by direct inspection, the only solution to the equation is $(a,b)=(10,8)$.
However, this contradicts the Lemma~\ref{l:lineineq} above, since $10 = a \not\ge b+p = 8+4 = 12$. Therefore, $\ghat_1(K) > 1$.

\noindent
{\bf The knot $K_5, K_6,$ and $K_7$:} Here we simply establish that $\ghat_{p-2}(K_p)\not=0$ for $5\leq p\leq 7$. 
We also have the inequality $\ghat_{p-2}(K_p) \le 1$ from resolving one double point of $C'_p$ and blowing up the rest as in the examples above. Thus we actually prove $\ghat_{p-2}(K_p) = 1$.

Suppose $\ghat_{p-2}(K_p) = 0$, then we produce a rational cuspidal curve $C_p$ as we did in the $K_3$ case; suppose that $[C_p] = ah - b_1e_1 - \dots - b_{p-2}e_{p-2} \in H_2(X_{p-2})$.

In light of the Adjunction Formula~\eqref{e:newadjunction}, if we want to apply Proposition~\ref{p:highself} as in the examples above, it is enough to prove that $(3a - b_1 - \dots - b_N)  \ge p+10$. We will do this for $p = 5, 6, 7, 8$ and $N = p-2$, 

For convenience of notation, from now on we drop the subscripts.

\smallskip
\noindent
{\bf For $p=5$:} By the Inequalities~\eqref{e:abbounds}, we obtain:
\[
3a-b_1-b_2-b_3 = (a-b_1) + (a-b_2) + (a-b_3) \ge 3p = p + 10.
\]

\smallskip
\noindent
{\bf For $p=6$:}  By the Inequalities~\eqref{e:abbounds}, we can write:
\[
3a-b_1-\dots-b_4 = a + (2a-b_1-\dots-b_4) \ge a + p = a+6,
\]
so if $a \ge 10$ we have $3a-b_1-\dots-b_4 \ge p+10$. There are obviously no solutions to Equation~\ref{e:newadjunction} if $a\leq 7$, this is also true for $a=8$. For $a=9$ the only solution is $(a,b_1,\dots,b_4) = (9,3,\dots,3)$. This curve has self-intersection $45$.

By looking at the classification of fillings of the corresponding contact structure on the boundary of the neighbourhood of the curve as in~\cite{GStarkston}, one sees that this contact structure does not have any strong fillings. The obstruction reduces to the not being able to embed $\mathcal{G}_4$ in $\CP$, \cite[Proposition~5.24]{GStarkston}, where $\mathcal{G}_4$ is two conics with an order-$4$ tangency and a line tangent to both.

This can be seen as follows. Given the hypothesized $C_6$, blow up at the two singularities of the curve, thus obtaining a configuration of three smoothly embedded symplectic spheres: the proper transform $C$ of the curve, of self-intersection $45-6^2-2^2 = 5$, and two $(-1)$--curves (the exceptional divisors), having tangency orders $6$ and $2$ with $C$.
Now blow up four more times at the intersection point of $C$ with the first exceptional divisor: the proper transform of $C$ is a $+1$--sphere, and the resulting configuration of curves is shown in Figure~\ref{f:K6}.

We apply McDuff's theorem~\cite{McDuff} to identify the proper transform of $C$ with a line in a blow-up of $\CP$. Using Lisca's arguments~\cite{Lisca-fillings}, the homological embedding of the configuration is forced, up to permutation of the indices: the embedding is displayed in Figure~\ref{f:K6}. Contracting the divisors $e_1,e_6,\dots,e_9$, and then $e_5$, $e_4$, $e_3$, $e_2$ (in this order) reduces our configuration to one containing $\mathcal{G}_4$ (the curve in the homology class $h-e_1-e_2$ can be disregarded).
To this end, it is clear that the two curves in the homology classes $2h-\sum e_i$ blow down to two conics; blowing down $e_5$ creates a transverse self-intersection between the two blown-down curves; $e_4$ passes through the point of intersection, and blowing it down creates a simple tangency. Contracting $e_3$ and $e_2$ in the same fashion creates a tangency of order $4$, which will be the only intersection point of the two conics.

\begin{figure}
\labellist
\pinlabel $\phantom{x}_{h}$ at 10 17
\pinlabel $\phantom{x}_{2h-e_1-\dots-e_5}$ at 4 55
\pinlabel $\phantom{x}_{e_1}$ at 110 121
\pinlabel $\phantom{x}_{2h-e_2-\dots-e_9}$ at 290 40
\pinlabel $\phantom{x}_{h-e_1-e_2}$ at 260 5
\pinlabel $\phantom{x}_{e_2-e_3}$ at 215 108
\pinlabel $\phantom{x}_{e_3-e_4}$ at 281 120
\pinlabel $\phantom{x}_{e_5}$ at 110 134
\pinlabel $\phantom{x}_{e_4-e_5}$ at 371 134
\pinlabel $\phantom{x}_{e_6}$ at 312 88
\pinlabel $\phantom{x}_{e_7}$ at 327 81
\pinlabel $\phantom{x}_{e_8}$ at 348 73
\pinlabel $\phantom{x}_{e_9}$ at 367 66
\endlabellist
\includegraphics[width=.75\textwidth]{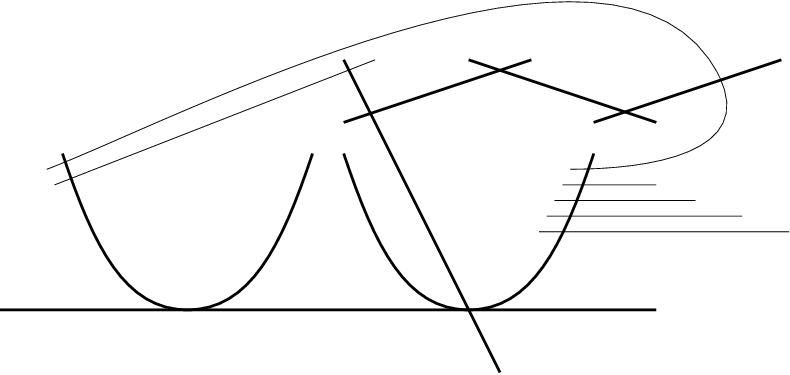}
\caption{The configuration associated to the curve of self-intersection $45$ with singularities of type $T_{2,3}$ (left) and $T_{6,7}$ (right), together with the embedding coming from McDuff's theorem. The bold curves represent the total transform of the curve described in the proof of the case $p=6$ in Proposition~\ref{goodslicenumberex}; the thin curves represent the exceptional divisors in the embedding.}\label{f:K6}
\end{figure}
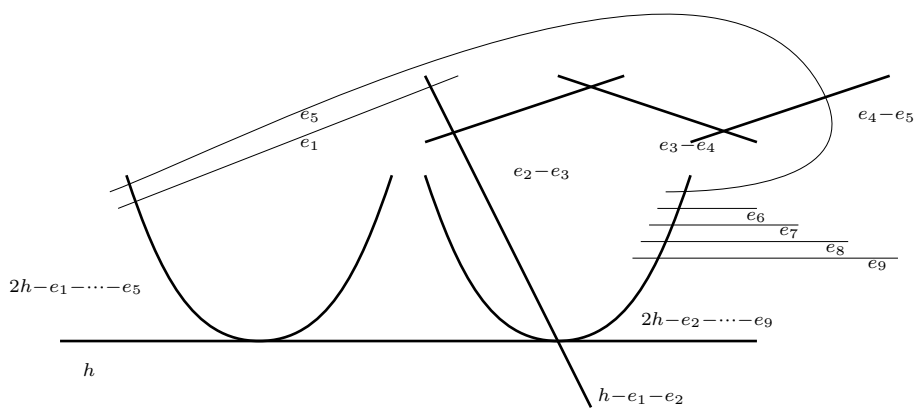

\smallskip
\noindent
{\bf For $p=7$:}  
From the Inequalities~\eqref{e:abbounds}, we obtain:
\[
3a-b_1-\dots-b_5 = a + (2a-b_1-\dots-b_4-b_5) \ge a.
\]
If $a \ge 17 = p+10$, we are done. It is clear there are no solutions to Equation~\ref{e:newadjunction} if $a\leq 8$.
A computer search now shows that there are no solutions to the adjunction formula for which $a$ between 9 and 16 either.
\end{proof}

\subsection{Hats in Hirzebruch caps}\label{Hhats}

We now investigate the {Hirzebruch cap} $(H_e,\omega_e)$ of $\Sst$ discussed at the beginning of this section. We call $F_e$ a $\CPI$--fiber of $H_e$, and $S_e$ and $S_e'$ the two sections with self-intersection $+e$ and $-e$ respectively; with a small abuse of notation, we use the same notation for the homology classes (either in $H_2(H_e)$ or in $H_2(H_e,\partial H_e))$.

\begin{prop}\label{p:hirzebruch}
For each $p\ge 2$ and $k\ge 1$ the knot $T_{p,kp + 1}$ admits a genus-0 hat in $(H_k,\omega_k)$, in the relative homology class $F_k+pS_k$.

Analogously, for each $p\ge 2$ and $k\ge 1$ the knot $T_{p,kp - 1}$ admits a genus-0 hat in $(H_k,\omega_k)$, in the relative homology class $pS_k$.
\end{prop}

\begin{proof}
We start by proving that the knots $T_{p,kp+1}$ have genus-0 hats in $H_k$.

Let $C$ denote the curve $C = V(x^{p+1} + y^{p}z)\subset \CP$. The curve $C$ has a singularity with link $T_{p,p+1}$ at the point $(0:0:1)$, is smooth away from $(0:0:1)$, and is rational.
That is, the complement of a small ball $B^4$ centered at $(0:0:1)$ is a disk, which is complex and hence symplectic with respect to the K\"ahler structure on $\CP$.
Consider the line $\ell = \{x=0\}\subset\CP$. It has two intersections with $C$: $(0:0:1)$, with multiplicity $p$, and $(0:1:0)$, with multiplicity $1$.
By blowing up at $(0:1:0)$, we obtain an embedded rational curve $C_1$, the proper transform of $C$, whose only singularity is of type $T_{p,p+1}$.

Observe that the proper transform of $\ell_1$ is a fiber $F_1$ of $X_1$ over $\CPI$, and that it intersects $C_1$ only at the singular point $x_1$, and it does so with multiplicity $p$. The exceptional divisor $E_1$ of the blow-up is a $(-1)$--section $S'_1$ of $X_1$, and it intersects the curve $C_1$ transversely at one point.

We now proceed by induction; suppose that, as in Figure~\ref{f:Ck}, we have created a curve $C_k$ in $H_k$ such that:
\begin{itemize}
\item $C_k$ has only one singularity of type $T_{p,kp+1}$;
\item $C_k$ intersects a fibre $F_k$ only at the singularity of $C_k$ with multiplicity $p$;
\item $C_k$ intersects a section $S'_k$ transversely at one point.
\end{itemize}
Now blow up at the intersection of $F_k$ and $S'_k$, creating the exceptional divisor $E_k$, and blow down the proper transform of $F_k$.
The singularity of $C_{k+1}$ at $x_{k+1}$ has gained a $p$ in its multiplicity sequence, hence its link is $T_{p,(k+1)p+1}$, as desired.
The curve $E_k$ blows down to a fiber $F_{k+1}$ of $H_{k+1}$, that intersects the image $C_{k+1}$ only at its singularity with multiplicity $p$.
Finally, the proper transfom of $S'_k$ is $S'_{k+1}$, and the contraction happens away from $S'_{k+1}$, hence $S'_{k+1}$ still intersects $C_{k+1}$ transversely at one point.

\begin{figure}
\includegraphics[width = 0.7\textwidth]{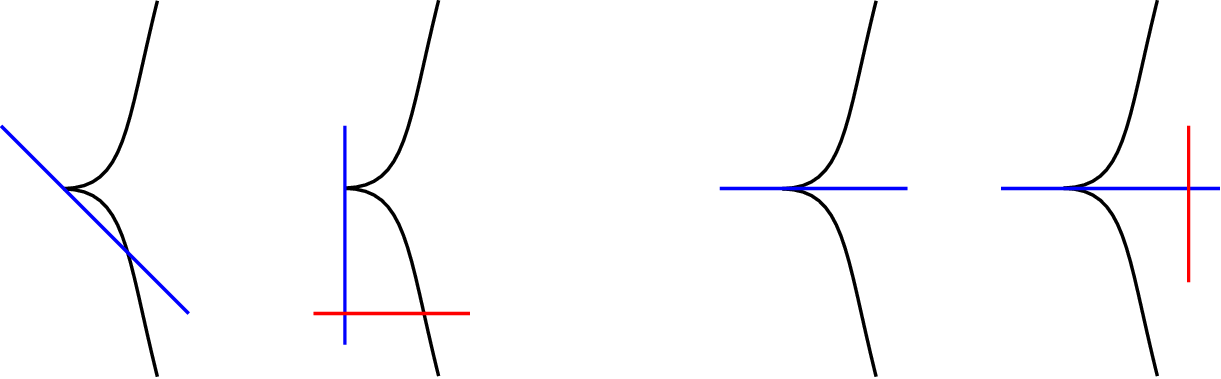}
\caption{From left to right: the curve $C$ (black) and the line $\ell$ (blue); the curve $C_k$ (black), the fiber $F_k$ (blue), and the section $S'_k$ (red); the curve $C'$ (black) and the line $\ell'$ (blue); the curve $C'_k$ (black), the fiber $F_k$ (blue), and the section $S'_k$ (red). In the second and fourth figure, the point at which we blow-up is the intersection of the red and blue curves, and the curve to be contracted is the proper transform of the blue curve.}\label{f:Ck}
\end{figure}

This proves the existence of the genus-0 hat; we now compute its relative homology class.
Indeed, with respect to the intersection pairing, the two bases $(F_k, S_k)$ and $(S'_k, F_k)$ of $H_2(H_k)$ are dual bases.
Since $C_k$ intersects $F_k$ with multiplicity $p$ and $S'_k$ with multiplicity $1$, its homology class is therefore $F_k+pS_k$; the corresponding hat, obtained by removing a small ball around the singularity, is in the relative homology class $F_k+pS_k$.

This concludes the proof in the case $T_{p,kp+1}$. The proof in the case $T_{p,kp-1}$ is very similar, and we only outline the differences here.

Instead of considering the curve $C$, we consider the curve $C' = V(x^{p} + y^{p-1}z)\subset \CP$, and instead of the line $\ell$ we consider $\ell' = \{y=0\}$.
The line $\ell'$ and the curve $C'$ intersect only once, with multiplicity $p$.
Blowing up at a generic point of $\ell'$ yields the starting point for the induction, as above.

However, now the fibre $F_1$ only intersects the curve once, with multiplicity $p$. We blow up once at a generic point of $F_1$ and blow down the proper transform: therefore, the section $S'_1$ (i.e. the exceptional divisor) is disjoint from the curve.
By doing so, we obtain a curve $C_1'$ whose only singularity picks up a new $p$ in the multiplicity, and the fibre $F_2'$ again intersects the curve once with multiplicity $p$; this allows to run the induction similarly as above.
In particular, we get curves $C_k'$ whose only singularity is of type $T_{p,kp-1}$.
Now $C_k'$ is disjoint from the section $S'_k$ and intersects the fiber $F_k$ with multiplicity $p$, so its homology class (as well as the relative homology class of the cap) is $pS_k$.
\end{proof}

\begin{ex}
Observe that, by contrast, some of these knots have very large hat genus. We will focus on the case $k=2$.
Consider the knot $T_{p,2p-1}$ first; the first three elements of the semigroup of the associated singularity are $0$, $p$, $2p-1$, hence, in the notation of \cite{BCG}, $\Gamma(3) = 2p-1$.
Using~\cite[Equation ($\star_j$), p. 15]{BCG} with $j=1$, one obtains $2p-1 = \Gamma(3) \le d$, that is the degree is at least $2p-1$.
In particular, $\ghat(T_{p,2p-1}) \ge \frac{(d-1)(d-2)}2 - g(T_{p,2p-1}) = (p-1)(p-2)$; this actually proves that the hat of Lemma~\ref{l:torushats} is the one of minimal degree for these knots, proving that, in fact, $\ghat(T_{p,2p-1}) = (p-1)(p-2)$.

We now exhibit a symplectic curve of degree $2p-1$ and genus $(p-1)(p-2)$ in $\CP$ whose only singularity is a cone over $T_{p,2p-1}$.
Indeed, the curve $V(x^{2p-1} + y^pz^{p-1})$ has two singularities, one of which is of type $T_{p,2p-1}$.
Smoothing the other singularity yields the desired symplectic curve.

Similarly, for the knot $T_{p,2p+1}$ we have that the third element of the semigroup is $2p$, yielding $d\ge2p$ as above.
In particular, $\ghat(T_{p,2p+1}) = \frac{(d-1)(d-2)}2 - g(T_{p,2p+1}) \ge (p-1)^2$.
However, by smoothing one of the two singularities of the curve $V(x^{2p+1} + y^pz^{p+1})$ we obtain a hat of degree $2p+1$, hence showing that
$\ghat(T_{p,2p+1}) \le p^2$.
Since there are no triangular numbers strictly between $g(T_{p,2p+1})+(p-1)^2$ and $g(T_{p,2p+1})+(p-1)^2$, we have $\ghat(T_{p,2p+1})\in\{(p-1)^2, p^2\}$.
For example, for $p=2$ there is a unicuspidal symplectic curve of degree $4$ whose singularity is of type $T_{2,5}$; likewise, for $p=3$ there is a unicuspidal symplectic curve of degree $6$ whose singularity is of type $T_{3,7}$.
\end{ex}

\section{A higher-dimensional example}\label{s:brieskorn}

We prove here Proposition~\ref{p:Brieskornhats}, which is a higher-dimensional analogue of Lemma~\ref{l:torushats}.
That is, we will prove that the link of an isolated complete intersection singularity, which is a contact submanifold of $(S^{2n-1}, \xi_{\rm std})$, has a hat in $\CPn \setminus B^n$.

The main ingredients of the theorem are two:
\begin{itemize}
\item the finite determinacy theorem for singularities of complete intersections, asserting that any singularity is determined by an appropriate truncation of its Taylor series (see~\cite{AdamusPatel, SrinivasTrivedi});
\item a deformation $\{(X_t,0)\}_{|t| \le 1}$ of an isolated complete intersection singularity $(X_0,0)$ induces a symplectic cobordism from the link of $(X_\epsilon,0)$ to that of $(X_0,0)$.
\end{itemize}

Let us be more precise on how a deformation gives rise to a symplectic cobordism.

Let $\Delta$ be the unit disk in $\C$ centered at $0$ and $\Delta^* = \Delta\setminus\{0\}$. A deformation of the singularity in $\C^n$ defined by the equations $G_1 = \dots = G_{n-d} = 0$ is a $1$--parameter family $\{G_1^t, \dots, G_{n-d}^t\}_{t\in \Delta}$ of power series such that $G_k^0 = G_k$. Suppose that the germs $\{G_1^t = \dots = G_{n-d}^t = 0\}_{t \in \Delta^*}$ have topologically isomorphic singularities at $O \in \C^n$ and that the corresponding subvarieties are all complete intersections. In this case, let $X_t$ denote (the germ at $O \in \C^n$ of) $\{G_1^t = \dots = G_{n-d}^t = 0\}$; with a small abuse of notation, we will also call $X_t$ the singularity of $X_t$ at the origin.
We say that $X_t$ for \emph{any} $t\neq 0$ is the \emph{generic singularity} in the family $\{X_t\}$, and that $X_0$ is the \emph{central singularity}.

Let $\epsilon > 0$ be a real number such that $X_0$ intersects $S^{2n-1}_\epsilon \subset \C^n$ transversely. We choose $\epsilon$ sufficiently small so that this intersection is the link of $X_0$.
For $t$ sufficiently small, $X_t$, too, intersects $S^{2n-1}_\epsilon$ transversely. Now choose $\eta$ sufficiently small such that $X_t$ intersects $S^{2n-1}_\eta$ transversely in the link of $X_t$. Then $X_t \cap (D^{2n}_\epsilon \setminus D^{2n}_\eta)$ is a symplectic cobordism from the link of $X_t$ to the link of $X_0$.

We split off an easy lemma. We put coordinates $\{(z_0 : \dots : z_n)\}$ on $\CPn$.

\begin{lemma}\label{l:simplecompleteintersection}
Let $1 \le d < n$ and $p$ be positive integers, and $A = (a_{i,j})_{1 \le i \le n-d, 1 \le j \le n}$ be a complex matrix of size $(n-d)\times n$ whose $(n-d)\times (n-d)$ minors are all non-zero. Then the equations $\sum_{j = 1}^n a_{1,j}z_j^p = \dots = \sum_{j = 1}^{n} a_{n-d,j}z_j^p =  0$ define a complete intersection $X_{A,p} \subset \CPn$ that has an isolated singularity at $(1:0:\dots:0)$ and is non-singular elsewhere.
In particular, the link of its singularity has a projective hat.
\end{lemma}

Note that none of the equations above involves $z_0$.

\begin{proof}
We compute the Jacobian of the map $\C^{n+1} \to \C^{n-d}$ sending
\[
(z_0, \dots, z_n) \mapsto \Big(\sum_{j = 1}^n a_{1,j}z_j^p, \dots, \sum_{j = 1}^{n} a_{n-d,j}z_j^p\Big).
\]
This is simply given by $B(z_0,\dots,z_n) = (b_{i,j}(z_0,\dots,z_n))_{1 \le i \le n-d, 0 \le j \le n}$, where
\[
b_{i,j}(z_0,\dots,z_n) = \left\{
\begin{array}{ll}
0 & \mbox{if $j = 0$}\\
pa_{i,j}z_j^{p-1} & \mbox{if $j > 0$}
\end{array}
\right.
\]
In order to check that $X_{A,p}$ is a complete intersection whose unique singularity is at $(1:0:\dots:0)$, it is enough to show that $B(z_0,\dots,z_n)$ has maximal rank along $X_{A,p}$, except when $z_1 = \dots = z_n = 0$.

This is easy to see, since the matrix $B(z_0,\dots,z_n)$ is obtained from $A$ by adding a column and multiplying each column by $pz_j^{p-1}$, so there is a non-vanishing minor as soon as at least $n-d$ among $z_1, \dots, z_n$ are non-vanishing. On the other hand, if at least $d+1$ among them are vanishing, then all of them vanish, since the equations defining $X_{A,p}$ are linear in $z_1^p, \dots, z_n^p$ and $A$ has all non-zero maximal minors.

The complement of a small ball around $(1:0:\cdots:0)$ gives the desired projective hat.
\end{proof}

\begin{proof}[Proof of Proposition~\ref{p:Brieskornhats}]
Suppose $(Y^{2d-1},\xi) \subset (S^{2n-1},\xi_{\rm std})$ is the link of an isolated complete intersection singularity $(X,0)$ in $\C^n$, defined by analytic functions $\tilde f_1, \dots, \tilde f_{n-d}$. By the finite determinacy theorem, there exists a constant $C$ such that $(X,0)$ is isotopic to the singularity defined by the truncations of the Taylor series of $\tilde f_1, \dots, \tilde f_{n-d}$ at degree $C$.

Choose $p = C+1$ and $A = (a_{i,j})$ to be a complex $(n-d)\times n$ matrix with non-zero maximal minors. For $i=1,\dots, n-d$ let
\[
g_i(z_1,\dots,z_n) = a_{i,1}z_1^p + \dots + a_{i,n}z_n^p
\]
and, for a complex number $t$,
\[
F^t_i = g_i + tf_i.
\]
In particular, notice that the the truncation of $F^t_i$ at degree $C$ is $tf_i$.
It follows that, when $t \neq 0$, the singularity of $\{F^t_1 = \dots = F^t_{n-d} = 0\}$ at the origin is isomorphic to $(X,0)$. On the other hand, when $t = 0$, then we have the complete intersection $X_{A,p}$ of Lemma~\ref{l:simplecompleteintersection}, and the link $(Y_{A,p},\xi_{A,p})$ of its singularity has a projective hat.

The family $\{\{F_1^t = \dots = F_{n-d}^t = 0\}\}_t$ describes a deformation of $\{g_1 = \dots = g_{n-d} = 0\}$, so there exists a symplectic cobordism in $S^{2n-1} \times [0,1]$ from $(Y,\xi)$ to $(Y_{A,p},\xi_{A,p})$, and gluing this to the projective hat of $(Y_{A,p},\xi_{A,p})$ we conclude the proof.
\end{proof}

\section{Applications of hats to fillings}\label{s:fillings}

In this section we will construct hats for some quasipositive knots, and we will see how these hats can be used to produce caps for their branched covers. In turn, we will use these caps to restrict the topology of exact fillings of these branched covers.

Recall that given a transverse link $K$ in a contact manifold $(Y,\xi)$ there is a natural contact structure $\xi_K$ induced on any cover of $Y$ branched over $K$ obtained by pulling back $\xi$ on the complement of $K$ and extending over the branched locus in a natural way,~\cite{Gonzalo87}.  For a transverse knot $K$ in $\Sst$ we denote by $\Sigma_r(K)$ the contact manifold obtained by $r$--fold cyclic branched cover of $\Sst$ branched over $K$. When $r=2$ we will leave off the subscript and just write $\Sigma(K)$. 

In what follows, we will only be dealing with quasipositive knot types. For each such knot type $K$ we will choose a specific quasipositive braid, which gives a \emph{specific} transverse representative $T$ smoothly isotopic to $K$.
This choice endows the branched cover $\Sigma_r(K)$ with a contact structure obtained as the $r$--fold cyclic cover of $\Sst$ branched over $T$. By an abuse of notation, we will denote it with $\xi_{K,r}$ instead of $\xi_{T,r}$; as above, if $r = 2$ we drop it from the notation and simply write $\xi_K$.

\begin{rmk}
We do not have examples for which our statements are sensitive to the choice of the transverse isotopy class $T$ in the smooth knot type $K$. {More generally, we are not aware of any examples in the literature of two transverse knots $T, T'$ with the same classical invariants, and such that $\xi_T$ is not contactomorphic to $\xi_{T'}$.

However, we do not see any reason why the statement should hold for arbitrary transverse representatives. More precisely, our proof of Theorem~\ref{mainfilling} will break down if, in one of the knot types $K$ of the statement, one can find another transverse representative $T'$ with $\hat{d}(T') > 6$ and $\xi_{T'} \neq \xi_T$, where $T$ is the transverse knot of topological type $K$ considered in Theorem~\ref{mainfilling}. The proof of Theorem~\ref{highercovers} would also break down for similar phenomena.}
\end{rmk}

In this section we prove Theorems~\ref{mainfilling} and~\ref{highercovers} which we recall here for the reader's convenience.

\begin{proof}[\bf Theorem~\ref{mainfilling}]{\em
Let $K \subset (S^3,\xi_{\rm std})$ be one of the transverse knots in Table~\ref{table:braids}. Let $(W,\omega)$ be an exact symplectic filling of $(\Sigma(K),\xi_K)$, with intersection form $Q_W$.
\begin{enumerate}
\item If $K$ is of type $12n_{242}$, then $W$ is spin, $H_1(W) = 0$, and $Q_W = E_8 \oplus H$.
\item If $K$ is of type $10_{124}$, $12n_{292}$, or $12n_{473}$, then $W$ is spin, $H_1(W) = 0$, and $Q_W = E_8$.
\item If $K$ is of type $m(12n_{121})$, then $W$ is spin, $H_1(W) = 0$, and $Q_W = H$.
\item If $K$ is of type $m(12n_{318})$, then $W$ is an integral homology ball.
\item If $K$ is of any of the following topological types, then $W$ is a rational homology ball:
\[
\begin{array}{lllll}
m(8_{20}), & m(9_{46}), &10_{140}, &m(10_{155}), &m(11n_{50}),\\
m(11n_{132}), &11n_{139}, &m(11n_{172}), &m(12n_{145}), &m(12n_{393}),\\
12n_{582}, &12n_{708}, &m(12n_{721}), &m(12n_{768}), &12n_{838}.
\end{array}
\]
\end{enumerate}}
\renewcommand{\qedsymbol}{}
\end{proof}

\begin{proof}[\bf Theorem~\ref{highercovers}]
\em{
Let $(\Sigma_r(K), \xi_{K,r})$ denote the $r$--fold cyclic cover of $\Sst$, branched over the transverse knot $K$ of Table~\ref{table:braids}. Let $(W,\omega)$ be an exact filling of $(\Sigma_r(K), \xi_{K,r})$.
\begin{enumerate}
\item If $K$ is a quasipositive braid closure of knot type $m(8_{20})$, $m(9_{46})$, $10_{140}$, $m(10_{155})$, $m(11{n_{50}})$, and $r=3,4$, then $W$ is a spin rational homology ball.
\item If $K$ is a quasipositive braid closure of knot type  $m(11{n_{132}})$, $11{n_{139}}$, $m(11{n_{172}})$, $m(12{n_{318}})$, $12{n_{708}}$, $m(12{n_{838}})$ and $r=3$, then $W$ is a spin rational homology ball.
\item If $K$ is a quasipositive braid closure of knot type $8_{21}$ and $r=3,4$, then $W$ is spin and $b_2(W) = 2(r-1)$.
\end{enumerate}}\renewcommand{\qedsymbol}{}
\end{proof}

\subsection{The pretzel knot $P(-2,3,7)$}\label{ss:Sigmafillings}

In this section we prove Part~\eqref{item1} of Theorem~\ref{mainfilling} and so we focus on $K = P(-2,3,7) = 12n_{242}$; this is a quasipositive knot with determinant 1, whose branched double cover is $\Sigma(K)=-\Sigma(2,3,7)$ \cite{Montesinos73}, i.e. a Brieskorn sphere with its orientation reversed.
This case will be paradigmatic for the other examples considered later.

For convenience we will denote the standard generators of the braid group $B_3$ by $x$ and $y$. The knot $K$ is represented by the braid word $xy^2x^2y^7 \in B_3$. We also recall the notation $\beta \uparrow \beta'$ introduced just before Lemma~\ref{l:garside} to indicate the braid $\beta'$ is obtained from $\beta$ by adding the square of a generator. 

\begin{lemma}\label{l:pretzel-hat}
The knot $K$ has a genus-$5$, degree-$6$ hat $H$ in $\CP$.
\end{lemma}

\begin{proof}
We are going to exhibit a genus-5 symplectic cobordism $\Sigma$ from $K$ to $T_{3,11}$.
Since $T_{3,11}$ is the only cusp of the rational curve $V((zy-x^2)^3- xy^5)$~\cite{55letters}, it has a disk hat of genus $0$.
Gluing the cobordism and the latter hat, we obtain the desired result.

The symplectic cobordism $\Sigma$ is obtained by performing a sequence of positive crossing changes, isotopies and conjugations.
The starting point will be the braid $xy^2x^2y^7$, whose closure is a transverse representative of $K$ with self-linking number $9$, and the goal will be the braid $(xy)^{11}$, whose closure is the unique transverse representative of $T_{3,11}$ with self-linking number $19$.
Recall that in the $3$--braid group, we have the relation $xyx = yxy$, and that the closures of braids are insensitive to conjugation (that we are going to denote with $\sim$). We are also going to denote with $\Delta^2 = (xy)^3 = (xyx)^2$ the full twist, which lies in the center of $B_3$ (in fact, it generates it).
To ease readability, we also underline the point of the braid word where we have introduced a new crossing.

We start by observing the following fact: given any word $w\in B_3$, with two crossing changes we can turn $w_0 = wxy^{2n+1}$ into $w_1 = w\Delta^2xy^{2n-1}$.
In fact,
\[w_0 = wxy^{2n+1} \uparrow wxy\underline{xx}y^{2n} \uparrow wxyxxy\underline{xx}y^{2n-1} = w(xyx)^2xy^{2n-1} = w\Delta^2xy^{2n-1} = w_1.\]

We denote such an operation by $w_0\uparrow\uparrow w_1$.

The sequence goes as follows:
\begin{align*}
xy^2x^2y^7 &\uparrow xy^2x\underline{y^2}xy^7 = xy(yxy)yxy^7 \sim yxyyxyyxyy^5=xyxyxyxyxy^5\\
=&\Delta^2xyxy^5\uparrow\uparrow \Delta^2xy\Delta^2xy^3\uparrow\uparrow \Delta^4xy\Delta^2xy=\Delta^6xyxy=(xy)^{11}.
\qedhere
\end{align*}
\end{proof}

We now use this hat for $K$ to build a nice symplectic cap for $\Sigma(K)$. 
\begin{prop}\label{p:BDC-cap}
There is a symplectic cap $(C, \omega_C)$ for $\Sigma(K)$ that embeds in a symplectic K3 surface. 
Moreover, $H_1(C) = 0$, the intersection form of $C$ is $E_8\oplus 2H$, and the canonical divisor $K_C$ vanishes.
\end{prop}

\begin{proof}
Since $K$ is quasipositive, it bounds a symplectic surface $F$ of genus equal to the quasipositive genus of $K$, which in turn can be computed from the self-linking number of $K$ by the adjunction formula in Lemma~\ref{filladjunction}, see also~\cite{BoileauOrevkov}.

In this case, $g(K) = g_s(K) = 5$, hence $g(F) = 5$.
Glue $F$ and the hat $H$ from Lemma~\ref{l:pretzel-hat} together: this yields a smooth symplectic curve $D\subset\CP$ of the same degree as the degree of the hat; that is, $D$ has degree 6 and genus 10.

Since the symplectic isotopy problem is true in degree 6~\cite{Shev} (see also~\cite{SiebertTian}), $D$ is isotopic to a complex curve of degree $6$, and the branched double cover of $\CP$ branched over a smooth sextic is a K3 surface (see, e.g.~\cite[Corollary~7.3.25]{GompfStipsicz99}).

Let $(C, \omega_C)$ be the double cover of $\CP\setminus B^4$ branched over $H$ and $\Sigma(F)$ be the double cover of $B^4$ branched over $F$. We notice that $\Sst$ in $\CP$ has a neighborhood that looks like a piece $[a,b]\times S^3$ of the symplectization of $\Sst$ and $D$ intersects this neighborhood in $[a,b]\times K$. The branched covering construction of contact and symplectic manifolds shows that a piece of the symplectization of $\Sigma(K)$ lies above $[a,b]\times S^3$ in the cover and so $(C,\omega_C)$ is a cap for $\Sigma_K$ (and $\Sigma(F)$ is a filling). 

One may easily compute $b_2(\Sigma(F)) = 10$ (see, for instance,~\cite[Section~7]{GompfStipsicz99}); 
moreover, since $\Sigma(K)$ is an integral homology sphere, the intersection forms on $H_2(\Sigma(F))$ and $H_2(C)$ are both unimodular.
The intersection form on K3 is $2E_8 \oplus 3H$, and thus $b_2(X) = b_2({\rm K3})-b_2(\Sigma(F)) = 12$ and $b^+_2(X) \le 3$.
The only unimodular intersection form of rank $12$ and $b_2^+ \le 3$ is $E_8+2H$.

Finally, the canonical class $K_X$ is the restriction of $K_{\rm K3} = 0$ to $X$, hence it vanishes, too.
\end{proof}

With the Calabi--Yau cap $(C,\omega_C)$ in hand Theorem~\ref{mainfilling} Part~\eqref{item1} will follow from the following results.

\begin{prop}\label{p:SVHM}
Suppose that a contact rational homology $3$--sphere $(Y,\xi)$ has a Calabi--Yau cap $(C,\omega_C)$ with $b_2^+(C) \geq 2$ and $b_2(C) \ge 7$.
Then all exact symplectic fillings embed in a K3 surface, have finite first homology, and have the same Betti numbers and signature.
Moreover, if $Y$ is an integral homology sphere, then every filling has trivial first homology.
\end{prop}
The proof is essentially the proof of~\cite[Proposition~3.1]{SivekVanHornMorris}, {\em cf}~\cite[Theorem~1.3]{LiMakYasui}. 

\begin{proof}
Suppose $(W,\omega_W)$ is an exact symplectic filling of $(Y,\xi)$, and let 
\[
(X,\omega) = (C,\omega_C) \cup_{(Y,\xi)} (W,\omega_W).
\] 
Let also $K_X$ denote the canonical class of $X$.

Since $W$ is an exact filling and $C$ is a Calabi--Yau cap, it follows that
\[
K_X\cdot[\omega] = K_X|_W \cdot [\omega_W] + K_X|_C\cdot[\omega_C] = K_X|_W \cdot 0 + 0\cdot [\omega_C] = 0.
\]
The facts that $b_2^+(X) \ge 2$ and $K_X\cdot[\omega_X] = 0$ together imply that the only Seiberg--Witten basic classes for $X$ are $\pm K_X$~\cite{Taubes-k0}. Since $K_X$ is represented by a symplectic embedded surface~\cite{Taubes-K} and $K_X\cdot[\omega_X] = 0$, in fact $K_X = 0$, therefore $X$ is symplectically minimal~\cite{FS-blowup}.

Hence the symplectic Kodaira dimension of $X$ is $0$~\cite{Li-k0}, and therefore $X$ has the rational homology of either a K3 surface, or of an Enriques surface, or of a $T^2$-bundle over $T^2$~\cite{MorganSzabo, Bauer, Li-K3}.
However, $C$ cannot embed in a torus bundle $T$ over the torus, since $b_2(C) \ge 7 > 6 \ge b_2(T)$.
Neither can $C$ embed in an Enriques surface $E$: indeed,  $b_2^+(C) \ge2 > 1 = b_2^+(E)$.

Hence, $X$ is a rational homology K3, i.e. $|H_1(X)| = n < \infty$.
Consider the kernel of the Abelianisation map $\pi_1(X) \to H_1(X)$, and the cover $(\widetilde{X},\widetilde{\omega})$ associated to its kernel. As signature is multiplicative under finite covers we see $\sigma(\widetilde X) = -16n$, but since $\widetilde{X}$ is also a compact symplectic manifold of Kodaira dimension 0, its signature must be $0$, $-8$, or $-16$. Thus $n=1$ and we have  $H_1(X) = 0$.

Let us look at the Mayer--Vietoris long exact sequence for $X = W \cup_{Y} C$:
\[
H_1(Y) \to H_1(C) \oplus H_1(W) \to H_1(X) = 0.
\]
Since $Y$ is a rational homology sphere, $H_1(W)$ is finite. If $Y$ is an integral homology sphere, $H_1(W) = 0$.

Finally, since $Y$ is a rational homology sphere, the intersection forms of $W$ and of $C$ are non-degenerate, and their direct sum embeds as a full-rank sub-lattice of $H_2(X) \cong 2E_8 \oplus 3H$.
The statements on $b_2(W)$ and $\sigma(W)$ readily follow; an Euler characteristics argument implies that $b_3(W)$ is invariant, too.
\end{proof}

\begin{proof}[Proof of  Theorem~\ref{mainfilling} Part~\eqref{item1}]
The cap $(C,\omega_C)$ of Proposition~\ref{p:BDC-cap} is a Calabi--Yau cap, and it has $b_2^+(C) = 2$ and $b_2(C)=12\ge 7$. Therefore, by Proposition~\ref{p:SVHM}, all exact fillings of $\xi$ are spin and have the same Betti numbers and signature. In the proof of Proposition~\ref{p:BDC-cap} we saw a filling with $b_2=10$ and $\sigma=-8$.

Since $-\Sigma(2,3,7)$ is an integral homology sphere, the intersection form of any filling is unimodular; since the filling is spin, it is also even. In particular,  the intersection form is $E_8 \oplus H$.
\end{proof}

We now establish Remark~\ref{rmkabtfillings} by constructing infinitely many symplectic fillings of $-\Sigma(2,3,7)$. We begin by constructing one such filling. 

\begin{lemma}\label{l:E10}
The contact structure $\xi$ is filled by the plumbing of Lagrangian spheres according to the graph $E_{10}$ in Figure~\ref{e10graph}.
\end{lemma}
Note that, in fact, $E_{10}$, as a lattice, is isomorphic to $E_8\oplus H$, by classification of indefinite unimodular forms (or by direct inspection).

\begin{figure}[h]\label{e10graph}
\begin{center}
\begin{tikzpicture}[scale=1]
	\draw (0,0) -- (8,0);
	\draw (2,-1) -- (2,0);
	
	\draw[fill=black] (0,0) circle(.08);
	\draw[fill=black] (1,0) circle(.08);
	\draw[fill=black] (2,0) circle(.08);
	\draw[fill=black] (3,0) circle(.08);
	\draw[fill=black] (4,0) circle(.08);
	\draw[fill=black] (5,0) circle(.08);
	\draw[fill=black] (6,0) circle(.08);
	\draw[fill=black] (7,0) circle(.08);
	\draw[fill=black] (8,0) circle(.08);
	\draw[fill=black] (2,-1) circle(.08);
\end{tikzpicture}
\end{center}
	\caption{The graph $E_{10}$.}
\end{figure}
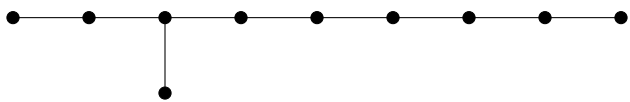

\begin{proof}[Proof (sketch)]
By lifting the monodromy of the disk open book for $(S^3, \xi_{\rm std})$, adapted to the $3$--braid $xy^2x^2y^7$, and converting to a contact surgery diagram as in~\cite{HKP}, we obtain the diagram on the left of Figure~\ref{f:E10}. Here $(+1)$--contact surgery is performed on the darker knots and $(-1)$--contact surgery is performed on the other knots. Since the darker knots are unlinked unknots, doing $(+1)$--surgery along them can also be viewed as attaching a 1--handle.

\begin{figure}
\includegraphics[width=\textwidth]{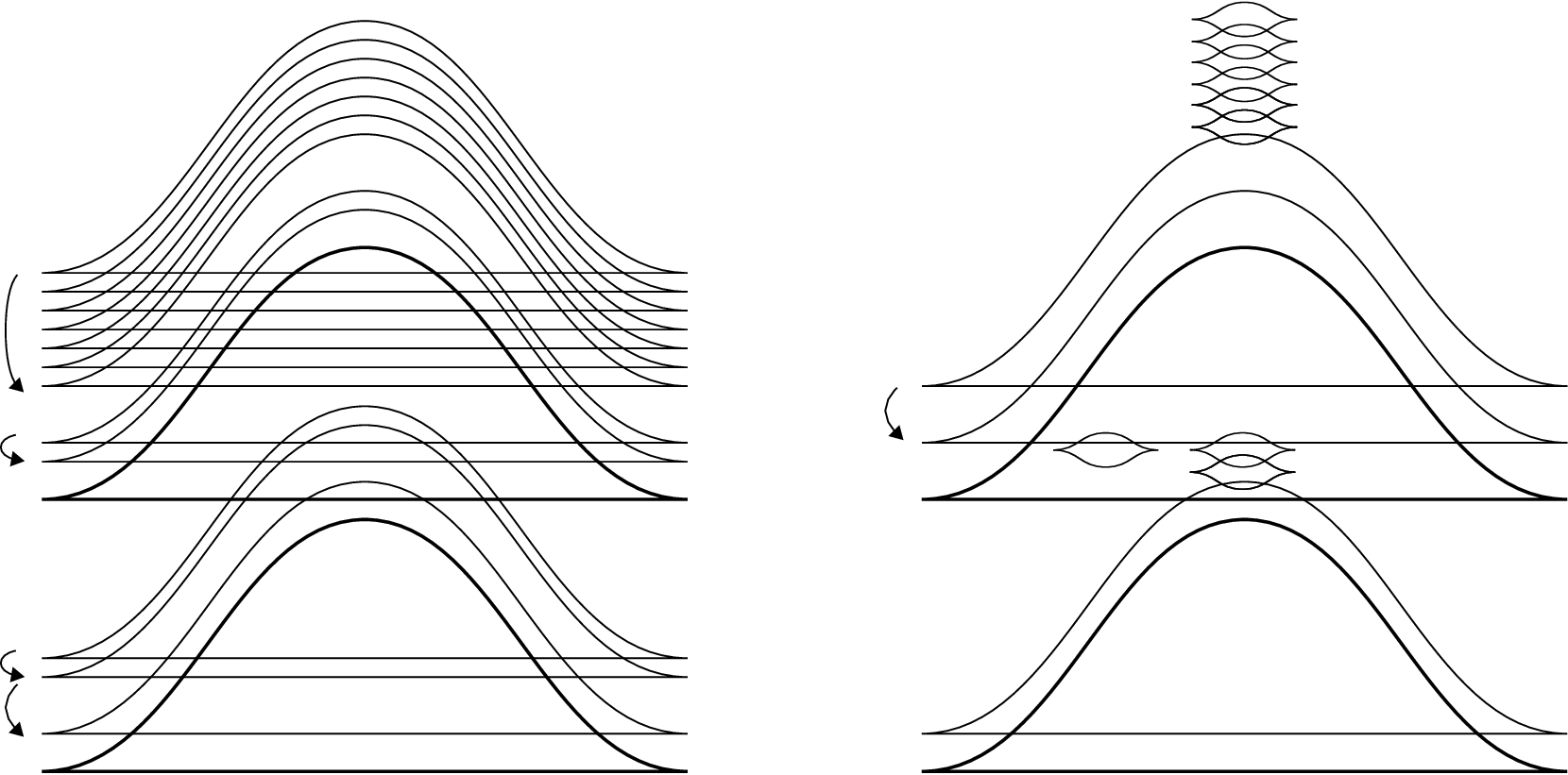}
\caption{Two contact surgery diagrams; the thicker components correspond to contact $+1$--surgery, the others to Legendrian surgery. The arrows indicate the handleslides described in the proof of Lemma~\ref{l:E10}.}\label{f:E10}
\end{figure}

By successively handlesliding~\cite{DingGeiges} the topmost unknot on the next one (as indicated by the long arrow on the top left), and performing the three handleslides indicated by the other three arrows, we obtain the diagram on the right.

We can now cancel the two bottommost knots, and perform a last handleslide as indicated by the arrow. The remaining contact $(+1)$--framed knot cancels with the remaining `big' $(-1)$--framed knot, leaving with the diagram comprising ten $\tb=-1$ unknots that link according to the $E_{10}$ graph.

This exhibits $\xi$ as the boundary of the $E_{10}$ plumbing of Lagrangian spheres, as required.
\end{proof}

\begin{proof}[Proof of Remark~\ref{rmkabtfillings}]
By Lemma~\ref{l:E10}, $\xi$ is the boundary of the plumbing $(P,\omega_P)$ of Lagrangian spheres, plumbed according to the $E_{10}$ graph.
We can deform the symplectic structure $\omega_P$ to make all spheres symplectic \cite{Gompf-new}.

Since the $E_{10}$ plumbing is not negative definite and its boundary is a homology sphere, it admits a family $\{N_\varepsilon\}$ of open neighborhoods with \emph{concave} boundary~\cite{LiMak}.

In particular, there is a symplectic structure on $\Sigma(K)\times[0,1]$ such that both boundary components are convex, and this is obtained by removing $N_\varepsilon$ from $P$ for some sufficiently small $\varepsilon$.
We can now cap off the component $\Sigma(K)\times\{0\}$ with caps with arbitrarily large $b_2^+$ \cite{EtnyreHonda02a}.
\end{proof}

\subsection{Quasipositive knots with few crossings}
In this section we will prove the other cases of Theorem~\ref{mainfilling} using the same technique as in the previous subsection. In particular we begin by finding quasipositive knots with degree-6 hats. 

\begin{lemma}\label{l:6-hats}
A quasipositive representative of each of the following knots has a degree-6 hat:

$m(8_{20})$, $m(9_{46})$, $10_{124}$, $10_{140}$, $m(10_{155})$, $m(11n_{50})$, $m(11n_{132})$, $11n_{139}$, $m(11n_{172})$, $m(12n_{121})$,

$m(12n_{145})$, $12n_{292}$, $m(12n_{318})$, $m(12n_{393})$, $12n_{473}$, $12n_{582}$, $12n_{708}$, $m(12n_{721})$, $m(12n_{768})$, $12n_{838}$.
\end{lemma}

We will prove the statement only for $m(8_{20})$, $m(9_{46})$, $10_{140}$, and $m(12n_{145})$, as a sample. The rest of the proof can be found in Appendix~\ref{a:braids}.

In the following, we denote with $x,y$ the generators of $B_3$ and with $x,y,z$ the generators of $B_4$. In what follows, we use $\uparrow_k$ to denote the insertion of $k$ pairs of crossings and we underline the new generators (or the generator that has been switched from negative to positive as well as the original negative crossing, or the negative full twist that we simplify),
$\sim$ to denote conjugation, and $\sim_D$ to denote Markov destabilisation. We also use $\Delta^2$ to denote the Garside element in $B_3$.

\begin{proof}
The knots are all quasipositive according to KnotInfo~\cite{knotinfo}. 
We argue case by case.

\noindent
{$\mathbf{m(8_{20})}$:} This is the closure of the $3$--braid $x^3yx^{-3}y$. We can write
\[
x^3y\underline{x^{-3}}y \uparrow_2 x^3y\underline{x}y = x^4yx \sim_D = x^5,
\]
and $T_{2,5}$ is the singularity at the point $(0:0:1)$ of the degree-5 curve $V(x^2z^3 - y^5)$, hence it has a degree-5 (and hence a degree-6 as well) hat.

\noindent
{$\mathbf{m(9_{46})}$:} This is the closure of the 4--braid $xy\inv xy\inv zyx\inv yz$. We we can write
\begin{align*}
x\underline{y\inv} xy\inv& zyx\inv yz \uparrow xyxy\inv zyx\inv yz = yxzyx\inv yz = yzxyx\inv yz\\ & = yzy\inv xy^2z = z\inv yzxy^2z \sim yzxy^2\sim zxy^3\sim x^3y \sim_D  x^3
\end{align*}
The closure of the latter is $T_{2,3}$ with its maximal self-linking number, hence we have produced a degree-3 (and hence degree-6) hat. 

\noindent
{$\mathbf{10_{140}}$:} This is the closure of the $4$--braid $x^{-3}yx^3yzy\inv z$. We have:
\begin{align*}
\underline{x^{-3}}yx^3yz\underline{y\inv} z &\uparrow_3 \underline{x}yx^3y(z\underline{y}z) = xyx^3y^2\underline{z}y \sim_D (xyx)x^2y^3\\ &\sim (xyx)y^5 = yxy^6 \sim_D = y^7,
\end{align*}
and the latter is the singularity of the degree-4 curve $V((zy - x^2)^2 - xy^3)$ at $(0:0:1)$ \cite{55letters}.

\noindent
{$\mathbf{m(12n_{318})}$:} This knot is the closure of the $4$--braid $xyzx\inv zy^{-2}xy\inv zyxy\inv$. We have:
\begin{align*}
xyzx\inv z\underline{y^{-2}}x&\underline{y\inv} zyx\underline{y\inv} \uparrow_3 xyzx\inv zx\underline{y} zyx\underline{y} = xyz^2yzyxy \sim yz^2yzyxyx \\
& = yz^2yzy^2\underline{x}y \sim_D yz^2yzy^3 \uparrow yz\underline{y^2}zyzy^3 = \Delta^2 zy^3 \sim \Delta^2 yzy^2 = (yz)^5,
\end{align*}
hence we have produced a cobordism from $m(12n_{318})$ to $T_{3,5}$, and $T_{3,5}$ is the singularity at the point $(1:0:0)$ of the degree-5 curve $V(x^2z^3 - y^5)$.\qedhere
\end{proof}

With care one could determine the genus of the hats constructed in the proof of Lemma~\ref{l:6-hats}, but notice that it is not necessary. If $K$ is any knot in the lemma it is quasipositive and so bounds a symplectic surface in $B^4$. That together with the hat $H$ for $K$ will give a symplectic surface $D$ in $\CP$ of degree 6, which has to have genus 10 by the Adjunction Equality. So the genus of the hat can be computed from the genus of the symplectic surface $K$ bounds in $B^4$ (or, equivalently, one can compute the self-linking of $K$ and use the equation in Lemma~\ref{filladjunction}). 

KnotInfo~\cite{knotinfo} tells us that the $4$--ball genus of 
 \begin{align*}
& m(8_{20}), m(9_{46}),  10_{140}, m(10_{155}), m(11n_{50}), m(11n_{132}), 11n_{139}, m(11n_{172}), \\
&m(12n_{145}), m(12n_{393}), 12n_{582}, 12n_{708}, m(12n_{721}), m(12n_{768}), \text{ and } 12n_{838}
 \end{align*}
is 0, the $4$--ball genus of
\[
10_{124}, 12n_{292}, \text{ and } 12n_{473}
\]
is $4$, and $m(12n_{121})$ has $4$--ball genus 1. 

Recall that the determinant of a knot is the order of the first homology group of its branched double cover; therefore, a knot has determinant $1$ if and only if its branched double cover is a homology sphere. In particular, the intersection form of any smooth $4$--manifold bounding the branched double cover is unimodular, if the knot has determinant $1$. The above knots with determinant 1 are
\[
10_{124} = T_{3,5}, m(12n_{121}), 12n_{292}, m(12n_{318}), \text{ and } 12n_{473}.
\]

\begin{proof}[Proof of Theorem~\ref{mainfilling}, Items~\ref{item2} through~\ref{item5}]
Each of these knots has a Calabi--Yau cap, obtained by taking the double cover of $\CP\setminus B^4$, branched over the hat of Lemma~\ref{l:6-hats}.

Knots in Items~\ref{item3} through~\ref{item4} have determinant 1, therefore the intersection form of their fillings is unimodular, and their rank is determined by the quasipositive genus of the knot.
For knots in Item~\ref{item2}, the cover of $B^4$ branched over the quasipositive surface the knot bounds has $b_2=8$. Thus the cap has second Betti number $14$ and also must have signature $-8$. Similarly for the knot in Item~\ref{item3}, the cap has second Betti number $20$ and signature $-16$; and for the knot in Item~\ref{item4} the cap has second Betti number $22$ and signature $-16$.  Thus they all satisfy the assumptions of Proposition~\ref{p:SVHM}. The statement for these knots follows immediately.

For knots in Item~\ref{item5}, the cap has second Betti number $22$ and signature $-16$, so it is a full-rank sublattice of the intersection lattice of a K3. It follows that the complement of the cap in the K3 is a rational homology ball. So Proposition~\ref{p:SVHM} says all exact fillings must be rational homology balls. 
\end{proof}

\subsection{Other cyclic covers}\label{occ}
So far we have only considered the case $r=2$, and made use of the fact that a (symplectic) K3 surface is a double cover of $\CP$ branched over a smooth sextic.
In fact, one can see that the K3 is also:
\begin{itemize}
\item a double cover of $\CPI \times \CPI$ branched over a smooth curve of bidegree $(4,4)$;
\item a triple cover of $\CPI \times \CPI$ branched over a smooth curve of bidegree $(3,3)$;
\item a quadruple cover of $\CP$ branched over a smooth quartic.
\end{itemize}
This can be seen by the ramification formula for branched covers~\cite[Lemma~I.17.1]{BHPV} and an Euler characteristic computation. The former shows that the canonical divisor of each of the previous branched covers vanishes, and the second that the Euler characteristics is $24$; the two conditions together identify K3 surfaces.

Moreover, as we noted in Proposition~\ref{p:hirzebruch}, we can find hats for $T_{3,5}$ and $T_{4,7}$ in the Hirzebruch cap $H_2\cong \CPI \times \CPI$. The homology computation of Proposition~\ref{p:hirzebruch} shows that the curves obtained by coning off the singularities have bidegrees $(3,3)$ and $(4,4)$, respectively. (Note that one needs to change basis in order for the computation to work: however, there is a symplectomorphism $\phi \colon H_0\to H_2$ such that $\phi_*(S_0 + F_0) = S_2$.)

Thus as we did in the previous subsections we can create a Calabi--Yau cap for $(\Sigma_r(K),\xi_{K,r})$ if we can find a cobordism from $K$ to:
\[
\begin{array}{lc}
T_{4,7} & \mbox{if } r=2;\\
T_{3,5} & \mbox{if } r=3;\\
T_{3,4}, T_{2,7}, T_{2,5}\#T_{2,3}, \mbox{ or } \#^3T_{2,3}, & \mbox{if } r=4.
\end{array}
\]
In the last line the first two get their degree-$4$ hats from \cite{55letters} (and are given by $V\!({zy^3  - x^4})$ and $V\!((zy - x^2)^2 - xy^3)$) and the third comes from Lemma~\ref{ratcurve}.

We do not explore all possibilities here, but rather restrict to a few examples. We note, however, that many of the computations carried out in the previous subsection can be used to give restrictions to $3$--fold and $4$--fold branched covers of some of the knots listed.

\begin{proof}[Proof of Theorem~\ref{highercovers}]
For each of the knots in the first class, which are all slice and quasipositive, we have found a symplectic cobordism to either $T_{2,7}$ (see the proof of Lemma~\ref{l:6-hats}, which gives cobordisms to $T_{2, 2k+1}$ for $k\leq 3$ and hence to $T_{2,7}$).
Since these are singularities of a degree-4 curve in $\CP$, we can find a degree-4 projective hat for each of them. Taking the $4$--fold branched cover of the hat yields a Calabi--Yau cap with second Betti number $22$ and $b_2^+ = 3$, thus allowing us to apply Proposition~\ref{p:SVHM}.

For all the knots in the first class, and the knots in the second class, we can also find cobordisms to $T_{3,5}$ (notice that one may easily use Lemma~\ref{addposcrossing} to construct a cobordism from $T_{2,7}$ to $T_{3,5}$ and the rest follow from the proof of Lemma~\ref{l:6-hats}), and thus obtain a hat in the Hirzebruch surface $\CPI\times\CPI$ of bidegree $(3,3)$. Taking the cyclic $3$--fold cover of the cap branched over the hat, yields another Calabi--Yau cap with second Betti number $22$ and $b_2^+=3$, and we can again apply Proposition~\ref{p:SVHM}.

Finally, can similarly argue for $8_{21}$; this is not a knot we have encountered before. It is the closure of the quasipositive $3$--braid $x^3yx^{-2}y^2$, and it has quasipositive genus 1. There is a genus-1 cobordism to $T_{2,3}\# T_{2,3}$ (obtained by adding two positive $x$ generators that cancel $x^{-2}$). It therefore admits a degree-$4$ projective hat and a bidegree-$(3,3)$ hat in $\CPI\times\CPI$.

One may see the degree-$4$ projective hat in several different ways; for instance, it is classically known that there is a rational curve of degree $4$ in $\CP$ whose singularities are three simple cusps (i.e. of type $T_{2,3}$); replacing one of the three singularities with a cusp yields the desired cap. Alternatively, one can deform a $T_{3,4}$--singularity to $T_{2,3}\#T_{2,3}$ by adding two generators (underlined) to the braid $x^3y^3$ to get to $(x\underline yx)xy(y\underline xy) = yxyxyxyx = (yx)^4$.

One may see the bidegree-$(3,3)$ hat by noting we can add two more generators to a braid word for $T_{3,4}$ to get $T_{3,5}$; as we have already observed, the latter knot has such a hat.

It is easy to check that the corresponding caps have second Betti numbers $16$ and $18$, respectively; moreover, we claim that these caps have $b_2^+ = 3$, thus allowing once again to apply Proposition~\ref{p:SVHM}.
To prove the claim, we notice that each of the two caps contains the complement of a filling of the $r$--fold cover of $S^3$ branched over $T_{2,3}\#T_{2,3}$; the cover is $\Sigma(2,3,r)\#\Sigma(2,3,r)$, endowed with the standard contact structure on each summand. These manifolds, however, possess only negative definite fillings (for instance, because they are Heegaard Floer L-spaces~\cite{OSz-genusbounds}, or because they are connected sums of links of simple singularities~\cite{OhtaOno}). In particular, the complement of the filling of $\Sigma(2,3,r)\#\Sigma(2,3,r)$ already has $b_2^+ = 3$, and {\em a fortiori} so does the cap of $\Sigma_r(8_{21})$.
\end{proof}

We are now ready to prove Theorem~\ref{last}. We recall that this theorem says:  
Let $(W,\omega_W)$ be a Stein filling of $(\Sigma(2,3,7),\xi_{\rm can})$. Then $W$ is spin, it has $H_1(W)=0$ and either $H_2(W) \cong E_8 \oplus 2H$ or $H_2(W) \cong \langle -1\rangle$; moreover, both cases occur.

In what follows, we denote with $\F$ the field with two elements; all Heegaard Floer homology groups will be taken with coefficients in $\F$.

\begin{proof}[Proof of Theorem~\ref{last}]
We begin by proving the last assertion; the Milnor fiber $M$ of the singularity $\{x^2 + y^3 + z^7=0\}$ is a Stein filling of $(\Sigma(2,3,7),\xi_{\rm can})$ that has $H_1(M) = 0$ (as it is homotopy equivalent to a wedge of spheres), it is spin, has $b_2(M) = 12$ and $\sigma(M) = -8$, therefore it realizes the first case. This can be seen, for instance, by viewing $M$ as the double cover of $B^4$ branched over a quasipositive surface for $T_{3,7}$; since the latter has genus $5$ and signature $-8$, the computations above follow.

The \emph{minimal resolution} of the singularity $\{x^2+y^3+z^7 = 0\}$, on the other hand, is a neighborhood of a rational curve (i.e. a sphere, possibly singular) with a singularity of type $T_{2,3}$ and self-intersection $-1$; this can be seen, for instance, from the normal crossing divisor resolution of the singularity, which is given by the following plumbing graph:
\begin{center}
\begin{tikzpicture}[scale=0.8]
	\draw (0.5,0) -- (3.5,0);
	\draw (2,-1.5) -- (2,0);
	\draw[fill=black] (0.5,0) circle(.08);
	\draw[fill=black] (2,0) circle(.08);
	\draw[fill=black] (3.5,0) circle(.08);
	\draw[fill=black] (2,-1.5) circle(.08);
	\node at (0.5,.3) {$-2$};
	\node at (3.5,.3) {$-7$};
	\node at (2,.3) {$-1$};
	\node at (1.5,-1.5) {$-3$};
\end{tikzpicture}
\end{center}
This is clearly not a minimal manifold, since the central vertex represents a $-1$--sphere; blowing it down, and then blowing down the contractions of the $-2$-- and $-3$--spheres yields the desired curve. This gives a minimal holomorphic filling of $(\Sigma(2,3,7),\xi_{\rm can})$; indeed, minimality follows from the adjunction formula, since the only primitive second homology class is represented by a symplectic curve of genus 1. See, for example,~\cite[Example 1.22]{Nemethi-fivelectures} for a reference. Now, work of Bogomolov and de Oliveira~\cite[Theorem~2']{BdO97} asserts that this holomorphic filling can be deformed to a Stein filling.

Let us now prove that these are the only two possibilities for the cohomology of fillings of $(\Sigma(2,3,7),\xi_{\rm can})$.

Let $K$ be the representative of $T_{3,7}$ with maximal self-linking number; as mentioned above, the contact 3--manifold $(\Sigma(2,3,7),\xi_{\rm can})$ is the double cover of $\Sst$ branched over $K$.

Since there is a deformation from $T_{3,11}$ to $T_{3,7}$, $K$ has a degree-6 projective hat $F$, which has genus $10-g(K) = 4$.
The double cover $(C,\omega_C)$ of the projective cap, branched over $F$, is a cap for $(\Sigma(2,3,7),\xi_{\rm can})$ that has $H_2(C)\cong E_8 \oplus H$. If $W$ is not negative definite, then gluing $C\cup W$, we obtain a symplectic Calabi--Yau $4$--manifold $X$: this essentially follows from Proposition~\ref{p:SVHM}, except that we need to use that $b_2^+(X) \ge 2$ instead of $b_2^+(C) \ge 2$.

Since $b_2(X) \ge 10$, we know that $X$ is not a $T^2$--bundle over $T^2$; since $b_2^+(X) \ge b_2^+(C) + 1 \ge 2$, $X$ cannot be an Enriques surface, either. Thus, as in the proof of  Proposition~\ref{p:SVHM}, $X$ is a K3 surface and we see that $H_2(W) \cong E_8 \oplus 2H$.

If $W$ is negative definite, we argue that its intersection form is diagonalizable: indeed, the Heegaard Floer correction term of $\Sigma(2,3,7)$ (in its unique spin$^c$ structure) vanishes~\cite[Section~8.1]{OzsvathSzabo-absolutely}; by~\cite[Section~9]{OzsvathSzabo-absolutely}, $W$ has diagonalizable intersection form \footnote{The proof of Theorem~9.1 only uses the fact that $d(S^3) = 0$. In fact, the statement that Ozsv\'ath and Szab\'o prove is the following: if $W$ is a negative definite $4$--manifold whose boundary is an integral homology sphere $Y$ with $d(Y) = 0$, then $W$ has diagonalizable intersection form.}.

Let $c = c(\xi_{\rm can})$, so that $\overline c$ is the Ozsv\'ath--Szab\'o contact invariant of $\overline{\xi}_{\rm can}$, where $\overline{c}$ is the image of $c$ under the isomorphism from ${\rm HF}^+(-\Sigma(2,3,7), \mathfrak{s}_\xi)$ to ${\rm HF}^+(-\Sigma(2,3,7), \overline{\mathfrak{s}}_\xi)$, \cite[Theorem~2.10]{Ghiggini06b}.

Recall from~\cite[Section~8.1]{OzsvathSzabo-absolutely} that, as graded vector spaces, ${\rm HF}^+(-\Sigma(2,3,7)) \cong \mathcal{T}^+_{(0)} \oplus \F_{(0)}$, where $\mathcal{T}^+ = \F[U,U^{-1}]/U\cdot \F[U]$ is a tower, and the subscript indicates that the degree of the bottom of the tower, the element that we call $1 \in {\rm HF}^+(-\Sigma(2,3,7))$, is in degree $0$, whereas the element $U^{k}$ lives in degree $2k$.
In fact, Ozsv\'ath and Szab\'o compute the group ${\rm HF}^+(\Sigma(2,3,7))$, from which ${\rm HF}^+(-\Sigma(2,3,7))$ can be recovered by duality~\cite[Proposition~2.5]{OzsvathSzabo-properties}.
Recall also from~\cite[Proof of Theorem~9.1 and Proposition~9.4]{OzsvathSzabo-absolutely} that if $Z$ is a cobordism from $Y$ to $Y'$, two integral homology spheres, and $\mathfrak s$ is any spin$^c$ structure on $Z$, then the map $F^\infty_{Z,\mathfrak s}:{\rm HF}^\infty(Y) \to {\rm HF}^\infty(Y')$ is an isomorphism if and only if $Z$ is negative definite.

With these generalities in mind, let us go back to the case at hand.
Since $\xi_{\rm can}$ has a filling $M$ with $b_2^+(M) > 0$ ($M$ the Milnor fiber mentioned above), $c$ is not conjugation-invariant in ${\rm HF}^+(-\Sigma(2,3,7))$, i.e. $c \neq \overline{c}$. Indeed, $F^+_{M,\mathfrak s_0} (c)= 1 \in {\rm HF}^+(-S^3)$ by functoriality of the contact invariant, but $F^+_{M,\mathfrak s_0}(1) = 0$, because $M$ is not negative definite.

Therefore, since ${\rm HF}^+_0(\Sigma(2,3,7)) \cong \F^{\oplus 2}$, $c$ is not conjugation-invariant, and $1$ is, we deduce that ${\rm HF}^+_0(\Sigma(2,3,7)) = \{0, 1, c, \overline{c}\}$.

Suppose now $b = b_2(W) > 1$. Since $W$ is a Stein filling of $\Sigma(2,3,7)$, which is an integral homology sphere, $H_1(W) = 0$, and therefore $H^2(W)$ is torsion-free. It follows that spin$^c$ structures on $W$ correspond to characteristic covectors in $H^2(W)$, via the first Chern class.
We are interested in spin$^c$ structures $\mathfrak s$ whose associated cobordism map $F^+_{W,\mathfrak s}$ has degree $0$, as these are the only spin$^c$ structures whose cobordism maps act non-trivially on $c$; since
\[
\deg F^+_{W,\mathfrak s} = \frac{c_1(\mathfrak{s})^2 - 2\chi(W) - 3\sigma(W))}4 = \frac{c_1(\mathfrak{s})^2 + b}4,
\]
asking that the degree be $0$ corresponds to asking that $c_1^2(\mathfrak{s}) = -b$.
There are exactly $2^b$ such spin$^c$ structures on $W$. In fact, their first Chern classes are in one-to-one correspondence with linear combinations of the form $\sum_{i=1}^b \pm e_i$, where $\{e_1,\dots,e_b\}$ is an orthonormal basis of $H^2(W)$. Since $b>1$, $2^b \ge 4$, hence there are at least four such spin$^c$ structures, as asserted.

We also claim that, for each such spin$^c$ structure $\mathfrak s$, either $F^+_{W,\mathfrak{s}}(c) \neq 0$ or $F^+_{W,\overline{\mathfrak{s}}}(c) \neq 0$. For, if both vanished, then $\ker F^+_{W,\mathfrak{s}}$ would contain $0, c, \overline{c}$ and hence be zero; however, we know that $F^+_{W,\mathfrak{s}}$ is a non-zero homomorphism (because $W$ is negative definite and $\deg F^+_{W,\mathfrak{s}} = 0$, we know that $F^+_{W,\mathfrak{s}}(1) = 1$).
In particular, if $b\ge 2$, there are at least two spin$^c$ structures such that $F^+_{W,\mathfrak{s}}(c) \neq 0$; however, this contradicts a result of Plamenevskaya~\cite[Proof of Theorem~4]{Plamenevskaya}, asserting that the canonical spin$^c$ structure is the only spin$^c$ structure $\mathfrak{s}$ on $W$ such that $F_{W,\mathfrak{s}}(c) \neq 0$.

So far, we have proved that $b \le 1$. We now argue that $b > 0$.
Indeed, if $b = 0$, then $W$ is a rational homology ball filling of $(\Sigma(2,3,7),\xi_{\rm can})$; since $\Sigma(2,3,7)$ is an integral homology sphere and $W$, which is a Stein domain, has a handle decomposition with no $3$--handles, we know that $H_1(W) = 0$. But then $W$ has even intersection form and $H_1(W) = 0$, therefore it is spin. This contradicts the fact that $\Sigma(2,3,7)$ has Rokhlin invariant $1$.

Summing up, if $W$ is negative definite, then we necessarily have $b_2(W) = 1$, and since the intersection form is unimodular, $H_2(W) \cong \langle -1\rangle$.
\end{proof}

Note that in the proof we are using the assumption that $W$ is a \emph{Stein} filling rather than just an exact one: indeed, we are using it in the second half of the proof, to exclude the case that $W$ is a rational homology ball, as well as when we are using functoriality of the Ozsv\'ath--Szab\'o contact invariant under \emph{Stein} cobordisms.
In fact, what we prove is that exact fillings are either negative definite or have intersection form $E_8 \oplus 2H$, and that Stein fillings that are negative definite have $b_2 = 1$.

We also observe that we can exhibit a Stein filling of $(\Sigma(2,3,7), \xi_{\rm can})$ as a handlebody. Let $\Lambda$ be a Legendrian trefoil with $\tb \Lambda = 0$. There are two such trefoils, with rotation numbers $\pm 1$, corresponding to two non-isotopic, conjugate contact structures on $\Sigma(2,3,7) = S^3_{-1}(T_{2,3})$. Since there are exactly two tight contact structures on $\Sigma(2,3,7)$~\cite{MarkTosun}, one of these two contact structure is the canonical one, and the corresponding handlebody is a Stein filling $W$ with $Q_W = \langle -1 \rangle$.

The same argument can be applied to show that all exact fillings of $(\Sigma(2,4,5),\xi_{\rm can})$ are either negative definite or have second Betti number $12$ and signature $-8$; the argument is slightly easier, since the first homology group here is $H_1(\Sigma(2,4,5)) \cong \Z/5\Z$, and therefore $\Sigma(2,4,5)$ cannot bound a rational homology ball. (By contrast, $\Sigma(2,3,7)$ \emph{does} bound a smooth, non-spin rational homology ball.)

\appendix

\section{Constructing the symplectic cobordisms via braids}\label{a:braids}

We begin by presenting the computation we omitted in the proof of Proposition~\ref{smalltorusknots}.

\begin{lemma}\label{l:T223toT67}
In the braid group $B_6$ with standard generators $\sigma_1,\dots,\sigma_5$, the following identity holds:
\[
(\sigma_1\cdots\sigma_5)^5 = \sigma_1\sigma_3\sigma_2\sigma_3\sigma_4\sigma_5\sigma_1\sigma_3\sigma_2\sigma_3\sigma_3\sigma_4\sigma_5\sigma_1\sigma_3\sigma_2\sigma_3\sigma_4\sigma_3\sigma_5\sigma_1\sigma_2\sigma_3\sigma_4\sigma_5.\]
\end{lemma}

\begin{proof}
We will only use the commutation relations $\sigma_i\sigma_j = \sigma_j\sigma_i$ whenever $|i-j| > 1$ and the braid relation $\sigma_i\sigma_{i+1}\sigma_i = \sigma_{i+1}\sigma_{i}\sigma_{i+1}$.

We start by cancelling the factors $\sigma_1$ and $\sigma_5(\sigma_1\cdots\sigma_5)$ which appear on the left and on the right, respectively, of each side of the equality. We are left to prove that:
\[
\sigma_2\sigma_3\sigma_4\sigma_5\sigma_1\sigma_2\sigma_3\sigma_4\sigma_5\sigma_1\sigma_2\sigma_3\sigma_4\sigma_5\sigma_1\sigma_2\sigma_3\sigma_4 =
\sigma_3\sigma_2\sigma_3\sigma_4\sigma_5\sigma_1\sigma_3\sigma_2\sigma_3\sigma_3\sigma_4\sigma_5\sigma_1\sigma_3\sigma_2\sigma_3\sigma_4\sigma_3.
\]
We will abide by the convention that we underline generators when something happens to them (e.g. we underline $\underleftarrow{\sigma_1}$ if we are using the commutation relation to move the generator $\sigma_1$ to the left).
We have:
\begin{align*}
&\sigma_2\sigma_3\sigma_4\sigma_5\sigma_1\sigma_2\sigma_3\sigma_4\sigma_5\underleftarrow{\sigma_1}\sigma_2\sigma_3\sigma_4\sigma_5\sigma_1\sigma_2\sigma_3\sigma_4 =\\ 
&\sigma_2\sigma_3\sigma_4\sigma_5\underline{\sigma_1\sigma_2\sigma_1}\sigma_3\sigma_4\sigma_5\sigma_2\sigma_3\sigma_4\sigma_5\sigma_1\sigma_2\sigma_3\sigma_4 =\\
&\sigma_2\sigma_3\sigma_4\sigma_5\underleftarrow{\sigma_2}\sigma_1\sigma_2\sigma_3\sigma_4\sigma_5\underleftarrow{\sigma_2}\sigma_3\sigma_4\sigma_5\sigma_1\sigma_2\sigma_3\sigma_4 =\\
&\underline{\sigma_2\sigma_3\sigma_2}\sigma_4\sigma_5\sigma_1\underline{\sigma_2\sigma_3\sigma_2}\sigma_4\sigma_5\underleftarrow{\sigma_3}\sigma_4\sigma_5\sigma_1\sigma_2\sigma_3\sigma_4 =\\
&\sigma_3\sigma_2\sigma_3\sigma_4\sigma_5\sigma_1\sigma_3\sigma_2\sigma_3\sigma_4\sigma_3\underline{\sigma_5\sigma_4\sigma_5}\sigma_1\sigma_2\sigma_3\sigma_4 =\\
&\sigma_3\sigma_2\sigma_3\sigma_4\sigma_5\sigma_1\sigma_3\sigma_2\sigma_3\underline{\sigma_4\sigma_3\sigma_4}\sigma_5\underrightarrow{\sigma_4}\sigma_1\sigma_2\sigma_3\sigma_4 =\\
&\sigma_3\sigma_2\sigma_3\sigma_4\sigma_5\sigma_1\sigma_3\sigma_2\sigma_3\sigma_3\sigma_4\sigma_3\sigma_5\sigma_1\sigma_2\underline{\sigma_4\sigma_3\sigma_4} =\\
&\sigma_3\sigma_2\sigma_3\sigma_4\sigma_5\sigma_1\sigma_3\sigma_2\sigma_3\sigma_3\sigma_4\sigma_5\sigma_1\sigma_3\sigma_2\sigma_3\sigma_4\sigma_3,
\end{align*}
as required.
\end{proof}

Here we provide the remaining computations to complete the proof of Lemma~\ref{l:6-hats}.
We use the notation above; for braids on $5$ strands we use the letter $w$ for the fourth generator. To de-clutter the notation, we also use capital letters to denote inverses.
Finally, we will use facts about (symplectic or complex) curves quite freely.

\begin{proof}[Proof of Lemma~\ref{l:6-hats} (continued)]
We argue case by case.

\noindent
$\mathbf{10_{124}}$: as noted above, this is $T_{3,5}$. This is the link of the degree-5 curve $\{x^3z^2-y^5=0\}$ at $(0:0:1)$, therefore it has a degree-5 (and hence a degree-6) hat.

\noindent
{$\mathbf{m(10_{155})}$:} this is the closure of the $3$--braid $x^3yX^2yX^2y$. We can write:
\[
x^3y\underline{X^2}y\underline{X^2}y \uparrow_4 x^3y\underline{x^{2}}y\underline{x^{2}}y = x^2(xyx)^2xy \sim \Delta^2yx^3 \uparrow\uparrow \Delta^4(yx) = (yx)^7;
\]
since there is a cobordism from $T_{3,7}$ to $T_{3,11}$, and $T_{3,11}$ is the singularity of a degree-6 curve, $m(10_{155})$ has a degree-6 hat.

\noindent
{$\mathbf{m(11n_{50})}$:} this is the closure of the $4$--braid $x^2yXyzYxY^2z$. As above:
\begin{align*}
x^2yXyz\underline{Y} x\underline{Y^2}z & \uparrow_2 x^2yXyz\underline{y}xz
= x^2yXy(zyz)x = x^2yXy^2\underline{z}yx\sim_D \\
&\sim_D x^2y\underline{X} y^3x \uparrow \underline{x^2}y\underline{x}\underline{y^2}x = x^5\underline{y}x^2 \sim_D x^7,
\end{align*}
and the latter is the singularity of a degree-4 curve.

\noindent
{$\mathbf{m(11n_{132})}$:} this is represented by the $4$--braid $X^2yxzYxYzy^2$. We now have, using the relation $zyxyz = zxyxz = xzyzx = xyzyz$:
\[
\underline{X^2}yxz\underline{Y} x\underline{Y} zy^2 \uparrow_3 yxz\underline{y} x\underline{y} zy^2 = yx^2y\underline{z}yx^2 \sim_D y\underline{x^2}\underline{y^2}x\underline{y^2} = (yx)^{4},
\]
and the latter is the singularity of a degree-4 curve.

\noindent
{$\mathbf{11n_{139}}$:}As noted above, this is the closure of the $5$--braid $x^2yXzYzwZyZw$. Using the relations $wzyzw = yzwzy$ (analogue as in the previous case) and $x^2yx = xyxy= yxy^2$:
\begin{align*}
x^2y\underline{x\inv} z\underline{y\inv} zw\underline{z\inv} y\underline{z\inv} w &\uparrow_3 (x^2y\underline{x}) z\underline{y} z(w\underline{z} y\underline{z} w) = y\underline{x}y^2zyzyz\underline{w}zy \sim_D\\
& \sim_D yy^2zyzyz^2y = y^2\Delta^2zy \sim \Delta^2(zy)^2 = (zy)^5,
\end{align*}
and the latter is the singularity of a degree-5 curve.

\noindent
{$\mathbf{m(11n_{172})}$:}As above, this is the closure of the $4$--braid $xyXyxzYxY^{2}z$. We can write:
\[
xy\underline{X} yxz\underline{Y} x\underline{Y^{2}}z \uparrow_3 xy\underline{x} yxz\underline{y} xz = xyxyx(zyz)x = xyxyxy\underline{z}yx \sim_D = \Delta^2yx \sim (xy)^4,
\]
which is the singularity of a degree-4 curve.

\noindent
{$\mathbf{m(12n_{121})}$:} This is the closure of the $4$--braid $xyX^2yzYxy^2z^2Y$. We have:
\begin{align*}
xy\underline{X^2}yz\underline{Y}xy^2z^2\underline{Y} &\uparrow_3 xy(yz\underline{y})xy^2z^2\underline{y} = x(yzy)zxy^2z^2y =\\
&= (xzyz^2x)y^2z^2y = z(xyx)z^2y^2z^2y = zy\underline{x}yz^2y^2z^2y \sim_D \\
&\sim_D zy^2z^2y^2z^2y \uparrow_2 zy^2z\underline{y^2}zy^2z\underline{y^2}zy = zy\Delta^4 = (zy)^7,
\end{align*}
and the latter has a cobordism to $(zy)^{11}$, hence it has a degree-6 hat.

\noindent
{$\mathbf{m(12n_{145})}$:} This is the closure of the $5$--braid $wZyZyX^2wyzYzyx$. We will use the identities $wzyzw = yzwzy$ and $\Delta^2 = yz^2yz^2$. We write:
\begin{align*}
w\underline{Z}y\underline{Z}y\underline{X^2}wyz\underline{Y}zyx &\uparrow_4 w\underline{z}y\underline{z}ywyz\underline{y}zy\underline{x} \sim_D (wzyzw)y^2zyzy = (yz\underline{w}zy)y^2zyzy \sim_D\\ 
& \sim_D yz^2y^3zyzy \uparrow (yz^2y\underline{z^2})y^2zyzy = \Delta^2y^2zyzy \uparrow\\
& \uparrow \Delta^2 y\underline{z^2}yzyzy = \Delta^4 yz = (yz)^7,
\end{align*}
which is the singularity of a degree-6 curve.

\noindent
{$\mathbf{12n_{292}}$:} this is the closure of the $4$--braid $xy^2x^3yZy^2xz^2$. We can write:
\begin{align*}
xy^2x^3yZy^2xz^2 &= xy^2x^3y(Zy^2z)xz = xy^2x^3y^2z^2\underline{Y} xz \uparrow xy^2x^3y^2z^2\underline{y} xz = \\
& = xy^2x^3y^2z(zyz)x = xy^2x^3y^2(zyz)yx = xy^2x^3y^3\underline{z}y^2x \sim_D\\
&\sim_D xy^2x^3y^5x \sim yx^2y^2x^3y^4 \uparrow yx\underline{y}\cdot\underline{y}xy\cdot yx^3y^4 = \Delta^2yx^3y^4\uparrow_4\\
&\uparrow_4 \Delta^2 yx\underline{y^2}x\underline{y^2}xy^2\underline{xy^2x}y^2 = \Delta^6yxy^2 = \Delta^6xyxy = (xy)^{11}.
\end{align*}

\noindent
{$\mathbf{m(12n_{393})}$:} this is the closure of the $5$--braid $yZwZyX^2zywz^2yx$. In the following, we will use the identity $wzyzw = yzwzy$:
\begin{align*}
y\underline{Z}w\underline{Z}y\underline{X^2}zywz^2yx &\uparrow_3 y\underline{z}(w\underline{z}yzw)yz^2y\underline{x} \sim_D yz(wzyzw)yz^2y = yz(yz\underline{w}zy)yz^2y \sim_D\\
& \sim_D yzyz^2y^2z^2y \uparrow_2 yzyz\underline{y^2}zy^2z\underline{y^2}zy \sim (yz)^7,
\end{align*}
and we conclude as in the cases above.

\noindent
{$\mathbf{12n_{473}}$:} this knot is the closure of the $4$--braid $xy^4z^2y^3xYz$. We write:
\begin{align*}
xy^4z^2y^3x\underline{Y} z &\uparrow xy^4z^2y^3x\underline{y}z \sim y^4z^2y^3(xyx)z = y^4z^2y^4\underline{x}yz \sim_D y^4z^2y^5z \sim zy^4z^2y^5 \uparrow_3 \\
&\uparrow_3 zy\underline{z^2}y\underline{z^2}y\underline{z^2}yz^2y^5 = \Delta^4yz^5 \uparrow\uparrow \Delta^6yz^3 \sim \Delta^6 zyz^2 = \Delta^6 yzyz = (yz)^{11}.
\end{align*}

\noindent
{$\mathbf{12n_{582}}$:} this knot is the closure of the $5$--braid $xYxyzwYwzyZWyZ$; using the identity $wzyzw = yzwzy$, we compute:
\begin{align*}
x\underline{Y}xyzw\underline{Y}wzy\underline{ZW}y \underline{Z} & \uparrow_5 (x\underline{y}x)yzw\underline{y}(wzy\underline{zw})y \underline{z} =  yxy^2zy^2(wzw)zy^2z \sim_D \\
& \sim_D y^3zy^2z\underline{w}z^2y^2z \sim_D y^3zy^2z^3y^2z \uparrow_3 zy^2\underline{z^2}yzy^2z^2\underline{y^2}zy\underline{z^2}y \sim\\
& \sim (yz)^5 zy^3z^2y^2 \uparrow \sim (yz)^5\underline{y^2}zy^3z^2y^2 = (yz)^{11}.
\end{align*}

\noindent
{$\mathbf{12n_{708}}$:} this is the closure of the $3$--braid $x Y^3  x Yxy Xy^3$.
\begin{align*}
x\underline{Y^3}  x \underline{Y}xy \underline{X}y^3 \uparrow x \underline{Y}  x \underline{y}xy \underline{x}y^3 = (yx)^4,
\end{align*}
whose closure is $T_{3,4}$.

\noindent
{$\mathbf{m(12n_{721})}$:} this is the closure of the $3$--braid $Y^5x^4y^2x$; using the identity $y^2xy^2x = (yx)^3$.
\begin{align*}
\underline{Y^5}x^4y^2x & \uparrow_6 \underline{y}x\underline{y^2}x\underline{y^2}x\underline{y^2}xy^2x = (yx)^7.
\end{align*}

\noindent
{$\mathbf{m(12n_{768})}$:} this is the closure of the $4$--braid $z^{-2}y^2zy^{-2}z^2yxy^{-1}x$.
\[
\underline{z^{-2}}y^2z\underline{y^{-2}}z^2yx\underline{y^{-1}}x \uparrow_3 y^2z^3\underline{xyx} =  y^2z^3y\underline{x}y \sim_D y^2z^3y^3,
\]
and the closure of $y^2z^3y^3$ is the connected sum $T_{2,5}\# T_{2,3}$, has a degree-4 hat (which is algebraic, since it comes from a rational cuspidal curve).

\noindent
{$\mathbf{12n_{838}}$:} this knot is the closure of the $5$--braid $xy Zw Xyzx Y Wzw$. Using the braid identities $xyzyx = zyxyz$, $wzyzw = yzwzy$, and $zy^2zy^2 = \Delta^2$:
\begin{align*}
xy\underline{Z}w\underline{X}yzx\underline{YW}zw & \uparrow_4 xy\underline{z}w\underline{x}yzx\underline{yw}zw = xyzw(xyx)zy(wzw) = (xyzyx)y(wzyzw)z =\\
&= zy\underline{x}yzy^2z\underline{w}zyz \sim_D zy^2zy^2z^2yz = \Delta^2 (zy)^2 = (zy)^5.\qedhere
\end{align*}
\end{proof}

\section{The generalized Thom conjecture}\label{a:majorThom}

Here we give an alternative proof of the generalized Thom conjecture, Theorem~\ref{t:majorThom}.
Recall that the theorem asserts that if $F$ is a  symplectic surface in a symplectic manifold $(X,\omega)$ with boundary $K$ in the contact manifold $Y=\partial X$, then $F$ is genus-minimizing in its homology class, relative to its boundary.

\begin{proof}[Proof of Theorem~\ref{t:majorThom}]
Fix a Seifert surface $S$ for $K$ in $Y$, and a Legendrian approximation $L$ of $K$; let $\rot_S(L)$ and $\slk_S(K)$ be the rotation number of $L$ and self-linking number of $K$ relative to $S$.

Attach a Weinstein handle to $(X,\omega)$ along $L$, thus obtaining a symplectic $4$--manifold $(X',\omega')$ with convex boundary. Let $\widehat F$ be the surface obtained by capping off $F$ with the core of the Weinstein handle.
We can embed $(X',\omega')$ in a minimal K\"ahler surface $(Z,\omega_Z)$ with $b_2^+(Z) > 1$ by \cite{LiscaMatic97}.

Since $b_2^+(Z) > 1$, the canonical class $K_Z$ of $Z$ is a Seiberg--Witten basic class \cite{Witten94}, and we can apply the adjunction inequality to any surface $G$ in the homology class $[\widehat F]$:
\begin{equation}\label{e:adjunctionThom}
2-2g(G) \le \langle c_1(Z), [G]\rangle - [G]\cdot [G] = \langle c_1(Z), [\widehat F]\rangle - [\widehat F]\cdot [\widehat F]
\end{equation}
We now set out to compute the right-hand side.

Call $F'$ the surface obtained by capping off $F$ with the Seifert surface $-S$, and $S'$ be the surface obtained by capping off $S$ with the core of the Weinstein handle.
Clearly we have that $[F'] + [S'] = [\widehat F]$.
Moreover, by~\cite[Proposition 2.3]{Gompf-handle},
\[
\langle c_1(Z), [S']\rangle = \langle c_1(X'), [S]\rangle = \rot_S(L).
\]
Since $F$ is symplectic, by Lemma~\ref{filladjunction} (and the following remark) we have:
\[
\langle c_1(Z), [F']\rangle = \langle c_1(X), [F']\rangle = \slk_S(K) + 1-2g(F) + [F']\cdot[F'].
\]
Thus
\[
 \langle c_1(Z), [\widehat F]\rangle = \langle c_1(Z), [F'] + [S']\rangle = \slk_S(K) + 1 - 2g(F) + [F']\cdot[F']+  \rot_S(L).
\]

Finally, the Weinstein handle is attached with contact framing$-1$ (hence smooth framing $\tb(L) - 1$); therefore, the last summand in~\eqref{e:adjunctionThom} is
\[
[\widehat F] \cdot [\widehat F] = [F']\cdot [F'] + [S'] \cdot[S'] = [F']\cdot [F'] + \tb(L) - 1.
\]

Putting the everything together, and recalling that $\tb(L) = \slk_S(K) + \rot_S(L)$, we obtain:
\begin{align*}
2-2g(G) &\le \slk_S(K) + 1 - 2g(F) + [F']\cdot[F'] + \rot_S(L) -[F']\cdot[F'] - \tb(L) + 1 \\
&= 2-2g(F).\qedhere
\end{align*}
\end{proof}

\end{document}